\documentclass[oneside,reqno]{amsart}
\usepackage[T1]{fontenc}		
\usepackage{ankita-packages}
\usepackage{ankita-commands}
\usepackage{ankita-thms}
\usepackage{ankita-spacing}
\usepackage{combgames}
\usepackage{scrextend}          
\usepackage{bm}
\usepackage[table]{xcolor}
\usetikzlibrary{positioning,decorations.pathreplacing}
\usepackage{mathrsfs}

\newcommand{\fs}{\mathcal{F}}
\newcommand{\ts}{\mathcal{T}}
\newcommand{\aw}{\mathrm{aw}}
\newcommand{\lp}{\mathscr{L}}
\newcommand{\rp}{\mathscr{R}}
\newcommand{\np}{\mathscr{N}}
\newcommand{\pp}{\mathscr{P}}
\newcommand{\hbeta}{\widehat\beta}

\usetikzlibrary {shapes.geometric}
\tikzset{star/.pic={
        \node [star, star point ratio=.08cm, draw, scale=0.4] at (0,0) { };
        }}
\newcommand{\cgfarstareq}[2]{{#1}\sim_{\cgfarstar}{#2}}
\makeatletter
\@namedef{subjclassname@2020}{\textup{2020} Mathematics Subject Classification}
\makeatother

\title{Partizan Subtraction with full and truncated support}

\author[Ankita Dargad]{Ankita Dargad}
\address{Department of Mathematics, Indian Institute of Technology Bombay, Powai, Mumbai 400076}
\email{ankitadargad.iitb@gmail.com}

\author[U. Larsson]{Urban Larsson}
\address{Department of Industrial Engineering and Operations Research, Indian Institute of Technology Bombay, Powai, Mumbai 400076}
\email{larsson@iitb.ac.in}

\subjclass[MSC 2020]{Primary 91A46; Secondary 91A05} 
\keywords{Combinatorial Game, Partizan Subtraction, Atomic Weight, Domination, Balance, Wealth Nim}

\date{\today}

\begin{document}
\begin{abstract}
    We investigate {\sc Full Support (FS)}, a {\sc Partizan Subtraction} game in which the players can remove any number of pebbles from the heap up to certain bounds that are typically different for Left and Right. The player with the richer move set always wins for all but finitely many heap sizes. We confirm this advantage by finding the general canonical form and the atomic weights of this game. To restore fairness (and peace), we introduce {\sc Truncated Support (TS)}, which essentially trims the larger subtraction set from below. If the truncation is shallow, the unfairness persists above a certain heap size, and one player continues ruling. If the truncation is deep, another player starts ruling. Interestingly, there is one more balanced truncation level in the middle, for which a non-trivial periodicity emerges, and where it has infinitely many $\mathcal P$ and $\mathcal N$-positions. We also explore the atomic weights for the lightly trimmed scenarios.
\end{abstract}
\maketitle

\section{Introduction}\label{sec: introduction FSTS}

Subtraction games form a classical family of combinatorial games. 
These are two-player games played on finite heaps of tokens. For each heap, each player is assigned a finite set of positive integers, called a {\em subtraction set}, which determines the number of pebbles the player can remove from the heap on a turn. The players move alternately, and a player who is unable to make a move loses the game. 

The two players are called Left (she) and Right (he), and given a heap, their subtraction sets are denoted by $S_L$ and $S_R$, respectively. Subtraction games have been studied extensively, with a large body of work devoted to impartial games, where $S_L=S_R$. However, comparatively less is known about {\sc Partizan Subtraction}, in which the move sets of the two players differ. In this paper, we study two sub-families of {\sc Partizan Subtraction}, which we call {\sc Full Support} and {\sc Truncated Support}. 

Consider the following narrative: 
\begingroup
\sffamily
\begin{quote}
During the annual polar-festivities, a penguin, Aquin (she), and a juvenile polar bear, Baloo (he), discover a fish buried beneath a heap of pebbles. They agree to remove the pebbles alternately, and the player who uncovers the fish first gets to eat it. Aquin, with her small flippers, can remove only a few pebbles at a time, whereas Baloo, with his larger paws, can remove both small and large portions of the heap. The heap is large enough, covering the whole fish, so whoever starts cannot fully remove the heap. 

Baloo offers Aquin the chance to start. It turns out that she still cannot win the fish. Baloo can exploit his additional power intelligently by removing a single pebble whenever the full heap cannot be removed. This strategy prevents Aquin from winning on the next move, as the remaining heap is still too large for her to remove it in one turn. 

Baloo's daddy, Polar Papa, then joins the game. However, his paws are so large that he is unable to remove very small portions of the heap. To their surprise, his immense power is irrelevant: Aquin discovered a way of play so that only a few pebbles remain at the end, so Polar Papa is unable to move, and Aquin can win the fish. 
 
Over the years, the game became popular among the polar bear and penguin families, and naturally evolved into a tournament-style team contest. In this contest, two mixed teams of polar bears and penguins spread out across several heaps, with each heap hosting one penguin and one polar bear from opposing teams. The teams then take turns, each time choosing a heap on which to play. The team that is unable to move loses the game and the other team gets all the fish. 
\end{quote}
\par
\endgroup
\begin{figure}[h!]
    \centering
    \includegraphics[width=0.5\linewidth]{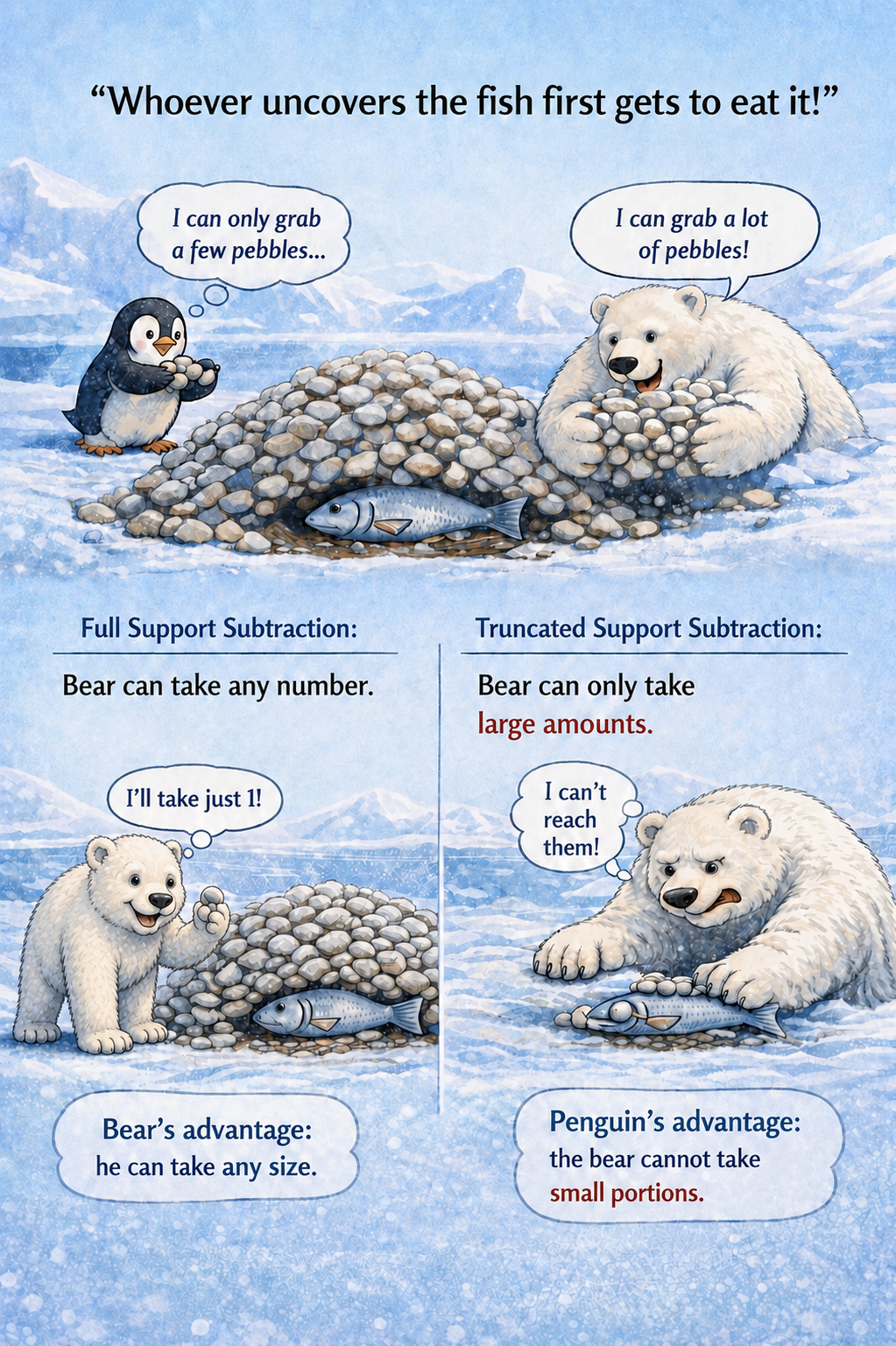}
    \caption{A metaphorical illustration of {\sc Full} and {\sc Truncated Support}.}
    \label{fig: penguin and bear}
\end{figure}

\begin{figure}[ht]
    \centering
    \includegraphics[width=0.9\linewidth]{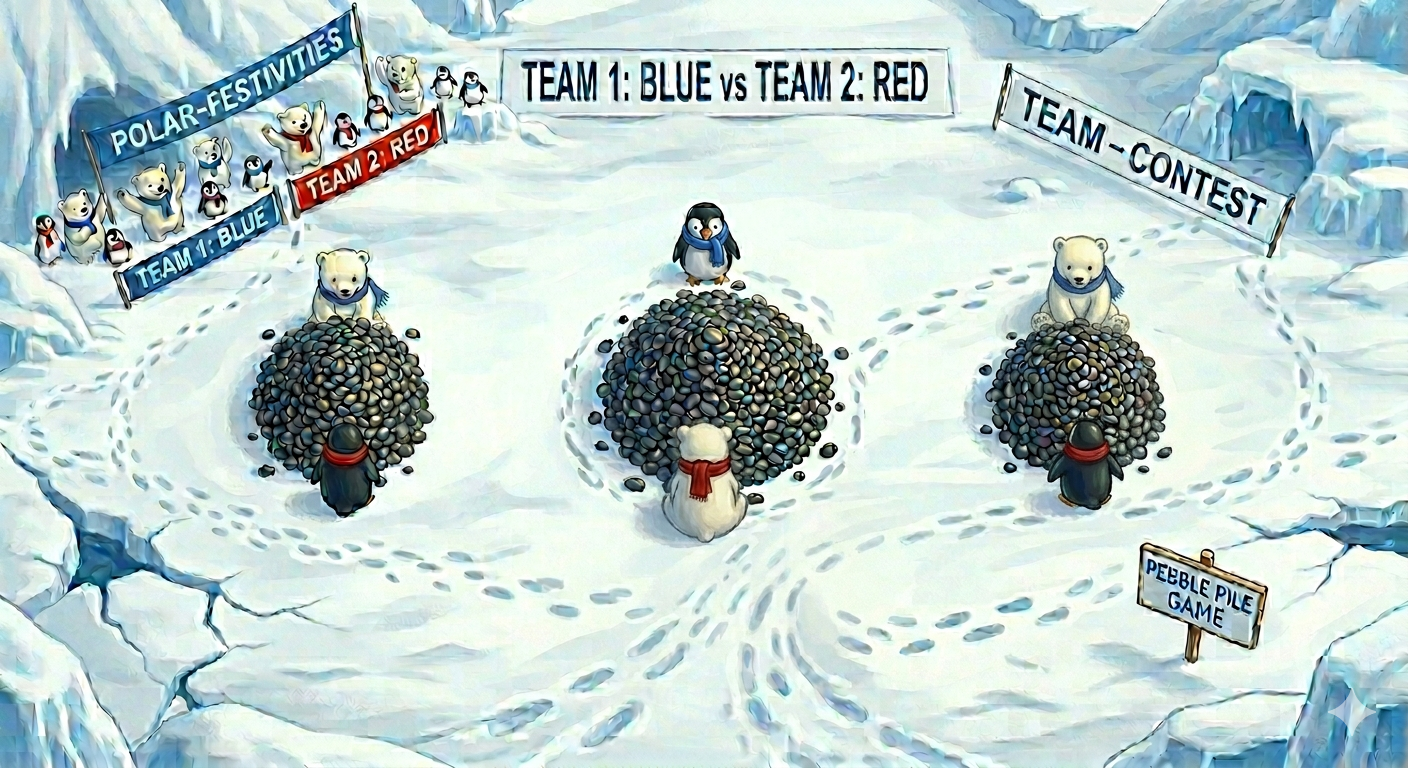}
    \caption{A team contest between Blue and Red team.}
    \label{fig: a team contest}
\end{figure}


The game played by Aquin and Baloo can be modeled using the tools of Combinatorial Game Theory (CGT)~\cite{S2013}; we review all the tools used in this paper in Appendix~\ref{sec: prelims}. Consider Aquin and Baloo as the standard CGT players, Left and Right, respectively. 
Let us denote the maximum number of pebbles Aquin and Baloo can remove from a heap by $a$ and $b$, respectively. Then the game played by Aquin and Baloo is equivalent to the partizan subtraction game with subtraction sets $S_L=[a]=\{1,\dots,a\}$ and $S_R=[b]$, $a,b\in \Nat=\{1,2,\ldots\}$. We call this {\sc Partizan Subtraction} ruleset family {\sc Full Support (FS)}.\footnote{We use the term ``Full Support'' to emphasize that players have the full support of all move sizes up to their maximum capacity, unlike sparse subtraction games.} A  {\sc FS} ruleset can also be viewed as an instance of {\sc Wealth Nim}~\cite{URYCG2020}, with $a$ and $b$ representing the {\em wealth} of the players. Consequently, this paper is the first to study a ruleset that lies at the intersection of {\sc Wealth Nim} and {\sc Partizan Subtraction}.

The multi-heap team contest is just a disjunctive sum of single-heap games. Thus, to understand the contest and determine the winner, we first need to analyze the single heap game.


        
        
        
        
        
        

Let us denote an {\sc FS} position with heap size $n\in \Nat_0= \Nat\cup\{0\} $, Left's subtraction set $[a]$, and Right's subtraction set $[b]$, by $\fs(n;a,b)$. Unless stated otherwise, we assume $0<a\le b$.  



As illustrated by the opening narrative, Baloo wins whenever the heap is large. Theoretically, this is the same as saying that Right wins {\sc FS}, independently of who starts, whenever $a<b$ and the heap is large. The theory also reveals that this advantage is deeper than mere winning. 

To observe this advantage, we use the standard CGT concept of canonical form~\cite{S2013}. Recall that a game $G$ in CGT is recursively defined as $G = \{G^{\mathcal{L}} \mid G^{\mathcal{R}}\}$, where $G^{\mathcal{L}}$ and $G^{\mathcal{R}}$ denote the sets of options available to Left and Right, respectively. Intuitively, the base of this recursion is the \emph{zero game}, denoted by $0 = \{\varnothing \mid \varnothing\}$, where $\varnothing$ denotes an empty set. And the canonical form of a game is its simplest equivalent representation, obtained by recursively removing irrelevant options and sometimes replacing options by appropriate simpler games. See Subsection~\ref{subsec: canonical form} for more details.

Right's advantage becomes transparent in the canonical form: for all large heap sizes, Right has a unique option in the canonical form, namely, to the $0$-game (see Table~\ref{tab: right options of FS}), while Left has to play to a non-zero game, if the heap size is larger than $a$. We establish this result in Theorem~\ref{thm: CF of FS}, and further show that no Left options are eliminated or simplified when reducing the game to its canonical form.

\begin{table}[ht!]
\centering
\renewcommand{\arraystretch}{1.2}
\caption{The Right options in the canonical forms of $\fs(n;3,5)$, $\fs(n;4,7)$ and $\fs(n;2,7)$, for small heap sizes.}
\label{tab: right options of FS}
\small
\begin{tabular}{|c|c|c|c|}
\hline

$n$ & $\fs(n;3,5)$ & $\fs(n;4,7)$ &  $\fs(n;2,7)$ \\ 
\hline
1  & $0$            & $0$               & $0$ \\ \hline
2  & $0,\cgstar$           & $0,\cgstar$              & $0,\cgstar$ \\ \hline
3  & $0,\cgstar,\cgstar2$         & $0,\cgstar,\cgstar2$            &  $0$\\ \hline
4  & $0$            & $0,\cgstar,\cgstar2,\cgstar3$        &  $0$\\ \hline
5  & $0$            & $0$               &  $0$\\ \hline
6  & $0$            & $0$               &  $0$\\ \hline
7  & $0$            & $0$               &  $0$\\ \hline
\end{tabular}
\end{table}

\begin{theorem}[Canonical Form of {\sc FS}]\label{thm: CF of FS}
    For fixed positive integers $a,b$, such that $a<b$, let $G:=\fs(n;a,b)$ be an instance of {\sc FS}. If $n\le a$, $G=*n$, and otherwise, the canonical form of $G$ is \begin{equation}
        \Set{\fs(n-1;a,b), \fs(n-2;a,b),\dots,\fs(n-a;a,b)\mid 0}.
    \label{eq:cf of FS}\end{equation} 
\end{theorem}

When $a=b$, {\sc FS} becomes impartial and 
    $\fs(n;a,a)=*(n \bmod (a+1))$ (see \cite{BCG2004}).

To play strategically in the disjunctive sum of games, players must evaluate all available heaps to determine which one requires an immediate move and on which one they can afford to {\em delay}.

Consider $\fs(n;a,b)$ with $a<b$. If $n\le a$, none of the players can delay making a move as their opponent has a direct move to $0$. If $n>a$, Right can delay making a move, because Left does not have a canonical move to $0$ and after Left's move, Right will still have a canonical move to $0$ by Theorem~\ref{thm: CF of FS}. In contrast, Left cannot afford to delay making a move, since Right has a canonical move to $0$, by Theorem~\ref{thm: CF of FS}. If the heap is very large, Left requires several moves to end the game, whereas Right can end the game at any intermediate stage. Therefore, Right's capacity to delay increases with increase in the heap size.

To understand how the capacity to delay determines the winner in a disjunctive sum of {\sc FS} games, consider the game $G+H$, where $G := \fs(10;3,4)$ and $H := \fs(4;3,2)$. In Table~\ref{fig: game comparison}, we see the minimum number of consecutive moves Left and Right require to end these games in their canonical forms.\\ 

\begin{table}[!htbp]\caption{The minimum number of consecutive moves required by each player to end the canonical forms of  the respective games.}
    \label{fig: game comparison}
    \centering
    \renewcommand{\arraystretch}{1.5}
    \begin{tabular}{c | c | c |}
        \multicolumn{1}{c}{} & 
        \multicolumn{1}{c}{\textbf{$G= \fs(10;3,4)$}} & 
        \multicolumn{1}{c}{\textbf{$H=\fs(4;3,2)$}} \\ \cline{2-3} 
        \textbf{Left} & needs four moves to end & can end any time\\ \cline{2-3} 
        \textbf{Right} & can end any time & needs two moves to end \\ \cline{2-3} 
    \end{tabular}
\end{table}

In $G+H$, Left delays playing on $H$, because her move there would leave the position $G+0$, and then, Right wins by moving to $0$ in $G$. Similarly, Right delays playing on $G$, because a move there would lead to the game $0+H$, and then, Left wins by moving to canonical $0$ in $H$. Furthermore, Left can delay playing on $H$ for one move, as Right requires two moves to end $H$, whereas Right can delay playing on $G$ for three moves, because Left requires four moves to end $G$. Since Right leads in delaying power, he will win the sum, namely, first he makes sure to end the game $H$ (where Left is strong) and then at last he  makes the winning move in $G$.

This notion of delaying power is formalized in combinatorial game theory through the concept of \emph{atomic weight} (see Subsection~\ref{subsec: infinitesimals}).
Intuitively, the atomic weight of a game measures how many moves a player can delay before the opponent can win the game. By convention, if Right (respectively, Left) can delay, then the atomic weight is negative (respectively, positive).

In the game $\fs(n;a,b)$, Left requires at least $\lceil n/a\rceil$ consecutive moves to empty the heap (end the game), whereas Right can move to $0$ in any non-empty sub-position in a single move. Therefore, in a disjunctive sum, Right may allow Left to play for $\lceil n/a\rceil-1$ turns on $\fs(n;a,b)$ and still execute the last move on this component. In our second main result, we prove that the atomic weight of {\sc FS} is exactly $-(\lceil n/a\rceil-1)$.

\begin{theorem}[Atomic Weight of FS]\label{thm: aw of FS}
    Consider $G=\fs(n;a,b)$, where $a\le b$. Then 
    \begin{enumerate}
        \item If $a=b$, then $\aw(G)=0$ for all $n$;
        \item If $a<b$, then $\aw(G)=-\lceil n/a\rceil +1$ for all $n$.
    \end{enumerate}
\end{theorem}

By this result, the Table~\ref{fig: game comparison} games satisfy \[\aw(G)=\aw(\fs(10;3,4))=-3, \;\text{ and  }\; \aw(H)=\aw(\fs(4;3,2))=1.\] 
By Lemma~\ref{lem: aw of sum of games}, Atomic Weight is additive and we get $\aw(G+H)=\aw(G)+\aw(H)=-2$. By the popular two-ahead rule (Lemma~\ref{thm: outcome using aw}), if atomic weight is less or equal to $-2$, Right wins. Hence, $G+H\in \rp$.


Theorem~\ref{thm: aw of FS} shows that, in {\sc Full Support}, in the narrative, Baloo has a substantial advantage. This is because he can remove large as well as small amount. However, if Polar Papa replaces Baloo, this advantage might not persists because he is unable to remove small amounts from the heap. This motivates the study of a modified ruleset, which we call {\sc Truncated Support (TS)}, in which the larger subtraction set (the subtraction set of Baloo) from {\sc Full Support} is truncated from below.


As before, let $a$ and $b$ denote the maximum numbers of pebbles Left and Right can remove in a single move, respectively, and let $a<b$. 
For an integer $\tau\in [1,b)$, called the \emph{truncation level}, the corresponding {\sc TS} position with heap size $n$ is denoted by $\ts_\tau(n;a,b)$, where the subtraction sets are $S_L=[a]$ and $S_R=\{\tau+1,\dots,b\}$.


As the truncation level $\tau$ increases, Right loses access to an increasing number of small moves; intuitively his advantage should decrease. This intuition is supported by Figure~\ref{fig: outcomes of ts(n,3,7)}, which illustrates the {\em outcomes} of $\ts_\tau(n;3,7)$ for different heap sizes at each truncation level $1\le \tau\le 6$. 
The outcome of a game $G$, denoted by $o(G)$, determines the winner of a game given the starting player. Every game belongs to one of the four outcome classes: $\lp$ (Left wins), $\mathscr{R}$ (Right wins), $\np$ (Next player wins), or $\pp$ (Previous player wins). 
In short, we write $G\in \lp$, if $o(G)=\lp$, and similarly, we write $G\in \np\cup\rp$ if $o(G)$ is either $\np$ or $\rp$. A game is $\lp$-position, if its outcome is $\lp$, and similar for other outcomes.

Return to Figure~\ref{fig: outcomes of ts(n,3,7)}. For small truncation levels, here $\tau\le3$, $\ts_\tau(n;3,7)\in\rp$ for all sufficiently large heap sizes $n$. Thus, removing only a few of Right's smallest moves is not enough to eliminate Right's long-term advantage. 
At $\tau=4$ however, the behavior changes; both $\pp$- and $\np$-positions occur infinitely often, indicating a more balanced outcome structure. For even larger truncation levels, Right loses any advantage, and the outcome becomes $\lp$ for all sufficiently large heap sizes.

\begin{figure}[h!]
\centering
\begin{tikzpicture}[x=0.45cm, y=0.5cm,scale=0.83]
    \tikzset{fade/.style={fill opacity=1}}
  \colorlet{cL}{blue!75}
  \colorlet{cR}{red!75}
  \colorlet{cN}{green!90!black} 
  \colorlet{cP}{black}
  \fill[cL, fade] (0.5, 0.5) rectangle (1.5, 1.5);
  \fill[cN, fade] (1.5, 0.5) rectangle (2.5, 1.5);
  \fill[cN, fade] (2.5, 0.5) rectangle (3.5, 1.5);
  \fill[cN, fade] (3.5, 0.5) rectangle (4.5, 1.5);
  \fill[cR, fade] (4.5, 0.5) rectangle (5.5, 1.5);
  \fill[cR, fade] (5.5, 0.5) rectangle (6.5, 1.5);
  \fill[cR, fade] (6.5, 0.5) rectangle (7.5, 1.5);
  \fill[cR, fade] (7.5, 0.5) rectangle (8.5, 1.5);
  \fill[cR, fade] (8.5, 0.5) rectangle (9.5, 1.5);
  \fill[cR, fade] (9.5, 0.5) rectangle (10.5, 1.5);
  \fill[cR, fade] (10.5, 0.5) rectangle (11.5, 1.5);
  \fill[cR, fade] (11.5, 0.5) rectangle (12.5, 1.5);
  \fill[cR, fade] (12.5, 0.5) rectangle (13.5, 1.5);
  \fill[cR, fade] (13.5, 0.5) rectangle (14.5, 1.5);
  \fill[cR, fade] (14.5, 0.5) rectangle (15.5, 1.5);
  \fill[cR, fade] (15.5, 0.5) rectangle (16.5, 1.5);
  \fill[cR, fade] (16.5, 0.5) rectangle (17.5, 1.5);
  \fill[cR, fade] (17.5, 0.5) rectangle (18.5, 1.5);
  \fill[cR, fade] (18.5, 0.5) rectangle (19.5, 1.5);
  \fill[cR, fade] (19.5, 0.5) rectangle (20.5, 1.5);
  \fill[cR, fade] (20.5, 0.5) rectangle (21.5, 1.5);
  \fill[cR, fade] (21.5, 0.5) rectangle (22.5, 1.5);
  \fill[cR, fade] (22.5, 0.5) rectangle (23.5, 1.5);
  \fill[cR, fade] (23.5, 0.5) rectangle (24.5, 1.5);
  \fill[cR, fade] (24.5, 0.5) rectangle (25.5, 1.5);
  \fill[cR, fade] (25.5, 0.5) rectangle (26.5, 1.5);
  \fill[cR, fade] (26.5, 0.5) rectangle (27.5, 1.5);
  \fill[cR, fade] (27.5, 0.5) rectangle (28.5, 1.5);
  \fill[cR, fade] (28.5, 0.5) rectangle (29.5, 1.5);
  \fill[cR, fade] (29.5, 0.5) rectangle (30.5, 1.5);
  \fill[cR, fade] (30.5, 0.5) rectangle (31.5, 1.5);
  \fill[cR, fade] (31.5, 0.5) rectangle (32.5, 1.5);
  \fill[cR, fade] (32.5, 0.5) rectangle (33.5, 1.5);
  \fill[cR, fade] (33.5, 0.5) rectangle (34.5, 1.5);
  \fill[cL, fade] (0.5, 1.5) rectangle (1.5, 2.5);
  \fill[cL, fade] (1.5, 1.5) rectangle (2.5, 2.5);
  \fill[cN, fade] (2.5, 1.5) rectangle (3.5, 2.5);
  \fill[cN, fade] (3.5, 1.5) rectangle (4.5, 2.5);
  \fill[cN, fade] (4.5, 1.5) rectangle (5.5, 2.5);
  \fill[cR, fade] (5.5, 1.5) rectangle (6.5, 2.5);
  \fill[cR, fade] (6.5, 1.5) rectangle (7.5, 2.5);
  \fill[cP, fade] (7.5, 1.5) rectangle (8.5, 2.5);
  \fill[cN, fade] (8.5, 1.5) rectangle (9.5, 2.5);
  \fill[cN, fade] (9.5, 1.5) rectangle (10.5, 2.5);
  \fill[cN, fade] (10.5, 1.5) rectangle (11.5, 2.5);
  \fill[cR, fade] (11.5, 1.5) rectangle (12.5, 2.5);
  \fill[cR, fade] (12.5, 1.5) rectangle (13.5, 2.5);
  \fill[cR, fade] (13.5, 1.5) rectangle (14.5, 2.5);
  \fill[cR, fade] (14.5, 1.5) rectangle (15.5, 2.5);
  \fill[cR, fade] (15.5, 1.5) rectangle (16.5, 2.5);
  \fill[cR, fade] (16.5, 1.5) rectangle (17.5, 2.5);
  \fill[cR, fade] (17.5, 1.5) rectangle (18.5, 2.5);
  \fill[cR, fade] (18.5, 1.5) rectangle (19.5, 2.5);
  \fill[cR, fade] (19.5, 1.5) rectangle (20.5, 2.5);
  \fill[cR, fade] (20.5, 1.5) rectangle (21.5, 2.5);
  \fill[cR, fade] (21.5, 1.5) rectangle (22.5, 2.5);
  \fill[cR, fade] (22.5, 1.5) rectangle (23.5, 2.5);
  \fill[cR, fade] (23.5, 1.5) rectangle (24.5, 2.5);
  \fill[cR, fade] (24.5, 1.5) rectangle (25.5, 2.5);
  \fill[cR, fade] (25.5, 1.5) rectangle (26.5, 2.5);
  \fill[cR, fade] (26.5, 1.5) rectangle (27.5, 2.5);
  \fill[cR, fade] (27.5, 1.5) rectangle (28.5, 2.5);
  \fill[cR, fade] (28.5, 1.5) rectangle (29.5, 2.5);
  \fill[cR, fade] (29.5, 1.5) rectangle (30.5, 2.5);
  \fill[cR, fade] (30.5, 1.5) rectangle (31.5, 2.5);
  \fill[cR, fade] (31.5, 1.5) rectangle (32.5, 2.5);
  \fill[cR, fade] (32.5, 1.5) rectangle (33.5, 2.5);
  \fill[cR, fade] (33.5, 1.5) rectangle (34.5, 2.5);
  \fill[cL, fade] (0.5, 2.5) rectangle (1.5, 3.5);
  \fill[cL, fade] (1.5, 2.5) rectangle (2.5, 3.5);
  \fill[cL, fade] (2.5, 2.5) rectangle (3.5, 3.5);
  \fill[cN, fade] (3.5, 2.5) rectangle (4.5, 3.5);
  \fill[cN, fade] (4.5, 2.5) rectangle (5.5, 3.5);
  \fill[cN, fade] (5.5, 2.5) rectangle (6.5, 3.5);
  \fill[cR, fade] (6.5, 2.5) rectangle (7.5, 3.5);
  \fill[cP, fade] (7.5, 2.5) rectangle (8.5, 3.5);
  \fill[cL, fade] (8.5, 2.5) rectangle (9.5, 3.5);
  \fill[cL, fade] (9.5, 2.5) rectangle (10.5, 3.5);
  \fill[cN, fade] (10.5, 2.5) rectangle (11.5, 3.5);
  \fill[cN, fade] (11.5, 2.5) rectangle (12.5, 3.5);
  \fill[cN, fade] (12.5, 2.5) rectangle (13.5, 3.5);
  \fill[cR, fade] (13.5, 2.5) rectangle (14.5, 3.5);
  \fill[cR, fade] (14.5, 2.5) rectangle (15.5, 3.5);
  \fill[cP, fade] (15.5, 2.5) rectangle (16.5, 3.5);
  \fill[cL, fade] (16.5, 2.5) rectangle (17.5, 3.5);
  \fill[cN, fade] (17.5, 2.5) rectangle (18.5, 3.5);
  \fill[cN, fade] (18.5, 2.5) rectangle (19.5, 3.5);
  \fill[cN, fade] (19.5, 2.5) rectangle (20.5, 3.5);
  \fill[cR, fade] (20.5, 2.5) rectangle (21.5, 3.5);
  \fill[cR, fade] (21.5, 2.5) rectangle (22.5, 3.5);
  \fill[cR, fade] (22.5, 2.5) rectangle (23.5, 3.5);
  \fill[cP, fade] (23.5, 2.5) rectangle (24.5, 3.5);
  \fill[cN, fade] (24.5, 2.5) rectangle (25.5, 3.5);
  \fill[cN, fade] (25.5, 2.5) rectangle (26.5, 3.5);
  \fill[cN, fade] (26.5, 2.5) rectangle (27.5, 3.5);
  \fill[cR, fade] (27.5, 2.5) rectangle (28.5, 3.5);
  \fill[cR, fade] (28.5, 2.5) rectangle (29.5, 3.5);
  \fill[cR, fade] (29.5, 2.5) rectangle (30.5, 3.5);
  \fill[cR, fade] (30.5, 2.5) rectangle (31.5, 3.5);
  \fill[cR, fade] (31.5, 2.5) rectangle (32.5, 3.5);
  \fill[cR, fade] (32.5, 2.5) rectangle (33.5, 3.5);
  \fill[cR, fade] (33.5, 2.5) rectangle (34.5, 3.5);
  \fill[cL] (0.5, 3.5) rectangle (1.5, 4.5);
  \fill[cL] (1.5, 3.5) rectangle (2.5, 4.5);
  \fill[cL] (2.5, 3.5) rectangle (3.5, 4.5);
  \fill[cL] (3.5, 3.5) rectangle (4.5, 4.5);
  \fill[cN] (4.5, 3.5) rectangle (5.5, 4.5);
  \fill[cN] (5.5, 3.5) rectangle (6.5, 4.5);
  \fill[cN] (6.5, 3.5) rectangle (7.5, 4.5);
  \fill[cP] (7.5, 3.5) rectangle (8.5, 4.5);
  \fill[cL] (8.5, 3.5) rectangle (9.5, 4.5);
  \fill[cL] (9.5, 3.5) rectangle (10.5, 4.5);
  \fill[cL] (10.5, 3.5) rectangle (11.5, 4.5);
  \fill[cL] (11.5, 3.5) rectangle (12.5, 4.5);
  \fill[cN] (12.5, 3.5) rectangle (13.5, 4.5);
  \fill[cN] (13.5, 3.5) rectangle (14.5, 4.5);
  \fill[cN] (14.5, 3.5) rectangle (15.5, 4.5);
  \fill[cP] (15.5, 3.5) rectangle (16.5, 4.5);
  \fill[cL] (16.5, 3.5) rectangle (17.5, 4.5);
  \fill[cL] (17.5, 3.5) rectangle (18.5, 4.5);
  \fill[cL] (18.5, 3.5) rectangle (19.5, 4.5);
  \fill[cL] (19.5, 3.5) rectangle (20.5, 4.5);
  \fill[cN] (20.5, 3.5) rectangle (21.5, 4.5);
  \fill[cN] (21.5, 3.5) rectangle (22.5, 4.5);
  \fill[cN] (22.5, 3.5) rectangle (23.5, 4.5);
  \fill[cP] (23.5, 3.5) rectangle (24.5, 4.5);
  \fill[cL] (24.5, 3.5) rectangle (25.5, 4.5);
  \fill[cL] (25.5, 3.5) rectangle (26.5, 4.5);
  \fill[cL] (26.5, 3.5) rectangle (27.5, 4.5);
  \fill[cL] (27.5, 3.5) rectangle (28.5, 4.5);
  \fill[cN] (28.5, 3.5) rectangle (29.5, 4.5);
  \fill[cN] (29.5, 3.5) rectangle (30.5, 4.5);
  \fill[cN] (30.5, 3.5) rectangle (31.5, 4.5);
  \fill[cP] (31.5, 3.5) rectangle (32.5, 4.5);
  \fill[cL] (32.5, 3.5) rectangle (33.5, 4.5);
  \fill[cL] (33.5, 3.5) rectangle (34.5, 4.5);
  \fill[cL, fade] (0.5, 4.5) rectangle (1.5, 5.5);
  \fill[cL, fade] (1.5, 4.5) rectangle (2.5, 5.5);
  \fill[cL, fade] (2.5, 4.5) rectangle (3.5, 5.5);
  \fill[cL, fade] (3.5, 4.5) rectangle (4.5, 5.5);
  \fill[cL, fade] (4.5, 4.5) rectangle (5.5, 5.5);
  \fill[cN, fade] (5.5, 4.5) rectangle (6.5, 5.5);
  \fill[cN, fade] (6.5, 4.5) rectangle (7.5, 5.5);
  \fill[cL, fade] (7.5, 4.5) rectangle (8.5, 5.5);
  \fill[cL, fade] (8.5, 4.5) rectangle (9.5, 5.5);
  \fill[cL, fade] (9.5, 4.5) rectangle (10.5, 5.5);
  \fill[cL, fade] (10.5, 4.5) rectangle (11.5, 5.5);
  \fill[cL, fade] (11.5, 4.5) rectangle (12.5, 5.5);
  \fill[cL, fade] (12.5, 4.5) rectangle (13.5, 5.5);
  \fill[cL, fade] (13.5, 4.5) rectangle (14.5, 5.5);
  \fill[cL, fade] (14.5, 4.5) rectangle (15.5, 5.5);
  \fill[cL, fade] (15.5, 4.5) rectangle (16.5, 5.5);
  \fill[cL, fade] (16.5, 4.5) rectangle (17.5, 5.5);
  \fill[cL, fade] (17.5, 4.5) rectangle (18.5, 5.5);
  \fill[cL, fade] (18.5, 4.5) rectangle (19.5, 5.5);
  \fill[cL, fade] (19.5, 4.5) rectangle (20.5, 5.5);
  \fill[cL, fade] (20.5, 4.5) rectangle (21.5, 5.5);
  \fill[cL, fade] (21.5, 4.5) rectangle (22.5, 5.5);
  \fill[cL, fade] (22.5, 4.5) rectangle (23.5, 5.5);
  \fill[cL, fade] (23.5, 4.5) rectangle (24.5, 5.5);
  \fill[cL, fade] (24.5, 4.5) rectangle (25.5, 5.5);
  \fill[cL, fade] (25.5, 4.5) rectangle (26.5, 5.5);
  \fill[cL, fade] (26.5, 4.5) rectangle (27.5, 5.5);
  \fill[cL, fade] (27.5, 4.5) rectangle (28.5, 5.5);
  \fill[cL, fade] (28.5, 4.5) rectangle (29.5, 5.5);
  \fill[cL, fade] (29.5, 4.5) rectangle (30.5, 5.5);
  \fill[cL, fade] (30.5, 4.5) rectangle (31.5, 5.5);
  \fill[cL, fade] (31.5, 4.5) rectangle (32.5, 5.5);
  \fill[cL, fade] (32.5, 4.5) rectangle (33.5, 5.5);
  \fill[cL, fade] (33.5, 4.5) rectangle (34.5, 5.5);
  \fill[cL, fade] (0.5, 5.5) rectangle (1.5, 6.5);
  \fill[cL, fade] (1.5, 5.5) rectangle (2.5, 6.5);
  \fill[cL, fade] (2.5, 5.5) rectangle (3.5, 6.5);
  \fill[cL, fade] (3.5, 5.5) rectangle (4.5, 6.5);
  \fill[cL, fade] (4.5, 5.5) rectangle (5.5, 6.5);
  \fill[cL, fade] (5.5, 5.5) rectangle (6.5, 6.5);
  \fill[cN, fade] (6.5, 5.5) rectangle (7.5, 6.5);
  \fill[cL, fade] (7.5, 5.5) rectangle (8.5, 6.5);
  \fill[cL, fade] (8.5, 5.5) rectangle (9.5, 6.5);
  \fill[cL, fade] (9.5, 5.5) rectangle (10.5, 6.5);
  \fill[cL, fade] (10.5, 5.5) rectangle (11.5, 6.5);
  \fill[cL, fade] (11.5, 5.5) rectangle (12.5, 6.5);
  \fill[cL, fade] (12.5, 5.5) rectangle (13.5, 6.5);
  \fill[cL, fade] (13.5, 5.5) rectangle (14.5, 6.5);
  \fill[cL, fade] (14.5, 5.5) rectangle (15.5, 6.5);
  \fill[cL, fade] (15.5, 5.5) rectangle (16.5, 6.5);
  \fill[cL, fade] (16.5, 5.5) rectangle (17.5, 6.5);
  \fill[cL, fade] (17.5, 5.5) rectangle (18.5, 6.5);
  \fill[cL, fade] (18.5, 5.5) rectangle (19.5, 6.5);
  \fill[cL, fade] (19.5, 5.5) rectangle (20.5, 6.5);
  \fill[cL, fade] (20.5, 5.5) rectangle (21.5, 6.5);
  \fill[cL, fade] (21.5, 5.5) rectangle (22.5, 6.5);
  \fill[cL, fade] (22.5, 5.5) rectangle (23.5, 6.5);
  \fill[cL, fade] (23.5, 5.5) rectangle (24.5, 6.5);
  \fill[cL, fade] (24.5, 5.5) rectangle (25.5, 6.5);
  \fill[cL, fade] (25.5, 5.5) rectangle (26.5, 6.5);
  \fill[cL, fade] (26.5, 5.5) rectangle (27.5, 6.5);
  \fill[cL, fade] (27.5, 5.5) rectangle (28.5, 6.5);
  \fill[cL, fade] (28.5, 5.5) rectangle (29.5, 6.5);
  \fill[cL, fade] (29.5, 5.5) rectangle (30.5, 6.5);
  \fill[cL, fade] (30.5, 5.5) rectangle (31.5, 6.5);
  \fill[cL, fade] (31.5, 5.5) rectangle (32.5, 6.5);
  \fill[cL, fade] (32.5, 5.5) rectangle (33.5, 6.5);
  \fill[cL, fade] (33.5, 5.5) rectangle (34.5, 6.5);
  
  \draw[yellow!90,dashed, line width=1pt]
    (0.5,4) -- (34.5,4);
  \foreach \n in {1, 2, ..., 34} \node[below, font=\tiny] at (\n, 0.25) {\n};
  \foreach \k in {1,2,3,4,5,6}{
    \ifnum\k=4
        \node[left,font=\scriptsize] at (0.5,\k) {\textbf{\k}};
    \else
        \node[left,font=\scriptsize] at (0.5,\k) {\k};
    \fi
}
  \node[below=0.5cm] at (17.5, 0.5) {Heap Size ($n$)};
  \node[rotate=90, above=1cm] at (0, 1.1) {Truncation Level ($\tau$)};




  \begin{scope}[shift={(35.5, 3.5)}]
    \node[right] at (-0.8, 2.3) {Outcomes};
    \draw[fill=cL, fade] (0, 0.8) rectangle (0.8, 1.6); \node[right] at (1, 1.2) {$\mathscr{L}$};
    \draw[fill=cR, fade] (0, -0.4) rectangle (0.8, 0.4); \node[right] at (1, 0) {$\mathscr{R}$};
    \draw[fill=cN, fade] (0, -1.6) rectangle (0.8, -0.8); \node[right] at (1, -1.2) {$\mathscr{N}$};
    \draw[fill=cP, fade] (0, -2.8) rectangle (0.8, -2.0); \node[right] at (1, -2.4) {$\mathscr{P}$};
  \end{scope}
\end{tikzpicture}
\caption{Outcomes of $\ts_\tau(n; 3, 7)$. Blue, red, green, and black cells represent $\mathscr{L}$, $\mathscr{R}$, $\mathscr{N}$, and $\mathscr{P}$ outcomes, respectively. The row corresponding to the truncation level at which $\pp$- and $\np$-positions occur infinitely often is highlighted.}
\label{fig: outcomes of ts(n,3,7)}
\end{figure}
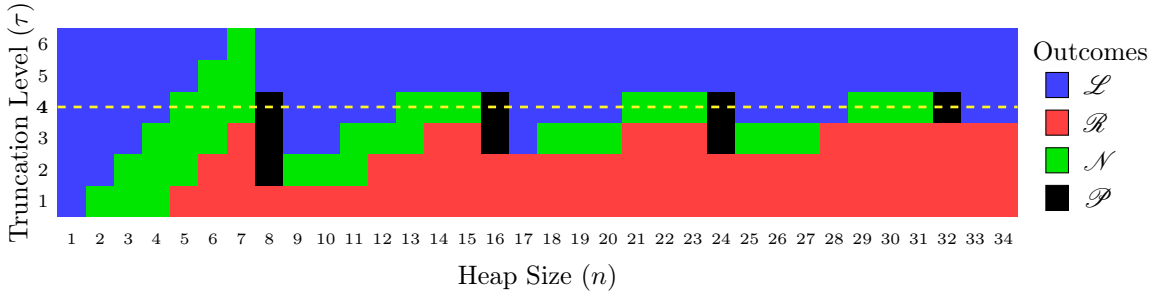

For fixed $a,b$ and $\tau$, we say that the ruleset
$\ts_\tau(n;a,b)$ is \emph{Right-dominating} if $\ts_\tau(n;a,b)\in \rp$, for all sufficiently large $n$. Similarly, if $\ts_\tau(n;a,b)\in \lp$ for all sufficiently large $n$, we say that the ruleset is \emph{Left-dominating}. If the ruleset is neither Left- nor Right-dominating, we call it \emph{balanced}.

This motivates the following questions: does the ruleset progress systematically from being Right-dominating at small truncation levels, to balanced in the middle, and finally Left-dominating at large truncation levels? If this is the case, what is the minimum truncation level at which the ruleset first becomes balanced? We refer to this specific truncation level as the \emph{critical threshold}. This naturally leads to another interesting question: is the ruleset balanced exclusively at the critical threshold, or does this balanced behavior occur across multiple truncation levels?
In line with these questions, Figure~\ref{fig: outcomes of ts(n,3,7)} and many other experiments had suggested to us that, for fixed $a$ and $b$, {\sc Truncated Support} is Right-dominating for truncation levels below $b-a$, balanced at truncation level $b-a$, and Left-dominating for truncation levels above $b-a$. Our third main result establishes this trichotomy.

\begin{theorem}[Outcomes of TS]\label{thm: outcomes of TS}
    Consider $G=\ts_\tau(n;a,b)$ where $a<b$. 
    \begin{enumerate}
        \item\label{item:thm: outcomes of TS:1} If $\tau>b-a$, then $G\in \lp$, for all $n\ge b+1$.
        \item\label{item:thm: outcomes of TS:2} If $\tau=b-a$, let $t = n\bmod (b+1)$, $0\le t\le b$. Then, for all $n\ge 0$,
        \begin{enumerate}
            \item $G\in \pp$, if $t=0$;  
            \item $G\in \lp$, if $1\le t\le \tau$;
            \item $G\in \np$, if $\tau+1\le t\le b$. 
            \end{enumerate}
        \item\label{item:thm: outcomes of TS:3} If $\tau<b-a$, then $G\in \rp$, for all $n\ge \left\lfloor\frac{b-a}{b-a-\tau}\right\rfloor\!(a+\tau+1)$.
        \end{enumerate}
\end{theorem}

Theorem~\ref{thm: outcomes of TS} proves that $b-a$ is the critical threshold. Below this threshold the ruleset is Right-dominating, and above it, the ruleset is Left-dominating. Therefore, we call the level $\kappa:=b-a$ the {\em knife's edge}. 

The second part of Theorem~\ref{thm: outcomes of TS} says that the outcome of {\sc TS} at the knife's edge is periodic with period $b+1$. Tables~\ref{tab:CF of ts_4(n,3,7)} and \ref{tab: cf of ts1(n;5,6)} suggest an even stronger phenomenon: not only is the outcome periodic, but the canonical forms of {\sc TS} at the knife's edge is also periodic with the same period length. Our next theorem establishes this periodicity of the canonical forms.


\begin{table}[h!]
    \centering
    \caption{Canonical forms of $\ts_{4}(n;3,7)$ for $0 \le n \le 15$.}
    \label{tab:CF of ts_4(n,3,7)}
    
    \setlength{\tabcolsep}{4pt}
    \renewcommand{\arraystretch}{1.2}
    
    \begin{minipage}{0.48\textwidth}
    \centering
    \small
    \begin{tabular}{|l|c|}
    \hline
    $n$ & $G$ \\ \hline
    0  & 0 \\ \hline
    1  & 1 \\ \hline
    2  & 2 \\ \hline
    3  & 3 \\ \hline
    4  & 4 \\ \hline
    5  & $\{4|0\}$ \\ \hline
    6  & $\{4, \{4|0\} | 0\}$ \\ \hline
    7  & $\{4, \{4, \{4|0\} | 0\} | 0\}$ \\ \hline
    
    \end{tabular}
    \end{minipage}
    \hfill
    \begin{minipage}{0.48\textwidth}
    \centering
    \small
    \begin{tabular}{|l|c|}
    \hline
    $n$ & $G$ \\ \hline
    8  & 0 \\ \hline
    9  & 1 \\ \hline
    10 & 2 \\ \hline
    11 & 3 \\ \hline
    12 & 4 \\ \hline
    13 & $\{4|0\}$ \\ \hline
    14 & $\{4, \{4|0\} | 0\}$ \\ \hline
    15 & $\{4, \{4, \{4|0\} | 0\} | 0\}$ \\ \hline
    \end{tabular}
    \end{minipage}
\end{table}

\begin{theorem}[Canonical Form Periodicity in TS]\label{thm: CF periodicity in TS at b-a}
    Fix $a,b$ and let $\kappa =b-a$. Let $G\coloneq\ts_{\kappa}(n;a,b)$ where $a<b$. Let $t= n \bmod (b+1)$, with $0\le t\le b$. Then $G=\ts_{\kappa}(t;a,b)$. 
\end{theorem}

Recall when playing disjunctive sums of games, information about outcome alone (Theorem~\ref{thm: outcomes of TS}) is often insufficient. And often canonical forms are not accessible. Still one needs additional information; then atomic weight comes to play.

Table~\ref{tab:aw of ts(n,3,7)} suggests a clear pattern in the atomic weights of $\ts_\tau(n;3,7)$. For $\tau<b-a$, the pattern begins exactly at heap size
$n'=\left\lfloor\frac{b-a}{b-a-\tau}\right\rfloor\!(a+\tau+1)$. By Theorem~\ref{thm: outcomes of TS}, starting from $n'$, all larger heap sizes have outcome $\rp$. If a heap size is larger than or equal to $n'$, Right can postpone making a move until Left is able to reduce the heap below $n'$. Since Left removes at most $a$ tokens per move, she needs at least
\[
\left\lceil\frac{n-(n'-1)}{a}\right\rceil
\]
to reduced the heap below $n'$. Thus, Right can delay his move for exactly $\left\lceil\frac{n-(n'-1)}{a}\right\rceil-1$ many turns, suggesting that the atomic weight is
\[
\left\lceil\frac{n-n'+1}{a}\right\rceil-1=\left\lfloor\frac{n-n'}{a}\right\rfloor.
\] 

Our fourth main result proves this formula for all $\tau<(b-a)/2$. Furthermore, we conjecture that the formula remains valid for $(b-a)/2\le\tau<b-a$. Note that {\sc TS} games are not all-small: if $\tau>0$ then a heap of size one has value $1$. Thus, atomic weights  do not exist for all possible heap sizes. However, it turns out that many {\sc TS} games are atomic, and have atomic weight.

\begin{theorem}[Atomic Weight of {\sc TS}]\label{thm: AW of TS}
    Let $a,b$ and $\tau$ are positive integers such that $a<b$. If $1\le \tau< (b-a)/2$, then for all $n\ge a+\tau+1$, 
    \begin{align}\label{eq:w}
    \aw\big(\ts_\tau(n;a,b)\big)=-\left\lfloor \frac{n-(a+\tau+1)}{a}\right\rfloor.
    \end{align}
\end{theorem}





\begin{table}[t!]
\centering
\caption{Atomic weights of $\ts_\tau(n;3,7)$ for $1 \le n \le 34$ and truncation levels $\tau\le 4$. A cell is colored red, if the cell and all the cells below it are $\rp$-positions.}
\label{tab:aw of ts(n,3,7)}

\setlength{\tabcolsep}{4pt}
\renewcommand{\arraystretch}{1.2}

\begin{minipage}{0.49\textwidth}
\centering
\small
\begin{tabular}{|l|*{6}{c|}}
\hline
\textbf{$n \backslash \tau$} & 1 & 2 & 3 & 4  \\ \hline
1 &      &      &      &           \\ \hline
 2 &      &      &      &         \\ \hline
 3 &      &      &      &          \\ \hline
 4 &      &      &      &           \\ \hline
 5 & \cellcolor{red!25}0    &      &      &           \\ \hline
 6 & \cellcolor{red!25}0    & 0    &      &          \\ \hline
 7 & \cellcolor{red!25}0   & 0    & 0    &          \\ \hline
 8 & \cellcolor{red!25}-1   & 0    & 0    & 0       \\ \hline
 9 & \cellcolor{red!25}-1   & 0    &      &           \\ \hline
 10 &\cellcolor{red!25} -1   & 1    &      &        \\ \hline
11 & \cellcolor{red!25}-2   & 1    &      &           \\ \hline
12 & \cellcolor{red!25}-2   & \cellcolor{red!25}0    &      &         \\ \hline
13 & \cellcolor{red!25}-2   & \cellcolor{red!25}0    &      &          \\ \hline
14 & \cellcolor{red!25}-3   & \cellcolor{red!25}0    & 0    &          \\ \hline
15 & \cellcolor{red!25}-3   & \cellcolor{red!25}-1   & 0    &           \\ \hline
16 & \cellcolor{red!25}\;-3 \;  &\; \cellcolor{red!25}-1\;   &\; 0\;    &\; 0\;         \\ \hline
17 & \cellcolor{red!25}-4   & \cellcolor{red!25}-1   &      &           \\ \hline
\end{tabular}
\end{minipage}
\hfill
\begin{minipage}{0.49\textwidth}
\centering
\small
\begin{tabular}{|l|*{6}{c|}}
\hline
\textbf{$n \backslash \tau$} & 1 & 2 & 3 & 4 \\ \hline
18 & \cellcolor{red!25}-4   &\cellcolor{red!25} -2   &      &            \\ \hline
19 & \cellcolor{red!25}-4   & \cellcolor{red!25}-2   &      &          \\ \hline 
20 & \cellcolor{red!25}-5   &\cellcolor{red!25} -2   &      &           \\ \hline
21 & \cellcolor{red!25}-5   &\cellcolor{red!25} -3   & 0    &          \\ \hline
22 &\cellcolor{red!25} -5   & \cellcolor{red!25}-3   & 0    &            \\ \hline
23 & \cellcolor{red!25}-6   & \cellcolor{red!25}-3   & 0    &           \\ \hline
24 &\cellcolor{red!25}-6   &\cellcolor{red!25}-4   & 0    & 0         \\ \hline
25 &\cellcolor{red!25} -6   & \cellcolor{red!25}-4   & 0    &           \\ \hline
26 &\cellcolor{red!25} -7   & \cellcolor{red!25}-4   & 1    &           \\ \hline
27 &\cellcolor{red!25} -7   & \cellcolor{red!25}-5   & 1    &            \\ \hline
28 &\cellcolor{red!25} -7   & \cellcolor{red!25}-5   &\cellcolor{red!25} 0    &            \\ \hline
29 &\cellcolor{red!25} -8   & \cellcolor{red!25}-5   &\cellcolor{red!25} 0    &            \\ \hline
30 &\cellcolor{red!25} -8   & \cellcolor{red!25}-6   &\cellcolor{red!25} 0    &            \\ \hline
31 & \cellcolor{red!25}-8   &\cellcolor{red!25} -6   & \cellcolor{red!25}-1   &            \\ \hline
32 & \;\cellcolor{red!25}-9\;   & \cellcolor{red!25}\;-6\;   & \cellcolor{red!25}\;-1\;   & \;0 \;         \\ \hline
33 & \cellcolor{red!25}-9   &\cellcolor{red!25} -7   & \cellcolor{red!25}-1   &            \\ \hline
34 & \cellcolor{red!25}-9   &\cellcolor{red!25} -7   & \cellcolor{red!25}-2   &           \\ \hline
\end{tabular}
\end{minipage}
\end{table}


The rest of the paper is organized as follows:\begin{enumerate}
    \item In Section~\ref{sec: literature review}, we cite the relevant papers and the papers that inspired this work.
    \item In Section~\ref{sec: CF of FS}, we prove our first main result, Theorem~\ref{thm: CF of FS}, about the canonical form of {\sc Full Support}.
    \item In Section~\ref{sec: AW of FS}, we compute the atomic weight of {\sc FS} proving our second main result, Theorem~\ref{thm: aw of FS}.
    \item In Section~\ref{sec: TS}, we introduce {\sc Truncated Support} and prove our third main result, Theorem~\ref{thm: outcomes of TS}, part-wise in subsections~\ref{subsec: higher truncation level}, \ref{subsec: balanced truncation level} and \ref{subsec: lower truncation level}. Along with this, we also prove Theorem~\ref{thm: CF periodicity in TS at b-a}, which says that the canonical form of {\sc TS} is periodic at the knife's edge.
    
    \item In Section~\ref{sec: AW of TS}, we prove our fourth main result, Theorem~\ref{thm: AW of TS}, which gives a partial formula for the atomic weights of {\sc TS}. 
\end{enumerate}




\section{Literature review}\label{sec: literature review}

{\sc Impartial Subtraction} games have been extensively studied and are well known for exhibiting periodic outcome sequences. In contrast, {\sc Partizan Subtraction} games are significantly more complex because the players have different move sets. Their analysis often requires algebraic tools from CGT, such as canonical forms, reduced canonical forms, temperatures, and atomic weights.

The study of {\sc Partizan Subtraction} began with the work of Fraenkel and Kotzig \cite{FA} (1987), who introduced partizan octal games. They proved that every partizan subtraction game with finite subtraction sets has an ultimately periodic outcome sequence and introduced the notion of dominance, which we refer to as Left- and Right-domination.


Mesdal et al.~\cite{M} later investigated finite partizan subtraction rulesets, where players may optionally split the heap after removal, through reduced canonical forms, extending the structural understanding of these games beyond their outcomes.

Larsson et al. \cite{LMNS} studied a different class of partizan subtraction games in which the players have complementary infinite subtraction sets generated by the $\pp$-positions of Wythoff Nim and Fibonacci sequences. They prove that every position is either a reduced canonical form switch or a canonical form number.

More recently, Duchêne et al.~\cite{DHNP} analyzed partizan subtraction games with subtraction sets of size two. Besides Left- and Right-domination, they identified three additional asymptotic behaviors of the outcome sequence: weak dominance, fairness, and ultimate impartiality. Weak dominance occurs when the outcome sequence is eventually periodic and the period contains either $\lp$ or $\rp$, but not both. The {\sc TS} ruleset at truncation level $b-a$ studied in this paper exhibits weak dominance, since its periodic outcome sequence contains $\lp$-positions but no $\rp$-positions. They prove NP-hardness of {\sc Partizan Subtraction}, if the ruleset is included in the input, by reduction to the unbounded knapsack problem. 



The {\sc Full Support} ruleset also belongs to the class of cumulative games, introduced by Larsson, Meir, and Zick \cite{URYCG2020}. In cumulative games, player cumulations are part of the rules of how to move; they may contribute to the players payoffs when the game ends. {\sc Wealth Nim} (also called {\sc Wealth Pebbles}) is one such ruleset family. {\sc Full Support} with $S_L=[a]$ and $S_R=[b]$ is equivalent to {\sc Wealth Nim} with players' wealths $a$ and $b$, and where there is no cumulative effect. 
\begin{figure}
    \centering
    \includegraphics[width=0.5\linewidth]{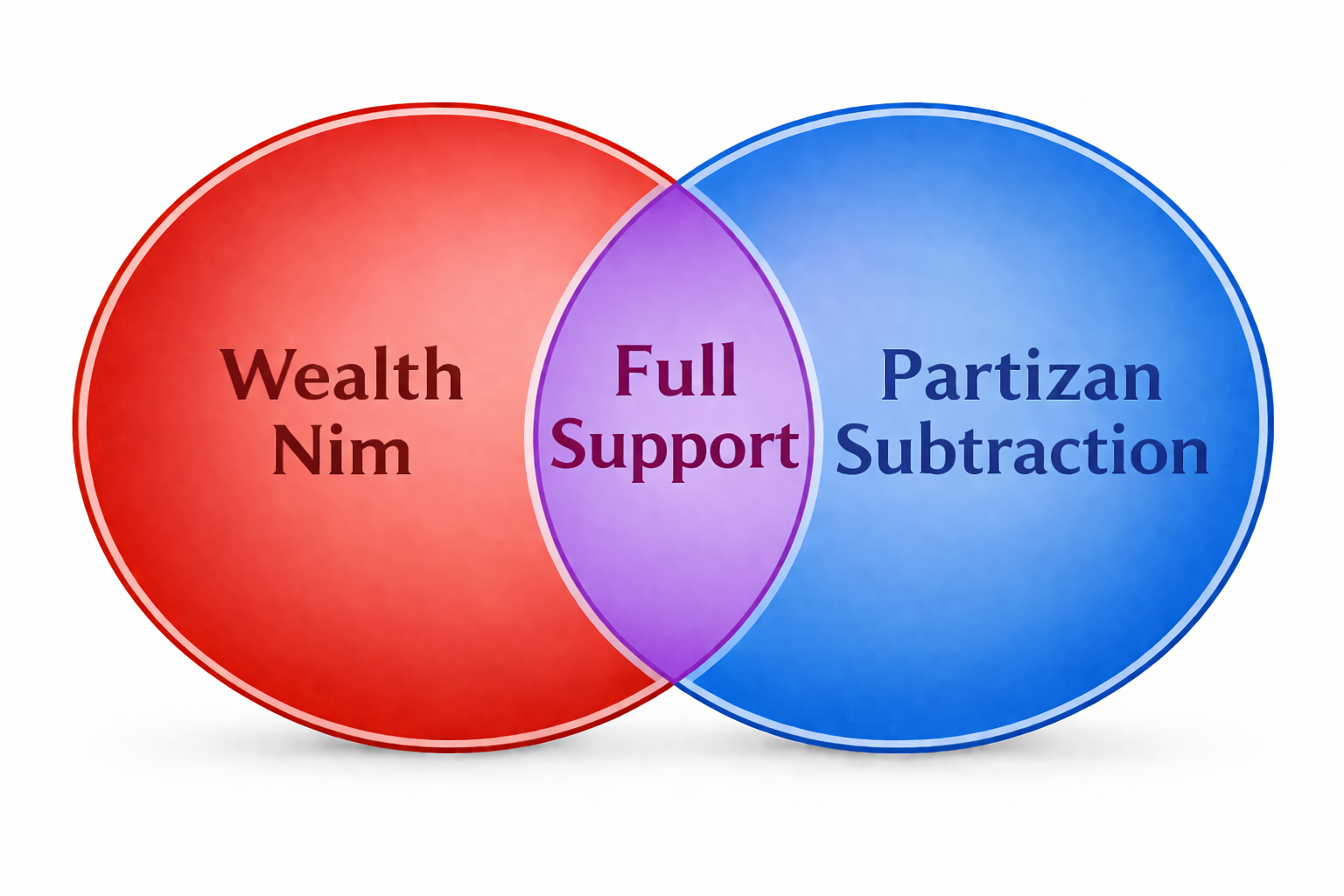}
    \caption{{\sc Full Support} lies at the intersection of {\sc Partizan Subtraction} and Cumulative Games.}
    \label{fig:placeholder}
\end{figure}

Another example of a cumulative game is {\sc Robin Hood} \cite{BDL}. 
Similar to {\sc FS}, Players in {\sc Robin Hood} can remove any number of tokens up to their wealth. However, the wealths of the players remain fixed throughout a game in {\sc FS}, whereas, in {\sc Robin Hood}, if a player remove $m$ tokens from the heap, the opponent's wealth decreases by $m$. Therefore, players want to play first in {\sc Robin Hood} and decrease their opponent's wealth. 




When both players have the same subtraction set, {\sc Full Support} becomes impartial. In this case, the nim-values were determined in \cite{BCG2004}. We restate the formula and include its folklore proof for completeness.


\begin{theorem}[Impartial FS]\label{thm: CF of FS for equal wealth}
Let $a\in\Nat_0$. For all  $n\in \Nat$,  let $r=n \bmod (a+1)$. Then $\fs(n;a,a)~=~*r.$
\end{theorem}
\begin{proof}

If $n\le a$, the game is {\sc nim}, and otherwise we  have 
\begin{align*}
\fs(n;a,a)
&= \{\fs(n-1;a,a), \fs(n-2;a,a), \dots, \fs(n-a;a,a)\} \\
&=\{*((n-1)\bmod (a+1)),\;*((n-2)\bmod(a+1)),\; \dots,\; *((n-a)\bmod(a+1))\}\\
&= \{*((r-1)\bmod(a+1)), *((r-2)\bmod(a+1)), \dots, *((r-a)\bmod(a+1))\} \\
&= *r.
\end{align*}

The second equality follows by induction, and the third is by the statement. For the final equality, observe that
\[
\{r,\; (r-1)\bmod(a+1),\; (r-2)\bmod(a+1), \dots,\; (r-a)\bmod(a+1)\}
= \{0,1,\dots,a\}.
\]
Hence, $\mathrm{mex}\{(r-1)\bmod(a+1), (r-2)\bmod(a+1), \dots, (r-a)\bmod(a+1)\}=r$.
\end{proof}

We next study the canonical form of {\sc Full Support} when the players have different subtraction sets.

\section{Canonical forms of Full Support}\label{sec: CF of FS}

We assume $a<b$ without loss of generality. 
\begin{definition}[Full Support]
    {\sc Full Support} ({\sc FS}) is a partizan subtraction ruleset played on a single heap, where the subtraction sets of Left and Right are defined by the positive integers $a$ and $b$ with $S_L=[a]$ and $S_R=[b]$. For a heap size $n$, the game is denoted by $\fs(n;a,b)$.
\end{definition}

Symmetrically, in $\fs(n;b,a)$, the subtraction sets are $S_L=[b]$ and $S_R=[a]$.
We start by exploring simple {\sc Full Support} positions.

\begin{observation}\label{obs: fs positions for small n}
    Let $G=\fs(n;a,b)$ be an instance of {\sc FS}. Then 
    \begin{enumerate}
        \item $G$ is all-small;
        \item $G=*n$ if $n\le \min(a,b)$.
    \end{enumerate}

    For the first statement, note that both players can remove one pebble whenever $n>0$, so every nonzero position has options for both players and hence, $G$ is all-small. For the second statement, if $n\le \min(a,b)$, then both players may remove any number of pebbles from $1$ to $n$. Thus $G$ coincides with a single heap of {\sc Nim} of size $n$, and therefore $G=*n$.

    \end{observation}
Next we prove that Right wins {\sc FS} for sufficiently large heap sizes. 

\begin{proposition}\label{prop: stronger wins}
    Let $a,b\in \Nat$ with $a<b$. If $n>a$, then $\fs(n;a,b)\in \rp$. 
\end{proposition}
\begin{proof}

    On Right's turn, if the heap size is at most \(b\), he removes all the tokens and wins the game. Otherwise, he removes exactly one token. Hence after Right's move, the heap is at least $b$, and therefore Left cannot remove the full heap. Moreover, by assumption $n>a$, so Left cannot win by starting. 
\end{proof}

Observe that the in the negative of a game, the subtraction sets get swapped, i.e. $-\fs(n;a,b)=\fs(n;b,a)$. In the next Lemma, we compare different {\sc FS} positions with the same Left and Right subtraction sets. We will use it later to find the canonical form of {\sc FS}.
\begin{lemma}\label{lem: game position comparison FS}
    Fix $a,b\in \Nat$, with $a<b$, and let $G(n):=\fs(n;a,b)$. Suppose that $n> m\ge 0$. 
    \begin{enumerate}
        \item\label{item: lem: game position comparison FS:1} if $n-m\le a$, then $G(n)\cgfuzzy G(m)$;
        \item if $n-m > a$, then  $G(n)<G(m)$.
    \end{enumerate}
\end{lemma}
\begin{proof}
    Consider item~$(1)$. We show that the first player wins $G(n)-G(m)$. Since \(0< n - m \leq a\), the starting player equalizes the heap sizes. This leads to the game $G(m)-G(m) =0$. Hence, the first player wins \( G(n) - G(m) \).

    Consider item (2). It suffices to prove that Right wins the game $G(n) - G(m)$. 

    If Right starts, he plays on the smaller heap $m$ and wins by induction, unless $m=0$, in which case, Right wins by Proposition~\ref{prop: stronger wins}.

    If Left starts and plays on the smaller heap $m$, Right wins by the same argument as before. If Left plays instead on the larger heap, she can remove at most $a$ tokens, hence again by induction, Right wins.  
\end{proof}

An interesting consequence of Lemma~\ref{lem: game position comparison FS} (1) is that none of the Left options dominate the other.

\begin{lemma}[Left Domination]\label{cor: no domination}
    The Left options of $\fs(n;a,b)$ are confused if $a<b$.
\end{lemma}
\begin{proof}
    The possible Left options of $G$ are 
    $$\fs(n-1;a,b),\fs(n-2;a.b),\dots,\fs(n-a;a,b).$$
    Any two of these options differ in heap size by at most $a-1$. Hence, by Lemma~\ref{lem: game position comparison FS}(\ref{item: lem: game position comparison FS:1}), all these options are confused. Thus, none of them dominates another.
\end{proof}

The next Lemma uses Lemma~\ref{lem: game position comparison FS} and reversibility to simplify the game form of {\sc FS}.

\begin{lemma}[Game Equivalence]\label{lem: game equivalence}
Let $G=\fs(n;a,b)$ be an instance of {\sc FS}, with $a<b$. If $n>a$, then
\begin{equation*}
    G=\Set{\fs(n-1;a,b),\; \fs(n-2;a,b),\; \dots,\; \fs(n-a;a,b)\mid 0}.
\end{equation*}
\end{lemma}




\begin{proof}
    We prove the statement by inducting on $n$. Let \[ H \coloneq\Set{\fs(n-1;a,b),\; \fs(n-2;a,b),\; \dots,\; \fs(n-a;a,b) \mid 0}.\] To prove that $G=H$, it suffices to  demonstrate that \( G - H \in \mathscr{P}\).

    If $n=a+1$, then
    \begin{align*}
        G &= \Set{\fs(1;a,b),\;\fs(2;a,b),\dots,\fs(a;a,b) \mid  \fs(0;a,b),\;\fs(1;a,b),\dots,\fs(a;a,b)}\\ 
        &=\Set{*,*2,\dots,*a\mid 0,*,\dots,*a} \tag{Obs~\ref{obs: fs positions for small n}}\\
        &=\Set{*,*2,\dots,*a\mid 0} \tag{by reversibility}\\
        &=\Set{\fs(1;a,b), \;\fs(2;a,b), \dots, \fs(a;a,b)\mid0}.
    \end{align*}
    As $G<0$ by Proposition~\ref{prop: stronger wins}

    Assume by induction that the statement holds for all options of $G$ with heap size greater than $a$.
    
    Now, if Left plays in $G$, then Right can mimic in $-H$ and win and similarly, if Right starts in $-H$ then Left can mimic in $G$ and win as the Left options are the same in both games. If Left starts and plays in $-H$, the game leads to $G+0$ and Right wins by Proposition~\ref{prop: stronger wins}.
    
    The only remaining case is when Right starts and plays in $G$. Let $G_j\coloneq\fs(n-j;a,b)$ where $j\in[b]$. Right's move in $G$ leads the game to $G_j-H$ for some $j\in[b]$. We divide this case into further cases depending on the size of $n-j$:

    \begin{enumerate}
        \item \bm{$n-j\leq a$}: In this case Left moves from $G_j-H$ to $0-H$ as $G_j=*(n-j)$ by Observation~\ref{obs: fs positions for small n}. From this position, Right's move leads to $\fs(n-k;b,a)$ for some $k\in[a]$. Since $n>a$, $n-k>0$, and hence, Left wins by Proposition~\ref{prop: stronger wins} and Observation~\ref{obs: fs positions for small n}.

        \item \bm{$n-j>a$}: In this case Left removes $a$ from $G_j$ leading the game to $\fs(n-j-a;a,b)-H$. From here Right can move on either component, which leads to two cases:
        \begin{itemize}
            \item If Right makes a move on $-H$ and removes $k\le a$, Left plays first in $\fs(n-j-a;a,b)+\fs(n-k;b,a)$. Since $k\in[a]$, $n-j-a<n-k$, and Left wins by Lemma~\ref{lem: game position comparison FS}.

            \item If Right moves on $\fs(n-j-a;a,b)$ and $n-j-a>a$, then by induction, the game reduces to $-H$ and Left wins by moving from $-H$ to 0. If $n-j-a\le a $, then by Observation~\ref{obs: fs positions for small n}, Right's move leads the game to $*m-H$ where $0\le m<n-j-a$. If $m=0$, Left wins by Proposition~\ref{prop: stronger wins}, and otherwise, Left wins by moving from $*m-H$ to $0-H$ as Right cannot end $-H$ immediately and by induction, Left can end the game in her next turn. 
        \end{itemize}
    \end{enumerate}This shows that $G-H$ is a $\mathscr{P}$-position, which completes the proof. 
\end{proof}

\begin{observation}\label{cor: increasing wealth ineffective}
    Consider $n,a,b\in \Nat$ and let $b>a$. Then, $\fs(n;a,b)=\fs(n;a,b+1)$.  
\end{observation}

This is immediate by Lemma~\ref{lem: game equivalence}. Namely, by induction, the Left options are the same, and the Right option is zero in both games. 

Now we prove the first main result, Theorem~\ref{thm: CF of FS}. This result establishes that the game form presented in Lemma~\ref{lem: game equivalence} cannot be simplified further; it already is the canonical form of the game.


\begin{proof}[Proof of Theorem~\ref{thm: CF of FS}]
Let $G=\Set{\fs(n-1;a,b), \fs(n-2;a,b),\dots,\fs(n-a;a,b)\mid 0}$. By Lemma~\ref{lem: game equivalence} and Lemma~\ref{cor: no domination}, it suffices to prove that none of the Left options reverses out. That is, we prove that, for all $G^L$, for all $G^{LR}$, $G\ngeq G^{LR}$. This is immediate by Lemma~\ref{lem: game position comparison FS}, since $G^{LR}$ is of the same form, but with a smaller heap size.   
\end{proof}

So far, we have analyzed {\sc FS} positions with different subtraction sets. The following theorem provides the canonical form of {\sc FS} when both players have the same subtraction set.

We have seen that the Right (the stronger player) always has an advantage for sufficiently large heap sizes. The next section quantifies this advantage using atomic weight theory.


\section{Atomic weights of Full Support}\label{sec: AW of FS}


Intuitively, the atomic weight of an all-small game can be interpreted as the number of moves a player can delay playing before the opponent can win the game. Now, consider an instance of {\sc FS}, $\fs(n;a,b)$ with $n,a,b>0$ and $a\le b$. If $a=b$, then by Theorem~\ref{thm: CF of FS for equal wealth}, $\fs(n;a,b)$ is a nimber and hence, the atomic weight of $\fs(n;a,b)$ is 0 by Lemma~\ref{lem: aw of nimbers}. If $n\leq a$, the position $\fs(n;a,b)$ is again a nimber, and hence the atomic weight $0$. Now consider the case $n>a$ and $a\ne b$. In this case, Left cannot delay playing as Right has an immediate move to $0$ from $\fs(n;a,b)$ by Theorem~\ref{thm: CF of FS}. In contrast, Right can delay playing as long as the heap size remains strictly greater than $a$, as Left cannot end the game. Left can remove at most $a$ tokens per move, and thus the minimum number of moves Left requires to reduce the heap size below $a+1$ is \begin{align}
    \lceil (n-a)/a\rceil=\lceil n/a\rceil-1.\label{eq: FS Left required moves}
\end{align}
Thus, Right can delay playing for $\lceil n/a\rceil -1$ moves. Let $$f(n,a)\coloneq\lceil n/a\rceil-1.$$ 
So the atomic weight of $\fs(n;a,b)$ is expected to be $-f(n,a)$ for $a<b$ and $n>a$. We now prove this claim using Definition~\ref{def: atomic weight} and Theorem~\ref{thm: check farstareq}. The theorem requires the outcome of $\fs(n;a,b)+\cgfarstar$ and the next lemma facilitates this computation.

\begin{lemma}\label{lem: outcome of G+farstar}
Let $n,a,b\in\Nat$ with $a<b$ and $n>a$. Then $\fs(n;a,b)+\cgfarstar$ is an $\rp$-position.
\end{lemma}

\begin{proof}
Let $G=\fs(n;a,b)$. By Theorem~\ref{thm: CF of FS}, the canonical form of $G$ is
\[
G=\Set{\fs(n-1;a,b), \dots, \fs(n-a;a,b)\mid 0}.
\]
We show that Right wins $G+\cgfarstar$ regardless of who starts.

If Right starts, he moves in $\cgfarstar$ to $0$, reducing the position to $G$. Since $n>a$, Right wins $G$ by Proposition~\ref{prop: stronger wins}.

Now suppose Left starts. She may either move in $G$ or in $\cgfarstar$.
\begin{itemize}
\item If Left moves in $G$ to $\fs(n-k;a,b)$ where $1\le k\le a$, then:
If $n-k\le a$, Right moves in $\cgfarstar$ to $*(n-k)$. Since $\fs(n-k;a,b)=*(n-k)$ in this case, the position becomes $*(n-k)+*(n-k)=0$, and Right wins. If $n-k>a$, Right moves in $\cgfarstar$ to $0$, reducing the position to $\fs(n-k;a,b)$, which Right wins by Proposition~\ref{prop: stronger wins}.

\item If Left moves in $\cgfarstar$ to $*m$ where $m\ge 0$, then:
If $m>0$, Right moves from $*m$ to $0$, reducing the position to $G$, which he wins by Proposition~\ref{prop: stronger wins}. If $m=0$, then $G+*0=G$, and Right wins $G$ by Proposition~\ref{prop: stronger wins}.
\end{itemize}
This completes the proof.
\end{proof}

Now we restate and prove the second main theorem of this section.

\textbf{Theorem~\ref{thm: aw of FS}.} \textit{Consider $G=\fs(n;a,b)$, where $a,b$ and $n$ are positive integers. Then the following hold:
    \begin{enumerate}
        \item If $a=b$, then $\aw(G)=0$ for all $n$;
        \item If $a<b$, then $\aw(G)=-f(n,a)$ for all $n$.
    \end{enumerate}}
\begin{proof}
    Item~(1): $G$ is a nimber by Theorem~\ref{thm: CF of FS for equal wealth}. Hence, by Lemma~\ref{lem: aw of nimbers}, the atomic weight of $G$ is $0$.
    
    Item~(2): Observe that $f(n,a)=f(n-a,a)+1$. By Theorem~\ref{thm: check farstareq}, it suffices to prove that $\cgdown\;\cgfarstar<G(n)+f(n,a)\cdot\cgup<\cgup\cdot\cgfarstar$ which is equivalent to:
    \begin{enumerate}
        \item[(i)] Left wins $G(n)+(f(n,a)+1)\cdot\cgup\;\cgfarstar$; and
        \item[(ii)] Right wins $G(n)+(f(n,a)-1)\cdot\cgup\;\cgfarstar$.
    \end{enumerate}
    For (i), if $n\le a$, the game simplifies to $*n+\cgup\;\cgfarstar>0$, and hence Left wins. So Suppose $n>a$. If Left start, she removes $a$ pebbles from $G(n)$, and thus plays to $H\coloneq G(n-a)+(f(n,a)+1)\cdot\cgup\;\cgfarstar$. By making this move, she reduces the delay advantage of Right on $G$-component by $1$. By induction, the atomic weight of this game is $\aw(H)= -f(n-a,a)+f(n,a)+1=2$, and therefore Left wins by the two-ahead-rule~\ref{thm: outcome using aw}. If Right starts and plays in $G(n)$, the game reduces to $(f(n,a)+1)\cdot\cgup\;\cgfarstar$ as Right has only one option in $G(n)$, namely to 0, by theorem~\ref{thm: CF of FS}. Since $\aw((f(n,a)+1)\cdot\cgup\;\cgfarstar)\ge 2$, Left wins by two-ahead-rule. So suppose he instead plays in the other component, then the game reduces to $G(n)+f(n,a)\cdot\cgup\;\cgfarstar$ by Lemma~\ref{lem: n-cgup-cgfarstar}. This move reduces the atomic weight of the game by $1$. Now Left responds by removing $a$ pebbles from $G$-component, leading the game to $G(n-a)+f(n,a)\cdot\cgup\;\cgfarstar$, which Left wins by induction as $f(n,a)=f(n-a,a)+1$. 

    For item~(ii), if $n\le a$, the game simplifies to $*n+\cgdown\;\cgfarstar$, which is $<0$ and hence Right wins. Furthermore, if $a<n\le 2a$, then the game is $G(n)+\cgfarstar$, and by lemma~\ref{lem: outcome of G+farstar}, Right wins. So suppose $n>2a$. If Right starts, he moves on the second component and reduces the game to $G(n)+(f(n,a)-2)\cdot\cgup\;\cgfarstar$. By doing so, he decrease the atomic weight by $1$. Left has only one option in the second component, namely to 0. If she plays on that component, the game reduces to $G(n)$, and hence Right wins by Lemma~\ref{prop: stronger wins}. So suppose she removes $x$ tokens from $G(n)$, and thus plays to $J\coloneq G(n-x)+(f(n,a)-2)\cdot\cgup\;\cgfarstar$. If $f(n-x,a)=f(n,a)$, then Right wins by two-ahead-rule. Otherwise, Right reduces the atomic weight by 1 by moving on the second component and wins by two-ahead-rule.
\end{proof}


We have shown that if $a< b$ in {\sc FS}, Right always has an advantage for all sufficiently large heaps. The next section concerns a modified ruleset, which gives a way to strengthen Left.


\section{Outcomes and canonical forms of Truncated Support}\label{sec: TS}

{\sc Truncated Support} is a modification of {\sc Full Support} where we delete small moves from the subtraction set of the stronger player in {\sc FS}. We assume $a<b$ without loss of generality.

\begin{definition}[Truncated Support]
     {\sc Truncated Support} ({\sc TS}) is a partizan subtraction ruleset played on a single heap, where the subtraction sets for Left and Right are determined by the positive integers $a$, $b$ and $\tau$, where the truncation level $1\le \tau<b$ with $S_L=[a]$ and $S_R=\Set{\tau+1,\;\tau+2,\dots,\; b}$. 
    For a heap size $n$, the game is denoted by $\ts_\tau(n; a, b)$. 
\end{definition}

Symmetrically, in $\ts_\tau(n;b,a)$, if $b>a$, the subtraction sets are $S_L=\{\tau+1,\dots,b\}$ and $S_R=[a]$.
We omit the truncation level subscript $\tau$ whenever it is fixed or clear from the context. 

We start by analyzing {\sc TS} positions for small heap sizes, for all truncation levels.
\begin{observation}\label{obs: TS=n if n<=tau}
    Let $a,b$ and $\tau$ be positive integers such that $a<b$ and $\tau<b$. 
If $n\le \tau$, then Right does not have any move, whereas Left can remove tokens one by one. Thus, for all $1\le n\le\tau$,  $\ts_\tau(n;a,b)=n$. 
\end{observation}

We now find the outcomes of {\sc TS} for heap sizes no larger than $b$, for all truncation levels.
\begin{lemma}\label{lem: outcome TS heap<=b}
    Consider $G=\ts_\tau(n;a,b)$ where $a<b$ and $1\le \tau<b$. Then, 
    \begin{enumerate}
        \item\label{item:lem: outcome TS heap<=b:1} $o(G)=\lp$ if $n\le \tau$;
        \item\label{item:lem: outcome TS heap<=b:2} $o(G)=\np$ if $\tau+1\le n\le \min(\tau+a,b)$;
        \item\label{item:lem: outcome TS heap<=b:3} $o(G)=\rp$ if $\min(\tau+a,b)+1\le n\le b$.
    \end{enumerate}
\end{lemma}
\begin{proof}
    If $n\le \tau$, Right does not have any legal move on $\ts_\tau(n;a,b)$ and Left can remove tokens one by one. Thus, by domination, $\ts_\tau(n;a,b)=n$ and Left wins $G$. 
    
    Suppose $\tau+1\le n\le \min(\tau+a,b)$. If Left starts on $G$, she can reduce the heap size to $\tau$ as $1\le n-\tau\le a$ and win by item~(1). If Right starts, he wins by removing the full heap as $\tau+1\le n\le b$.

    Now Suppose $\min(\tau+a,b)+1\le n\le b$. If $b\le \tau+a$, there is nothing to prove. Otherwise, if Right starts on $G$, he wins by removing the full heap. If Left starts on $G$, she can reduce the heap size at most to $\tau+1$ and then, Right removes the full heap in the next turn and wins.
\end{proof}

For fixed $a$ and $b$, as we increase $\tau$, Right loses access to increasing number of small moves and hence, intuitively loses his power. Figure~\ref{fig: outcomes of ts(n,3,7)} indicates the same. As discussed in the Section~\ref{sec: introduction FSTS}, the {\sc TS} ruleset seems to be Right-dominating for truncation level below $b-a$, balanced at truncation level $b-a$, and Left-dominating for truncation level above $b-a$. In the next three subsections, we work on these 3 cases, starting with Higher truncation level. 

\subsection{Higher truncation levels}\label{subsec: higher truncation level}
Throughout this subsection, we assume $\tau>b-a$. The outcomes of $\ts_\tau(n;a,b)$ for heap sizes $n\le b$ are already determined by Lemma~\ref{lem: outcome TS heap<=b}.
We now prove item~(1) of Theorem~\ref{thm: outcomes of TS} which states that the outcome of $\ts_\tau(n;a.b)$ is $\lp$ for all $n\ge b+1$ if $\tau>b-a$. 

\begin{proof}[Proof of Theorem~\ref{thm: outcomes of TS}(\ref{item:thm: outcomes of TS:1})]
    First, we consider the case $n=b+1$. If Left starts, she removes $b+1-\tau$ tokens, reaching $\ts_\tau(\tau;a,b)=\tau$, and hence wins; 
    this move is legal as $1<b+1-\tau\le a$. If Right starts, his move leads to $\ts_\tau(j;a,b)$ where $1\le j\le b+1-(\tau+1)<a$. Left then wins by removing the remaining heap. 
    
    Now suppose $n>b+1$. If Left starts, she removes one token and win as the resulting position is $\lp$ by the induction hypothesis because $n-1\ge b+1$. If Right starts, his move leads to a non-empty heap as $n>b+1$. By the induction hypothesis together with Lemma~\ref{lem: outcome TS heap<=b}, every resulting position is either $\lp$ or $\np$. In either case, Left wins as the next player.
\end{proof}

Thus {\sc TS} is Left-dominating for all truncation levels above $b-a$. Next we analyze the truncation level $b-a$.

\subsection{Middle truncation level}\label{subsec: balanced truncation level}
Throughout this subsection, we fix the truncation level $\kappa:=b-a$. 
Recall that a ruleset is balanced if it is neither Left- nor Right-dominating. 
We now prove item~(2) of Theorem~\ref{thm: outcomes of TS}, which states that the outcomes of {\sc TS} at the truncation level $\kappa$ are periodic with period $b+1$. As a consequence, the {\sc TS} ruleset is balanced at this truncation level. 

\begin{proof}[Proof of Theorem~\ref{thm: outcomes of TS}(\ref{item:thm: outcomes of TS:2})]
    Given $t=n\bmod (b+1)$, we must show that $G\in \pp$ if $t=0$, $G\in \lp$ if $1\le t\le \kappa$ and $G\in \np$ if $\kappa+1\le t\le b$.

    If $n\le b$, then $t=n$ and the statement is the same as Lemma~\ref{lem: outcome TS heap<=b}. Hence assume $n>b$. 
    
    
    %
    
    Consider the case when $t=0$. We must prove $G\in \pp$. 
    
    Suppose Left starts on $G$ and her move leads to heap size $n_\ell$. Then $n_\ell\ge n-a$. Since $t=0$ and $b-a={\kappa}$, $$ n_\ell \bmod (b+1) \ge {\kappa}+1,$$ and $\ts_{\kappa}(n_\ell;a,b)\in \np$ by induction. Thus, Right wins. 
    
    Now suppose Right starts and moves to heap size $n_r$. Then $ n_r\ge n-b$ and thus, $n_r \bmod (b+1)\ge 1$, and the outcome of $\ts_{\kappa}(n_r;a,b)$ is either $\lp$ or $\np$ by induction. Therefore, Left wins. 

    Consider the case when $1\le t\le {\kappa}$. We must prove $G\in\lp$. 
    
    Suppose Right starts and moves to heap size $n_r$. Then $n-b\le n_r\le n-({\kappa}+1)$. Since $1\le t\le {\kappa}$, $$2\le n_r \bmod (b+1)\le b,$$ and $\ts_{\kappa}(n_r;a,b)\in \lp\cup\np$, by induction. Thus, Left wins. 
    
    Suppose Left starts. By using induction, she wins by moving to heap size $n-1$, as $0\le (n-1) \bmod (b+1)\le {\kappa}-1$. 

    Consider the case when ${\kappa}+1\le t\le b$. We must prove $G\in \np$.
    
    If Right starts, by induction he wins by moving to heap size $n-t$ as $(n-t)\bmod (b+1)=0$; This is a legal move as ${\kappa}+1\le t\le b$.
    
    If Left starts, by induction she wins by removing $t-{\kappa}$ tokens as $(n-t+{\kappa})\bmod (b+1)={\kappa}$; This is a legal move as $1\le t-{\kappa}\le b-{\kappa}=a$.
\end{proof}


Tables~\ref{tab:CF of ts_4(n,3,7)} and \ref{tab: cf of ts1(n;5,6)} illustrate that the canonical form of {\sc TS} at truncation level $\kappa$ is also periodic with the period $b+1$. We prove this periodicity of the canonical form in the next theorem. We showed that in {\sc Full support}, all Right options reverses out to $0$, and none of the Left options are dominated. However, in {\sc TS} at truncation level $b-a$, both phenomenons are not hold.
In $\ts_1(n;5,6)$ (Table~\ref{tab: cf of ts1(n;5,6)}), not all the Right options reverse out to 0 and the Left option $\ts_1(2;5,6)$ is dominated by $\ts_1(3;5,6)$ and $\ts_1(4;5,6)$ is dominated by $\ts_1(5;5,6)$. 



%


\begin{table}[h!]
\centering
\caption{Left and Right options in the canonical form of $\ts_{1}(n;5,6)$ for $0 \le n \le 14$, Where $\ts_1(5;5,6)=\{1,\{1,\{1\vert 0\}\vert 0\},\{1,\{1,\{1\vert 0\}\vert 0\}\vert 0,\{1\vert 0\}\}\vert 0,\{1\vert 0\}\}$}
\label{tab: cf of ts1(n;5,6)}
\renewcommand{\arraystretch}{1.2}
\begin{tabular}{|c|p{7.6cm}|p{6cm}|}
\hline
$n$ & Left options & Right options \\
\hline
0  & $\varnothing$ & $\varnothing$ \\ \hline
1  & $0$ & $\varnothing$ \\ \hline
2  & $1$ & $0$ \\ \hline
3  & $1,\;\;\{1\vert 0\}$ & $0$ \\ \hline
4  & $1,\;\;\{1,\{1\vert 0\}\vert 0\}$ & $0,\;\;\{1\vert 0\}$ \\ \hline
5  & $1,\;\;\{1,\{1\vert 0\}\vert 0\},\;\;\{1,\{1,\{1\vert 0\}\vert 0\}\vert 0,\{1\vert 0\}\}$ 
   & $0,\;\;\{1\vert 0\}$ \\ \hline
6  & $1,\;\;\{1,\{1\vert 0\}\vert 0\},\;\;\ts_1(5;5,6) $ 
   & $0,\;\;\{1\vert 0\},\;\;\{1,\{1,\{1\vert 0\}\vert 0\}\vert 0,\{1\vert 0\}\}$ \\ \hline
7  & $\varnothing$ & $\varnothing$ \\ \hline
8  & $0$ & $\varnothing$ \\ \hline
9  & $1$ & $0$ \\ \hline
10 & $1,\;\;\;\;\{1\vert 0\}$ & $0$ \\ \hline
11 & $1,\;\;\;\{1,\{1\vert 0\}\vert 0\}$ & $0,\;\;\{1\vert 0\}$ \\ \hline
12 & $1,\;\;\{1,\{1\vert 0\}\vert 0\},\;\;\{1,\{1,\{1\vert 0\}\vert 0\}\vert 0,\{1\vert 0\}\}$ 
   & $0,\;\;\{1\vert 0\}$ \\ \hline
13 & $1,\;\;\{1,\{1\vert 0\}\vert 0\},\;\;\ts_1(5;5,6) $ 
   & $0,\;\;\{1\vert 0\},\;\;\{1,\{1,\{1\vert 0\}\vert 0\}\vert 0,\{1\vert 0\}\}$ \\ \hline
14 & $\varnothing$  & $\varnothing$ \\ \hline
\end{tabular}
\end{table}


\noindent {\bf Theorem~\ref{thm: CF periodicity in TS at b-a}} 
Let $G\coloneq\ts_{{\kappa}}(n;a,b)$ where $a<b$. 
Let $t= n \bmod (b+1)$, with $0\le t\le b$. Then $G=\ts_{{\kappa}}(t;a,b)$.
\begin{proof}
    If $0\le n\le b$, then $t=n$, so the statement is immediate. Hence, suppose $n=v(b+1)+t$ where $v\ge 1$. Let $H=\ts_{{\kappa}}(t;b,a)$. Then, it suffices to prove that $G+H\in \pp$. 
    
    If Left starts and plays on $G$ to $\ts_{{\kappa}}\big(v(b+1)+t-i;a,b\big)+H$ where $i\in [a]$. Then Right removes $b+1-i$ tokens from the same component. This is legal as $b-a+1\le b+1-i\le b$ for all $i\in[a]$. This results in the position $\ts_{\kappa}\big((v-1)(b+1)+t;a,b\big)+H$, which is a $\pp$-position by induction. Hence, Right wins $G+H$. 

    If Left instead starts by moving on $H$, then Right mimics her move on $G$, removing the same number of tokens. 
    This is legal because the Left's subtraction set in $H$ coincides with the Right's subtraction set in $G$,
    and the heap size in $G$ is larger. Right's move leads to $\ts\big(v(b+1)+ t';a,b\big)+\ts( t';b,a)$ for some $0\le  t'< t$. Right continues this mimic strategy as long as Left plays on $H$. 
    Eventually, Left has no legal move remaining in $H$, so she is forced to make a move on $G$. As soon as Left moves on $G$, Right wins by the previous argument.

    If Right starts and plays on $G$ to $\ts_{\kappa}\big(v(b+1)+t-j;a,b\big)+H$ where ${\kappa}+1\le j\le b$. Then Left removes $b+1-j$ tokens from the same component. This is legal as $1\le b+1-j\le a$ for all $j$. This leads the game to $\ts_{\kappa}\big((v-1)(b+1)+t;a,b\big)+H$, which is a $\pp$-position by induction. Hence, Left wins $G+H$. 

    If Right instead starts by moving on $H$, Left mimics his move on $G$, removing the same number of tokens. Right continues this mimic strategy as long as Left plays on $H$. Once Right plays on component $G$, Left wins by the previous argument. Right is forced to make a move on $G$ eventually, as he no longer have a move on component $H$ at that point. 
\end{proof}

Theorem~\ref{thm: CF periodicity in TS at b-a} shows that there are at most $b+1$ distinct canonical forms of {\sc TS} at the knife's edge. Among these, we already know the first ${\kappa}+1$ canonical forms (for heap sizes $0,1,\dots,{\kappa}$) by Lemma~\ref{lem: outcome TS heap<=b}. While a complete description of the remaining canonical forms remains open, we make partial progress by identifying the dominated options of these positions. For details, refer to Proposition~\ref{prop: domination at b-a}. 



In the next subsection, we analyze the {\sc TS} ruleset at lower truncation levels.


\subsection{Lower truncation levels}\label{subsec: lower truncation level}

Throughout this subsection, we assume $\tau<b-a$. Figures~\ref{fig: outcomes of ts(n,3,7)}, \ref{fig: outcome of ts_k(n;5,11)}, together with many other experiments, suggest that the {\sc TS} ruleset is Right-dominating at all lower truncation levels ($\tau<b-a$). That is, the outcome of $\ts_\tau(n;a,b)$ is $\rp$ for all large $n$.

To establish this result, we first analyze the outcome pattern of {\sc TS} for smaller heap sizes, where the outcome has not yet stabilized as $\rp$.
Fix $a$, $b$, and $\tau$. We partition the heap sizes $n\ge1$ into consecutive blocks of length $b+1$. The $\beta$-th {\em block} is the interval
\begin{equation}
[(\beta-1)(b+1)+1,\; \beta(b+1)],
\label{eq:beta}
\end{equation}
where $\beta\ge1$. We will use $\beta$ to index these blocks.

The next two lemmas show that, before the outcome stabilizes as $\rp$, each block begins with a segment of $\lp$-positions, followed by exactly $a$ consecutive $\np$-positions, then a segment of $\rp$-positions, and finally a single $\pp$-position at the end of the block.


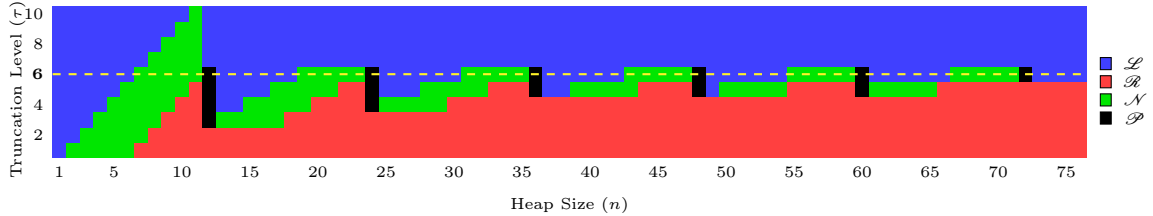
\begin{figure}[h!]
\centering
\begin{tikzpicture}[x=0.45cm, y=0.5cm, scale=0.4]
  \tikzset{fade/.style={fill opacity=1}}

  \colorlet{cL}{blue!75}
  \colorlet{cR}{red!75}
  \colorlet{cN}{green!90!black} 
  \colorlet{cP}{black}
  
  \fill[cL,fade] (0.5, 0.5) rectangle (1.5, 1.5);
  \fill[cN, fade] (1.5, 0.5) rectangle (2.5, 1.5);
  \fill[cN, fade] (2.5, 0.5) rectangle (3.5, 1.5);
  \fill[cN, fade] (3.5, 0.5) rectangle (4.5, 1.5);
  \fill[cN, fade] (4.5, 0.5) rectangle (5.5, 1.5);
  \fill[cN, fade] (5.5, 0.5) rectangle (6.5, 1.5);
  \fill[cR, fade] (6.5, 0.5) rectangle (7.5, 1.5);
  \fill[cR, fade] (7.5, 0.5) rectangle (8.5, 1.5);
  \fill[cR, fade] (8.5, 0.5) rectangle (9.5, 1.5);
  \fill[cR, fade] (9.5, 0.5) rectangle (10.5, 1.5);
  \fill[cR, fade] (10.5, 0.5) rectangle (11.5, 1.5);
  \fill[cR, fade] (11.5, 0.5) rectangle (12.5, 1.5);
  \fill[cR, fade] (12.5, 0.5) rectangle (13.5, 1.5);
  \fill[cR, fade] (13.5, 0.5) rectangle (14.5, 1.5);
  \fill[cR, fade] (14.5, 0.5) rectangle (15.5, 1.5);
  \fill[cR, fade] (15.5, 0.5) rectangle (16.5, 1.5);
  \fill[cR, fade] (16.5, 0.5) rectangle (17.5, 1.5);
  \fill[cR, fade] (17.5, 0.5) rectangle (18.5, 1.5);
  \fill[cR, fade] (18.5, 0.5) rectangle (19.5, 1.5);
  \fill[cR, fade] (19.5, 0.5) rectangle (20.5, 1.5);
  \fill[cR, fade] (20.5, 0.5) rectangle (21.5, 1.5);
  \fill[cR, fade] (21.5, 0.5) rectangle (22.5, 1.5);
  \fill[cR, fade] (22.5, 0.5) rectangle (23.5, 1.5);
  \fill[cR, fade] (23.5, 0.5) rectangle (24.5, 1.5);
  \fill[cR, fade] (24.5, 0.5) rectangle (25.5, 1.5);
  \fill[cR, fade] (25.5, 0.5) rectangle (26.5, 1.5);
  \fill[cR, fade] (26.5, 0.5) rectangle (27.5, 1.5);
  \fill[cR, fade] (27.5, 0.5) rectangle (28.5, 1.5);
  \fill[cR, fade] (28.5, 0.5) rectangle (29.5, 1.5);
  \fill[cR, fade] (29.5, 0.5) rectangle (30.5, 1.5);
  \fill[cR, fade] (30.5, 0.5) rectangle (31.5, 1.5);
  \fill[cR, fade] (31.5, 0.5) rectangle (32.5, 1.5);
  \fill[cR, fade] (32.5, 0.5) rectangle (33.5, 1.5);
  \fill[cR, fade] (33.5, 0.5) rectangle (34.5, 1.5);
  \fill[cR, fade] (34.5, 0.5) rectangle (35.5, 1.5);
  \fill[cR, fade] (35.5, 0.5) rectangle (36.5, 1.5);
  \fill[cR, fade] (36.5, 0.5) rectangle (37.5, 1.5);
  \fill[cR, fade] (37.5, 0.5) rectangle (38.5, 1.5);
  \fill[cR, fade] (38.5, 0.5) rectangle (39.5, 1.5);
  \fill[cR, fade] (39.5, 0.5) rectangle (40.5, 1.5);
  \fill[cR, fade] (40.5, 0.5) rectangle (41.5, 1.5);
  \fill[cR, fade] (41.5, 0.5) rectangle (42.5, 1.5);
  \fill[cR, fade] (42.5, 0.5) rectangle (43.5, 1.5);
  \fill[cR, fade] (43.5, 0.5) rectangle (44.5, 1.5);
  \fill[cR, fade] (44.5, 0.5) rectangle (45.5, 1.5);
  \fill[cR, fade] (45.5, 0.5) rectangle (46.5, 1.5);
  \fill[cR, fade] (46.5, 0.5) rectangle (47.5, 1.5);
  \fill[cR, fade] (47.5, 0.5) rectangle (48.5, 1.5);
  \fill[cR, fade] (48.5, 0.5) rectangle (49.5, 1.5);
  \fill[cR, fade] (49.5, 0.5) rectangle (50.5, 1.5);
  \fill[cR, fade] (50.5, 0.5) rectangle (51.5, 1.5);
  \fill[cR, fade] (51.5, 0.5) rectangle (52.5, 1.5);
  \fill[cR, fade] (52.5, 0.5) rectangle (53.5, 1.5);
  \fill[cR, fade] (53.5, 0.5) rectangle (54.5, 1.5);
  \fill[cR, fade] (54.5, 0.5) rectangle (55.5, 1.5);
  \fill[cR, fade] (55.5, 0.5) rectangle (56.5, 1.5);
  \fill[cR, fade] (56.5, 0.5) rectangle (57.5, 1.5);
  \fill[cR, fade] (57.5, 0.5) rectangle (58.5, 1.5);
  \fill[cR, fade] (58.5, 0.5) rectangle (59.5, 1.5);
  \fill[cR, fade] (59.5, 0.5) rectangle (60.5, 1.5);
  \fill[cR, fade] (60.5, 0.5) rectangle (61.5, 1.5);
  \fill[cR, fade] (61.5, 0.5) rectangle (62.5, 1.5);
  \fill[cR, fade] (62.5, 0.5) rectangle (63.5, 1.5);
  \fill[cR, fade] (63.5, 0.5) rectangle (64.5, 1.5);
  \fill[cR, fade] (64.5, 0.5) rectangle (65.5, 1.5);
  \fill[cR, fade] (65.5, 0.5) rectangle (66.5, 1.5);
  \fill[cR, fade] (66.5, 0.5) rectangle (67.5, 1.5);
  \fill[cR, fade] (67.5, 0.5) rectangle (68.5, 1.5);
  \fill[cR, fade] (68.5, 0.5) rectangle (69.5, 1.5);
  \fill[cR, fade] (69.5, 0.5) rectangle (70.5, 1.5);
  \fill[cR, fade] (70.5, 0.5) rectangle (71.5, 1.5);
  \fill[cR, fade] (71.5, 0.5) rectangle (72.5, 1.5);
  \fill[cR, fade] (72.5, 0.5) rectangle (73.5, 1.5);
  \fill[cR, fade] (73.5, 0.5) rectangle (74.5, 1.5);
  \fill[cR, fade] (74.5, 0.5) rectangle (75.5, 1.5);
  \fill[cR, fade] (75.5, 0.5) rectangle (76.5, 1.5);
  \fill[cL, fade] (0.5, 1.5) rectangle (1.5, 2.5);
  \fill[cL, fade] (1.5, 1.5) rectangle (2.5, 2.5);
  \fill[cN, fade] (2.5, 1.5) rectangle (3.5, 2.5);
  \fill[cN, fade] (3.5, 1.5) rectangle (4.5, 2.5);
  \fill[cN, fade] (4.5, 1.5) rectangle (5.5, 2.5);
  \fill[cN, fade] (5.5, 1.5) rectangle (6.5, 2.5);
  \fill[cN, fade] (6.5, 1.5) rectangle (7.5, 2.5);
  \fill[cR, fade] (7.5, 1.5) rectangle (8.5, 2.5);
  \fill[cR, fade] (8.5, 1.5) rectangle (9.5, 2.5);
  \fill[cR, fade] (9.5, 1.5) rectangle (10.5, 2.5);
  \fill[cR, fade] (10.5, 1.5) rectangle (11.5, 2.5);
  \fill[cR, fade] (11.5, 1.5) rectangle (12.5, 2.5);
  \fill[cR, fade] (12.5, 1.5) rectangle (13.5, 2.5);
  \fill[cR, fade] (13.5, 1.5) rectangle (14.5, 2.5);
  \fill[cR, fade] (14.5, 1.5) rectangle (15.5, 2.5);
  \fill[cR, fade] (15.5, 1.5) rectangle (16.5, 2.5);
  \fill[cR, fade] (16.5, 1.5) rectangle (17.5, 2.5);
  \fill[cR, fade] (17.5, 1.5) rectangle (18.5, 2.5);
  \fill[cR, fade] (18.5, 1.5) rectangle (19.5, 2.5);
  \fill[cR, fade] (19.5, 1.5) rectangle (20.5, 2.5);
  \fill[cR, fade] (20.5, 1.5) rectangle (21.5, 2.5);
  \fill[cR, fade] (21.5, 1.5) rectangle (22.5, 2.5);
  \fill[cR, fade] (22.5, 1.5) rectangle (23.5, 2.5);
  \fill[cR, fade] (23.5, 1.5) rectangle (24.5, 2.5);
  \fill[cR, fade] (24.5, 1.5) rectangle (25.5, 2.5);
  \fill[cR, fade] (25.5, 1.5) rectangle (26.5, 2.5);
  \fill[cR, fade] (26.5, 1.5) rectangle (27.5, 2.5);
  \fill[cR, fade] (27.5, 1.5) rectangle (28.5, 2.5);
  \fill[cR, fade] (28.5, 1.5) rectangle (29.5, 2.5);
  \fill[cR, fade] (29.5, 1.5) rectangle (30.5, 2.5);
  \fill[cR, fade] (30.5, 1.5) rectangle (31.5, 2.5);
  \fill[cR, fade] (31.5, 1.5) rectangle (32.5, 2.5);
  \fill[cR, fade] (32.5, 1.5) rectangle (33.5, 2.5);
  \fill[cR, fade] (33.5, 1.5) rectangle (34.5, 2.5);
  \fill[cR, fade] (34.5, 1.5) rectangle (35.5, 2.5);
  \fill[cR, fade] (35.5, 1.5) rectangle (36.5, 2.5);
  \fill[cR, fade] (36.5, 1.5) rectangle (37.5, 2.5);
  \fill[cR, fade] (37.5, 1.5) rectangle (38.5, 2.5);
  \fill[cR, fade] (38.5, 1.5) rectangle (39.5, 2.5);
  \fill[cR, fade] (39.5, 1.5) rectangle (40.5, 2.5);
  \fill[cR, fade] (40.5, 1.5) rectangle (41.5, 2.5);
  \fill[cR, fade] (41.5, 1.5) rectangle (42.5, 2.5);
  \fill[cR, fade] (42.5, 1.5) rectangle (43.5, 2.5);
  \fill[cR, fade] (43.5, 1.5) rectangle (44.5, 2.5);
  \fill[cR, fade] (44.5, 1.5) rectangle (45.5, 2.5);
  \fill[cR, fade] (45.5, 1.5) rectangle (46.5, 2.5);
  \fill[cR, fade] (46.5, 1.5) rectangle (47.5, 2.5);
  \fill[cR, fade] (47.5, 1.5) rectangle (48.5, 2.5);
  \fill[cR, fade] (48.5, 1.5) rectangle (49.5, 2.5);
  \fill[cR, fade] (49.5, 1.5) rectangle (50.5, 2.5);
  \fill[cR, fade] (50.5, 1.5) rectangle (51.5, 2.5);
  \fill[cR, fade] (51.5, 1.5) rectangle (52.5, 2.5);
  \fill[cR, fade] (52.5, 1.5) rectangle (53.5, 2.5);
  \fill[cR, fade] (53.5, 1.5) rectangle (54.5, 2.5);
  \fill[cR, fade] (54.5, 1.5) rectangle (55.5, 2.5);
  \fill[cR, fade] (55.5, 1.5) rectangle (56.5, 2.5);
  \fill[cR, fade] (56.5, 1.5) rectangle (57.5, 2.5);
  \fill[cR, fade] (57.5, 1.5) rectangle (58.5, 2.5);
  \fill[cR, fade] (58.5, 1.5) rectangle (59.5, 2.5);
  \fill[cR, fade] (59.5, 1.5) rectangle (60.5, 2.5);
  \fill[cR, fade] (60.5, 1.5) rectangle (61.5, 2.5);
  \fill[cR, fade] (61.5, 1.5) rectangle (62.5, 2.5);
  \fill[cR, fade] (62.5, 1.5) rectangle (63.5, 2.5);
  \fill[cR, fade] (63.5, 1.5) rectangle (64.5, 2.5);
  \fill[cR, fade] (64.5, 1.5) rectangle (65.5, 2.5);
  \fill[cR, fade] (65.5, 1.5) rectangle (66.5, 2.5);
  \fill[cR, fade] (66.5, 1.5) rectangle (67.5, 2.5);
  \fill[cR, fade] (67.5, 1.5) rectangle (68.5, 2.5);
  \fill[cR, fade] (68.5, 1.5) rectangle (69.5, 2.5);
  \fill[cR, fade] (69.5, 1.5) rectangle (70.5, 2.5);
  \fill[cR, fade] (70.5, 1.5) rectangle (71.5, 2.5);
  \fill[cR, fade] (71.5, 1.5) rectangle (72.5, 2.5);
  \fill[cR, fade] (72.5, 1.5) rectangle (73.5, 2.5);
  \fill[cR, fade] (73.5, 1.5) rectangle (74.5, 2.5);
  \fill[cR, fade] (74.5, 1.5) rectangle (75.5, 2.5);
  \fill[cR, fade] (75.5, 1.5) rectangle (76.5, 2.5);
  \fill[cL, fade] (0.5, 2.5) rectangle (1.5, 3.5);
  \fill[cL, fade] (1.5, 2.5) rectangle (2.5, 3.5);
  \fill[cL, fade] (2.5, 2.5) rectangle (3.5, 3.5);
  \fill[cN, fade] (3.5, 2.5) rectangle (4.5, 3.5);
  \fill[cN, fade] (4.5, 2.5) rectangle (5.5, 3.5);
  \fill[cN, fade] (5.5, 2.5) rectangle (6.5, 3.5);
  \fill[cN, fade] (6.5, 2.5) rectangle (7.5, 3.5);
  \fill[cN, fade] (7.5, 2.5) rectangle (8.5, 3.5);
  \fill[cR, fade] (8.5, 2.5) rectangle (9.5, 3.5);
  \fill[cR, fade] (9.5, 2.5) rectangle (10.5, 3.5);
  \fill[cR, fade] (10.5, 2.5) rectangle (11.5, 3.5);
  \fill[cP, fade] (11.5, 2.5) rectangle (12.5, 3.5);
  \fill[cN, fade] (12.5, 2.5) rectangle (13.5, 3.5);
  \fill[cN, fade] (13.5, 2.5) rectangle (14.5, 3.5);
  \fill[cN, fade] (14.5, 2.5) rectangle (15.5, 3.5);
  \fill[cN, fade] (15.5, 2.5) rectangle (16.5, 3.5);
  \fill[cN, fade] (16.5, 2.5) rectangle (17.5, 3.5);
  \fill[cR, fade] (17.5, 2.5) rectangle (18.5, 3.5);
  \fill[cR, fade] (18.5, 2.5) rectangle (19.5, 3.5);
  \fill[cR, fade] (19.5, 2.5) rectangle (20.5, 3.5);
  \fill[cR, fade] (20.5, 2.5) rectangle (21.5, 3.5);
  \fill[cR, fade] (21.5, 2.5) rectangle (22.5, 3.5);
  \fill[cR, fade] (22.5, 2.5) rectangle (23.5, 3.5);
  \fill[cR, fade] (23.5, 2.5) rectangle (24.5, 3.5);
  \fill[cR, fade] (24.5, 2.5) rectangle (25.5, 3.5);
  \fill[cR, fade] (25.5, 2.5) rectangle (26.5, 3.5);
  \fill[cR, fade] (26.5, 2.5) rectangle (27.5, 3.5);
  \fill[cR, fade] (27.5, 2.5) rectangle (28.5, 3.5);
  \fill[cR, fade] (28.5, 2.5) rectangle (29.5, 3.5);
  \fill[cR, fade] (29.5, 2.5) rectangle (30.5, 3.5);
  \fill[cR, fade] (30.5, 2.5) rectangle (31.5, 3.5);
  \fill[cR, fade] (31.5, 2.5) rectangle (32.5, 3.5);
  \fill[cR, fade] (32.5, 2.5) rectangle (33.5, 3.5);
  \fill[cR, fade] (33.5, 2.5) rectangle (34.5, 3.5);
  \fill[cR, fade] (34.5, 2.5) rectangle (35.5, 3.5);
  \fill[cR, fade] (35.5, 2.5) rectangle (36.5, 3.5);
  \fill[cR, fade] (36.5, 2.5) rectangle (37.5, 3.5);
  \fill[cR, fade] (37.5, 2.5) rectangle (38.5, 3.5);
  \fill[cR, fade] (38.5, 2.5) rectangle (39.5, 3.5);
  \fill[cR, fade] (39.5, 2.5) rectangle (40.5, 3.5);
  \fill[cR, fade] (40.5, 2.5) rectangle (41.5, 3.5);
  \fill[cR, fade] (41.5, 2.5) rectangle (42.5, 3.5);
  \fill[cR, fade] (42.5, 2.5) rectangle (43.5, 3.5);
  \fill[cR, fade] (43.5, 2.5) rectangle (44.5, 3.5);
  \fill[cR, fade] (44.5, 2.5) rectangle (45.5, 3.5);
  \fill[cR, fade] (45.5, 2.5) rectangle (46.5, 3.5);
  \fill[cR, fade] (46.5, 2.5) rectangle (47.5, 3.5);
  \fill[cR, fade] (47.5, 2.5) rectangle (48.5, 3.5);
  \fill[cR, fade] (48.5, 2.5) rectangle (49.5, 3.5);
  \fill[cR, fade] (49.5, 2.5) rectangle (50.5, 3.5);
  \fill[cR, fade] (50.5, 2.5) rectangle (51.5, 3.5);
  \fill[cR, fade] (51.5, 2.5) rectangle (52.5, 3.5);
  \fill[cR, fade] (52.5, 2.5) rectangle (53.5, 3.5);
  \fill[cR, fade] (53.5, 2.5) rectangle (54.5, 3.5);
  \fill[cR, fade] (54.5, 2.5) rectangle (55.5, 3.5);
  \fill[cR, fade] (55.5, 2.5) rectangle (56.5, 3.5);
  \fill[cR, fade] (56.5, 2.5) rectangle (57.5, 3.5);
  \fill[cR, fade] (57.5, 2.5) rectangle (58.5, 3.5);
  \fill[cR, fade] (58.5, 2.5) rectangle (59.5, 3.5);
  \fill[cR, fade] (59.5, 2.5) rectangle (60.5, 3.5);
  \fill[cR, fade] (60.5, 2.5) rectangle (61.5, 3.5);
  \fill[cR, fade] (61.5, 2.5) rectangle (62.5, 3.5);
  \fill[cR, fade] (62.5, 2.5) rectangle (63.5, 3.5);
  \fill[cR, fade] (63.5, 2.5) rectangle (64.5, 3.5);
  \fill[cR, fade] (64.5, 2.5) rectangle (65.5, 3.5);
  \fill[cR, fade] (65.5, 2.5) rectangle (66.5, 3.5);
  \fill[cR, fade] (66.5, 2.5) rectangle (67.5, 3.5);
  \fill[cR, fade] (67.5, 2.5) rectangle (68.5, 3.5);
  \fill[cR, fade] (68.5, 2.5) rectangle (69.5, 3.5);
  \fill[cR, fade] (69.5, 2.5) rectangle (70.5, 3.5);
  \fill[cR, fade] (70.5, 2.5) rectangle (71.5, 3.5);
  \fill[cR, fade] (71.5, 2.5) rectangle (72.5, 3.5);
  \fill[cR, fade] (72.5, 2.5) rectangle (73.5, 3.5);
  \fill[cR, fade] (73.5, 2.5) rectangle (74.5, 3.5);
  \fill[cR, fade] (74.5, 2.5) rectangle (75.5, 3.5);
  \fill[cR, fade] (75.5, 2.5) rectangle (76.5, 3.5);
  \fill[cL, fade] (0.5, 3.5) rectangle (1.5, 4.5);
  \fill[cL, fade] (1.5, 3.5) rectangle (2.5, 4.5);
  \fill[cL, fade] (2.5, 3.5) rectangle (3.5, 4.5);
  \fill[cL, fade] (3.5, 3.5) rectangle (4.5, 4.5);
  \fill[cN, fade] (4.5, 3.5) rectangle (5.5, 4.5);
  \fill[cN, fade] (5.5, 3.5) rectangle (6.5, 4.5);
  \fill[cN, fade] (6.5, 3.5) rectangle (7.5, 4.5);
  \fill[cN, fade] (7.5, 3.5) rectangle (8.5, 4.5);
  \fill[cN, fade] (8.5, 3.5) rectangle (9.5, 4.5);
  \fill[cR, fade] (9.5, 3.5) rectangle (10.5, 4.5);
  \fill[cR, fade] (10.5, 3.5) rectangle (11.5, 4.5);
  \fill[cP, fade] (11.5, 3.5) rectangle (12.5, 4.5);
  \fill[cL, fade] (12.5, 3.5) rectangle (13.5, 4.5);
  \fill[cL, fade] (13.5, 3.5) rectangle (14.5, 4.5);
  \fill[cN, fade] (14.5, 3.5) rectangle (15.5, 4.5);
  \fill[cN, fade] (15.5, 3.5) rectangle (16.5, 4.5);
  \fill[cN, fade] (16.5, 3.5) rectangle (17.5, 4.5);
  \fill[cN, fade] (17.5, 3.5) rectangle (18.5, 4.5);
  \fill[cN, fade] (18.5, 3.5) rectangle (19.5, 4.5);
  \fill[cR, fade] (19.5, 3.5) rectangle (20.5, 4.5);
  \fill[cR, fade] (20.5, 3.5) rectangle (21.5, 4.5);
  \fill[cR, fade] (21.5, 3.5) rectangle (22.5, 4.5);
  \fill[cR, fade] (22.5, 3.5) rectangle (23.5, 4.5);
  \fill[cP, fade] (23.5, 3.5) rectangle (24.5, 4.5);
  \fill[cN, fade] (24.5, 3.5) rectangle (25.5, 4.5);
  \fill[cN, fade] (25.5, 3.5) rectangle (26.5, 4.5);
  \fill[cN, fade] (26.5, 3.5) rectangle (27.5, 4.5);
  \fill[cN, fade] (27.5, 3.5) rectangle (28.5, 4.5);
  \fill[cN, fade] (28.5, 3.5) rectangle (29.5, 4.5);
  \fill[cR, fade] (29.5, 3.5) rectangle (30.5, 4.5);
  \fill[cR, fade] (30.5, 3.5) rectangle (31.5, 4.5);
  \fill[cR, fade] (31.5, 3.5) rectangle (32.5, 4.5);
  \fill[cR, fade] (32.5, 3.5) rectangle (33.5, 4.5);
  \fill[cR, fade] (33.5, 3.5) rectangle (34.5, 4.5);
  \fill[cR, fade] (34.5, 3.5) rectangle (35.5, 4.5);
  \fill[cR, fade] (35.5, 3.5) rectangle (36.5, 4.5);
  \fill[cR, fade] (36.5, 3.5) rectangle (37.5, 4.5);
  \fill[cR, fade] (37.5, 3.5) rectangle (38.5, 4.5);
  \fill[cR, fade] (38.5, 3.5) rectangle (39.5, 4.5);
  \fill[cR, fade] (39.5, 3.5) rectangle (40.5, 4.5);
  \fill[cR, fade] (40.5, 3.5) rectangle (41.5, 4.5);
  \fill[cR, fade] (41.5, 3.5) rectangle (42.5, 4.5);
  \fill[cR, fade] (42.5, 3.5) rectangle (43.5, 4.5);
  \fill[cR, fade] (43.5, 3.5) rectangle (44.5, 4.5);
  \fill[cR, fade] (44.5, 3.5) rectangle (45.5, 4.5);
  \fill[cR, fade] (45.5, 3.5) rectangle (46.5, 4.5);
  \fill[cR, fade] (46.5, 3.5) rectangle (47.5, 4.5);
  \fill[cR, fade] (47.5, 3.5) rectangle (48.5, 4.5);
  \fill[cR, fade] (48.5, 3.5) rectangle (49.5, 4.5);
  \fill[cR, fade] (49.5, 3.5) rectangle (50.5, 4.5);
  \fill[cR, fade] (50.5, 3.5) rectangle (51.5, 4.5);
  \fill[cR, fade] (51.5, 3.5) rectangle (52.5, 4.5);
  \fill[cR, fade] (52.5, 3.5) rectangle (53.5, 4.5);
  \fill[cR, fade] (53.5, 3.5) rectangle (54.5, 4.5);
  \fill[cR, fade] (54.5, 3.5) rectangle (55.5, 4.5);
  \fill[cR, fade] (55.5, 3.5) rectangle (56.5, 4.5);
  \fill[cR, fade] (56.5, 3.5) rectangle (57.5, 4.5);
  \fill[cR, fade] (57.5, 3.5) rectangle (58.5, 4.5);
  \fill[cR, fade] (58.5, 3.5) rectangle (59.5, 4.5);
  \fill[cR, fade] (59.5, 3.5) rectangle (60.5, 4.5);
  \fill[cR, fade] (60.5, 3.5) rectangle (61.5, 4.5);
  \fill[cR, fade] (61.5, 3.5) rectangle (62.5, 4.5);
  \fill[cR, fade] (62.5, 3.5) rectangle (63.5, 4.5);
  \fill[cR, fade] (63.5, 3.5) rectangle (64.5, 4.5);
  \fill[cR, fade] (64.5, 3.5) rectangle (65.5, 4.5);
  \fill[cR, fade] (65.5, 3.5) rectangle (66.5, 4.5);
  \fill[cR, fade] (66.5, 3.5) rectangle (67.5, 4.5);
  \fill[cR, fade] (67.5, 3.5) rectangle (68.5, 4.5);
  \fill[cR, fade] (68.5, 3.5) rectangle (69.5, 4.5);
  \fill[cR, fade] (69.5, 3.5) rectangle (70.5, 4.5);
  \fill[cR, fade] (70.5, 3.5) rectangle (71.5, 4.5);
  \fill[cR, fade] (71.5, 3.5) rectangle (72.5, 4.5);
  \fill[cR, fade] (72.5, 3.5) rectangle (73.5, 4.5);
  \fill[cR, fade] (73.5, 3.5) rectangle (74.5, 4.5);
  \fill[cR, fade] (74.5, 3.5) rectangle (75.5, 4.5);
  \fill[cR, fade] (75.5, 3.5) rectangle (76.5, 4.5);
  \fill[cL, fade] (0.5, 4.5) rectangle (1.5, 5.5);
  \fill[cL, fade] (1.5, 4.5) rectangle (2.5, 5.5);
  \fill[cL, fade] (2.5, 4.5) rectangle (3.5, 5.5);
  \fill[cL, fade] (3.5, 4.5) rectangle (4.5, 5.5);
  \fill[cL, fade] (4.5, 4.5) rectangle (5.5, 5.5);
  \fill[cN, fade] (5.5, 4.5) rectangle (6.5, 5.5);
  \fill[cN, fade] (6.5, 4.5) rectangle (7.5, 5.5);
  \fill[cN, fade] (7.5, 4.5) rectangle (8.5, 5.5);
  \fill[cN, fade] (8.5, 4.5) rectangle (9.5, 5.5);
  \fill[cN, fade] (9.5, 4.5) rectangle (10.5, 5.5);
  \fill[cR, fade] (10.5, 4.5) rectangle (11.5, 5.5);
  \fill[cP, fade] (11.5, 4.5) rectangle (12.5, 5.5);
  \fill[cL, fade] (12.5, 4.5) rectangle (13.5, 5.5);
  \fill[cL, fade] (13.5, 4.5) rectangle (14.5, 5.5);
  \fill[cL, fade] (14.5, 4.5) rectangle (15.5, 5.5);
  \fill[cL, fade] (15.5, 4.5) rectangle (16.5, 5.5);
  \fill[cN, fade] (16.5, 4.5) rectangle (17.5, 5.5);
  \fill[cN, fade] (17.5, 4.5) rectangle (18.5, 5.5);
  \fill[cN, fade] (18.5, 4.5) rectangle (19.5, 5.5);
  \fill[cN, fade] (19.5, 4.5) rectangle (20.5, 5.5);
  \fill[cN, fade] (20.5, 4.5) rectangle (21.5, 5.5);
  \fill[cR, fade] (21.5, 4.5) rectangle (22.5, 5.5);
  \fill[cR, fade] (22.5, 4.5) rectangle (23.5, 5.5);
  \fill[cP, fade] (23.5, 4.5) rectangle (24.5, 5.5);
  \fill[cL, fade] (24.5, 4.5) rectangle (25.5, 5.5);
  \fill[cL, fade] (25.5, 4.5) rectangle (26.5, 5.5);
  \fill[cL, fade] (26.5, 4.5) rectangle (27.5, 5.5);
  \fill[cN, fade] (27.5, 4.5) rectangle (28.5, 5.5);
  \fill[cN, fade] (28.5, 4.5) rectangle (29.5, 5.5);
  \fill[cN, fade] (29.5, 4.5) rectangle (30.5, 5.5);
  \fill[cN, fade] (30.5, 4.5) rectangle (31.5, 5.5);
  \fill[cN, fade] (31.5, 4.5) rectangle (32.5, 5.5);
  \fill[cR, fade] (32.5, 4.5) rectangle (33.5, 5.5);
  \fill[cR, fade] (33.5, 4.5) rectangle (34.5, 5.5);
  \fill[cR, fade] (34.5, 4.5) rectangle (35.5, 5.5);
  \fill[cP, fade] (35.5, 4.5) rectangle (36.5, 5.5);
  \fill[cL, fade] (36.5, 4.5) rectangle (37.5, 5.5);
  \fill[cL, fade] (37.5, 4.5) rectangle (38.5, 5.5);
  \fill[cN, fade] (38.5, 4.5) rectangle (39.5, 5.5);
  \fill[cN, fade] (39.5, 4.5) rectangle (40.5, 5.5);
  \fill[cN, fade] (40.5, 4.5) rectangle (41.5, 5.5);
  \fill[cN, fade] (41.5, 4.5) rectangle (42.5, 5.5);
  \fill[cN, fade] (42.5, 4.5) rectangle (43.5, 5.5);
  \fill[cR, fade] (43.5, 4.5) rectangle (44.5, 5.5);
  \fill[cR, fade] (44.5, 4.5) rectangle (45.5, 5.5);
  \fill[cR, fade] (45.5, 4.5) rectangle (46.5, 5.5);
  \fill[cR, fade] (46.5, 4.5) rectangle (47.5, 5.5);
  \fill[cP, fade] (47.5, 4.5) rectangle (48.5, 5.5);
  \fill[cL, fade] (48.5, 4.5) rectangle (49.5, 5.5);
  \fill[cN, fade] (49.5, 4.5) rectangle (50.5, 5.5);
  \fill[cN, fade] (50.5, 4.5) rectangle (51.5, 5.5);
  \fill[cN, fade] (51.5, 4.5) rectangle (52.5, 5.5);
  \fill[cN, fade] (52.5, 4.5) rectangle (53.5, 5.5);
  \fill[cN, fade] (53.5, 4.5) rectangle (54.5, 5.5);
  \fill[cR, fade] (54.5, 4.5) rectangle (55.5, 5.5);
  \fill[cR, fade] (55.5, 4.5) rectangle (56.5, 5.5);
  \fill[cR, fade] (56.5, 4.5) rectangle (57.5, 5.5);
  \fill[cR, fade] (57.5, 4.5) rectangle (58.5, 5.5);
  \fill[cR, fade] (58.5, 4.5) rectangle (59.5, 5.5);
  \fill[cP, fade] (59.5, 4.5) rectangle (60.5, 5.5);
  \fill[cN, fade] (60.5, 4.5) rectangle (61.5, 5.5);
  \fill[cN, fade] (61.5, 4.5) rectangle (62.5, 5.5);
  \fill[cN, fade] (62.5, 4.5) rectangle (63.5, 5.5);
  \fill[cN, fade] (63.5, 4.5) rectangle (64.5, 5.5);
  \fill[cN, fade] (64.5, 4.5) rectangle (65.5, 5.5);
  \fill[cR, fade] (65.5, 4.5) rectangle (66.5, 5.5);
  \fill[cR, fade] (66.5, 4.5) rectangle (67.5, 5.5);
  \fill[cR, fade] (67.5, 4.5) rectangle (68.5, 5.5);
  \fill[cR, fade] (68.5, 4.5) rectangle (69.5, 5.5);
  \fill[cR, fade] (69.5, 4.5) rectangle (70.5, 5.5);
  \fill[cR, fade] (70.5, 4.5) rectangle (71.5, 5.5);
  \fill[cR, fade] (71.5, 4.5) rectangle (72.5, 5.5);
  \fill[cR, fade] (72.5, 4.5) rectangle (73.5, 5.5);
  \fill[cR, fade] (73.5, 4.5) rectangle (74.5, 5.5);
  \fill[cR, fade] (74.5, 4.5) rectangle (75.5, 5.5);
  \fill[cR, fade] (75.5, 4.5) rectangle (76.5, 5.5);
  \fill[cL] (0.5, 5.5) rectangle (1.5, 6.5);
  \fill[cL] (1.5, 5.5) rectangle (2.5, 6.5);
  \fill[cL] (2.5, 5.5) rectangle (3.5, 6.5);
  \fill[cL] (3.5, 5.5) rectangle (4.5, 6.5);
  \fill[cL] (4.5, 5.5) rectangle (5.5, 6.5);
  \fill[cL] (5.5, 5.5) rectangle (6.5, 6.5);
  \fill[cN] (6.5, 5.5) rectangle (7.5, 6.5);
  \fill[cN] (7.5, 5.5) rectangle (8.5, 6.5);
  \fill[cN] (8.5, 5.5) rectangle (9.5, 6.5);
  \fill[cN] (9.5, 5.5) rectangle (10.5, 6.5);
  \fill[cN] (10.5, 5.5) rectangle (11.5, 6.5);
  \fill[cP] (11.5, 5.5) rectangle (12.5, 6.5);
  \fill[cL] (12.5, 5.5) rectangle (13.5, 6.5);
  \fill[cL] (13.5, 5.5) rectangle (14.5, 6.5);
  \fill[cL] (14.5, 5.5) rectangle (15.5, 6.5);
  \fill[cL] (15.5, 5.5) rectangle (16.5, 6.5);
  \fill[cL] (16.5, 5.5) rectangle (17.5, 6.5);
  \fill[cL] (17.5, 5.5) rectangle (18.5, 6.5);
  \fill[cN] (18.5, 5.5) rectangle (19.5, 6.5);
  \fill[cN] (19.5, 5.5) rectangle (20.5, 6.5);
  \fill[cN] (20.5, 5.5) rectangle (21.5, 6.5);
  \fill[cN] (21.5, 5.5) rectangle (22.5, 6.5);
  \fill[cN] (22.5, 5.5) rectangle (23.5, 6.5);
  \fill[cP] (23.5, 5.5) rectangle (24.5, 6.5);
  \fill[cL] (24.5, 5.5) rectangle (25.5, 6.5);
  \fill[cL] (25.5, 5.5) rectangle (26.5, 6.5);
  \fill[cL] (26.5, 5.5) rectangle (27.5, 6.5);
  \fill[cL] (27.5, 5.5) rectangle (28.5, 6.5);
  \fill[cL] (28.5, 5.5) rectangle (29.5, 6.5);
  \fill[cL] (29.5, 5.5) rectangle (30.5, 6.5);
  \fill[cN] (30.5, 5.5) rectangle (31.5, 6.5);
  \fill[cN] (31.5, 5.5) rectangle (32.5, 6.5);
  \fill[cN] (32.5, 5.5) rectangle (33.5, 6.5);
  \fill[cN] (33.5, 5.5) rectangle (34.5, 6.5);
  \fill[cN] (34.5, 5.5) rectangle (35.5, 6.5);
  \fill[cP] (35.5, 5.5) rectangle (36.5, 6.5);
  \fill[cL] (36.5, 5.5) rectangle (37.5, 6.5);
  \fill[cL] (37.5, 5.5) rectangle (38.5, 6.5);
  \fill[cL] (38.5, 5.5) rectangle (39.5, 6.5);
  \fill[cL] (39.5, 5.5) rectangle (40.5, 6.5);
  \fill[cL] (40.5, 5.5) rectangle (41.5, 6.5);
  \fill[cL] (41.5, 5.5) rectangle (42.5, 6.5);
  \fill[cN] (42.5, 5.5) rectangle (43.5, 6.5);
  \fill[cN] (43.5, 5.5) rectangle (44.5, 6.5);
  \fill[cN] (44.5, 5.5) rectangle (45.5, 6.5);
  \fill[cN] (45.5, 5.5) rectangle (46.5, 6.5);
  \fill[cN] (46.5, 5.5) rectangle (47.5, 6.5);
  \fill[cP] (47.5, 5.5) rectangle (48.5, 6.5);
  \fill[cL] (48.5, 5.5) rectangle (49.5, 6.5);
  \fill[cL] (49.5, 5.5) rectangle (50.5, 6.5);
  \fill[cL] (50.5, 5.5) rectangle (51.5, 6.5);
  \fill[cL] (51.5, 5.5) rectangle (52.5, 6.5);
  \fill[cL] (52.5, 5.5) rectangle (53.5, 6.5);
  \fill[cL] (53.5, 5.5) rectangle (54.5, 6.5);
  \fill[cN] (54.5, 5.5) rectangle (55.5, 6.5);
  \fill[cN] (55.5, 5.5) rectangle (56.5, 6.5);
  \fill[cN] (56.5, 5.5) rectangle (57.5, 6.5);
  \fill[cN] (57.5, 5.5) rectangle (58.5, 6.5);
  \fill[cN] (58.5, 5.5) rectangle (59.5, 6.5);
  \fill[cP] (59.5, 5.5) rectangle (60.5, 6.5);
  \fill[cL] (60.5, 5.5) rectangle (61.5, 6.5);
  \fill[cL] (61.5, 5.5) rectangle (62.5, 6.5);
  \fill[cL] (62.5, 5.5) rectangle (63.5, 6.5);
  \fill[cL] (63.5, 5.5) rectangle (64.5, 6.5);
  \fill[cL] (64.5, 5.5) rectangle (65.5, 6.5);
  \fill[cL] (65.5, 5.5) rectangle (66.5, 6.5);
  \fill[cN] (66.5, 5.5) rectangle (67.5, 6.5);
  \fill[cN] (67.5, 5.5) rectangle (68.5, 6.5);
  \fill[cN] (68.5, 5.5) rectangle (69.5, 6.5);
  \fill[cN] (69.5, 5.5) rectangle (70.5, 6.5);
  \fill[cN] (70.5, 5.5) rectangle (71.5, 6.5);
  \fill[cP] (71.5, 5.5) rectangle (72.5, 6.5);
  \fill[cL] (72.5, 5.5) rectangle (73.5, 6.5);
  \fill[cL] (73.5, 5.5) rectangle (74.5, 6.5);
  \fill[cL] (74.5, 5.5) rectangle (75.5, 6.5);
  \fill[cL] (75.5, 5.5) rectangle (76.5, 6.5);
  \fill[cL, fade] (0.5, 6.5) rectangle (1.5, 7.5);
  \fill[cL, fade] (1.5, 6.5) rectangle (2.5, 7.5);
  \fill[cL, fade] (2.5, 6.5) rectangle (3.5, 7.5);
  \fill[cL, fade] (3.5, 6.5) rectangle (4.5, 7.5);
  \fill[cL, fade] (4.5, 6.5) rectangle (5.5, 7.5);
  \fill[cL, fade] (5.5, 6.5) rectangle (6.5, 7.5);
  \fill[cL, fade] (6.5, 6.5) rectangle (7.5, 7.5);
  \fill[cN, fade] (7.5, 6.5) rectangle (8.5, 7.5);
  \fill[cN, fade] (8.5, 6.5) rectangle (9.5, 7.5);
  \fill[cN, fade] (9.5, 6.5) rectangle (10.5, 7.5);
  \fill[cN, fade] (10.5, 6.5) rectangle (11.5, 7.5);
  \fill[cL, fade] (11.5, 6.5) rectangle (12.5, 7.5);
  \fill[cL, fade] (12.5, 6.5) rectangle (13.5, 7.5);
  \fill[cL, fade] (13.5, 6.5) rectangle (14.5, 7.5);
  \fill[cL, fade] (14.5, 6.5) rectangle (15.5, 7.5);
  \fill[cL, fade] (15.5, 6.5) rectangle (16.5, 7.5);
  \fill[cL, fade] (16.5, 6.5) rectangle (17.5, 7.5);
  \fill[cL, fade] (17.5, 6.5) rectangle (18.5, 7.5);
  \fill[cL, fade] (18.5, 6.5) rectangle (19.5, 7.5);
  \fill[cL, fade] (19.5, 6.5) rectangle (20.5, 7.5);
  \fill[cL, fade] (20.5, 6.5) rectangle (21.5, 7.5);
  \fill[cL, fade] (21.5, 6.5) rectangle (22.5, 7.5);
  \fill[cL, fade] (22.5, 6.5) rectangle (23.5, 7.5);
  \fill[cL, fade] (23.5, 6.5) rectangle (24.5, 7.5);
  \fill[cL, fade] (24.5, 6.5) rectangle (25.5, 7.5);
  \fill[cL, fade] (25.5, 6.5) rectangle (26.5, 7.5);
  \fill[cL, fade] (26.5, 6.5) rectangle (27.5, 7.5);
  \fill[cL, fade] (27.5, 6.5) rectangle (28.5, 7.5);
  \fill[cL, fade] (28.5, 6.5) rectangle (29.5, 7.5);
  \fill[cL, fade] (29.5, 6.5) rectangle (30.5, 7.5);
  \fill[cL, fade] (30.5, 6.5) rectangle (31.5, 7.5);
  \fill[cL, fade] (31.5, 6.5) rectangle (32.5, 7.5);
  \fill[cL, fade] (32.5, 6.5) rectangle (33.5, 7.5);
  \fill[cL, fade] (33.5, 6.5) rectangle (34.5, 7.5);
  \fill[cL, fade] (34.5, 6.5) rectangle (35.5, 7.5);
  \fill[cL, fade] (35.5, 6.5) rectangle (36.5, 7.5);
  \fill[cL, fade] (36.5, 6.5) rectangle (37.5, 7.5);
  \fill[cL, fade] (37.5, 6.5) rectangle (38.5, 7.5);
  \fill[cL, fade] (38.5, 6.5) rectangle (39.5, 7.5);
  \fill[cL, fade] (39.5, 6.5) rectangle (40.5, 7.5);
  \fill[cL, fade] (40.5, 6.5) rectangle (41.5, 7.5);
  \fill[cL, fade] (41.5, 6.5) rectangle (42.5, 7.5);
  \fill[cL, fade] (42.5, 6.5) rectangle (43.5, 7.5);
  \fill[cL, fade] (43.5, 6.5) rectangle (44.5, 7.5);
  \fill[cL, fade] (44.5, 6.5) rectangle (45.5, 7.5);
  \fill[cL, fade] (45.5, 6.5) rectangle (46.5, 7.5);
  \fill[cL, fade] (46.5, 6.5) rectangle (47.5, 7.5);
  \fill[cL, fade] (47.5, 6.5) rectangle (48.5, 7.5);
  \fill[cL, fade] (48.5, 6.5) rectangle (49.5, 7.5);
  \fill[cL, fade] (49.5, 6.5) rectangle (50.5, 7.5);
  \fill[cL, fade] (50.5, 6.5) rectangle (51.5, 7.5);
  \fill[cL, fade] (51.5, 6.5) rectangle (52.5, 7.5);
  \fill[cL, fade] (52.5, 6.5) rectangle (53.5, 7.5);
  \fill[cL, fade] (53.5, 6.5) rectangle (54.5, 7.5);
  \fill[cL, fade] (54.5, 6.5) rectangle (55.5, 7.5);
  \fill[cL, fade] (55.5, 6.5) rectangle (56.5, 7.5);
  \fill[cL, fade] (56.5, 6.5) rectangle (57.5, 7.5);
  \fill[cL, fade] (57.5, 6.5) rectangle (58.5, 7.5);
  \fill[cL, fade] (58.5, 6.5) rectangle (59.5, 7.5);
  \fill[cL, fade] (59.5, 6.5) rectangle (60.5, 7.5);
  \fill[cL, fade] (60.5, 6.5) rectangle (61.5, 7.5);
  \fill[cL, fade] (61.5, 6.5) rectangle (62.5, 7.5);
  \fill[cL, fade] (62.5, 6.5) rectangle (63.5, 7.5);
  \fill[cL, fade] (63.5, 6.5) rectangle (64.5, 7.5);
  \fill[cL, fade] (64.5, 6.5) rectangle (65.5, 7.5);
  \fill[cL, fade] (65.5, 6.5) rectangle (66.5, 7.5);
  \fill[cL, fade] (66.5, 6.5) rectangle (67.5, 7.5);
  \fill[cL, fade] (67.5, 6.5) rectangle (68.5, 7.5);
  \fill[cL, fade] (68.5, 6.5) rectangle (69.5, 7.5);
  \fill[cL, fade] (69.5, 6.5) rectangle (70.5, 7.5);
  \fill[cL, fade] (70.5, 6.5) rectangle (71.5, 7.5);
  \fill[cL, fade] (71.5, 6.5) rectangle (72.5, 7.5);
  \fill[cL, fade] (72.5, 6.5) rectangle (73.5, 7.5);
  \fill[cL, fade] (73.5, 6.5) rectangle (74.5, 7.5);
  \fill[cL, fade] (74.5, 6.5) rectangle (75.5, 7.5);
  \fill[cL, fade] (75.5, 6.5) rectangle (76.5, 7.5);
  \fill[cL, fade] (0.5, 7.5) rectangle (1.5, 8.5);
  \fill[cL, fade] (1.5, 7.5) rectangle (2.5, 8.5);
  \fill[cL, fade] (2.5, 7.5) rectangle (3.5, 8.5);
  \fill[cL, fade] (3.5, 7.5) rectangle (4.5, 8.5);
  \fill[cL, fade] (4.5, 7.5) rectangle (5.5, 8.5);
  \fill[cL, fade] (5.5, 7.5) rectangle (6.5, 8.5);
  \fill[cL, fade] (6.5, 7.5) rectangle (7.5, 8.5);
  \fill[cL, fade] (7.5, 7.5) rectangle (8.5, 8.5);
  \fill[cN, fade] (8.5, 7.5) rectangle (9.5, 8.5);
  \fill[cN, fade] (9.5, 7.5) rectangle (10.5, 8.5);
  \fill[cN, fade] (10.5, 7.5) rectangle (11.5, 8.5);
  \fill[cL, fade] (11.5, 7.5) rectangle (12.5, 8.5);
  \fill[cL, fade] (12.5, 7.5) rectangle (13.5, 8.5);
  \fill[cL, fade] (13.5, 7.5) rectangle (14.5, 8.5);
  \fill[cL, fade] (14.5, 7.5) rectangle (15.5, 8.5);
  \fill[cL, fade] (15.5, 7.5) rectangle (16.5, 8.5);
  \fill[cL, fade] (16.5, 7.5) rectangle (17.5, 8.5);
  \fill[cL, fade] (17.5, 7.5) rectangle (18.5, 8.5);
  \fill[cL, fade] (18.5, 7.5) rectangle (19.5, 8.5);
  \fill[cL, fade] (19.5, 7.5) rectangle (20.5, 8.5);
  \fill[cL, fade] (20.5, 7.5) rectangle (21.5, 8.5);
  \fill[cL, fade] (21.5, 7.5) rectangle (22.5, 8.5);
  \fill[cL, fade] (22.5, 7.5) rectangle (23.5, 8.5);
  \fill[cL, fade] (23.5, 7.5) rectangle (24.5, 8.5);
  \fill[cL, fade] (24.5, 7.5) rectangle (25.5, 8.5);
  \fill[cL, fade] (25.5, 7.5) rectangle (26.5, 8.5);
  \fill[cL, fade] (26.5, 7.5) rectangle (27.5, 8.5);
  \fill[cL, fade] (27.5, 7.5) rectangle (28.5, 8.5);
  \fill[cL, fade] (28.5, 7.5) rectangle (29.5, 8.5);
  \fill[cL, fade] (29.5, 7.5) rectangle (30.5, 8.5);
  \fill[cL, fade] (30.5, 7.5) rectangle (31.5, 8.5);
  \fill[cL, fade] (31.5, 7.5) rectangle (32.5, 8.5);
  \fill[cL, fade] (32.5, 7.5) rectangle (33.5, 8.5);
  \fill[cL, fade] (33.5, 7.5) rectangle (34.5, 8.5);
  \fill[cL, fade] (34.5, 7.5) rectangle (35.5, 8.5);
  \fill[cL, fade] (35.5, 7.5) rectangle (36.5, 8.5);
  \fill[cL, fade] (36.5, 7.5) rectangle (37.5, 8.5);
  \fill[cL, fade] (37.5, 7.5) rectangle (38.5, 8.5);
  \fill[cL, fade] (38.5, 7.5) rectangle (39.5, 8.5);
  \fill[cL, fade] (39.5, 7.5) rectangle (40.5, 8.5);
  \fill[cL, fade] (40.5, 7.5) rectangle (41.5, 8.5);
  \fill[cL, fade] (41.5, 7.5) rectangle (42.5, 8.5);
  \fill[cL, fade] (42.5, 7.5) rectangle (43.5, 8.5);
  \fill[cL, fade] (43.5, 7.5) rectangle (44.5, 8.5);
  \fill[cL, fade] (44.5, 7.5) rectangle (45.5, 8.5);
  \fill[cL, fade] (45.5, 7.5) rectangle (46.5, 8.5);
  \fill[cL, fade] (46.5, 7.5) rectangle (47.5, 8.5);
  \fill[cL, fade] (47.5, 7.5) rectangle (48.5, 8.5);
  \fill[cL, fade] (48.5, 7.5) rectangle (49.5, 8.5);
  \fill[cL, fade] (49.5, 7.5) rectangle (50.5, 8.5);
  \fill[cL, fade] (50.5, 7.5) rectangle (51.5, 8.5);
  \fill[cL, fade] (51.5, 7.5) rectangle (52.5, 8.5);
  \fill[cL, fade] (52.5, 7.5) rectangle (53.5, 8.5);
  \fill[cL, fade] (53.5, 7.5) rectangle (54.5, 8.5);
  \fill[cL, fade] (54.5, 7.5) rectangle (55.5, 8.5);
  \fill[cL, fade] (55.5, 7.5) rectangle (56.5, 8.5);
  \fill[cL, fade] (56.5, 7.5) rectangle (57.5, 8.5);
  \fill[cL, fade] (57.5, 7.5) rectangle (58.5, 8.5);
  \fill[cL, fade] (58.5, 7.5) rectangle (59.5, 8.5);
  \fill[cL, fade] (59.5, 7.5) rectangle (60.5, 8.5);
  \fill[cL, fade] (60.5, 7.5) rectangle (61.5, 8.5);
  \fill[cL, fade] (61.5, 7.5) rectangle (62.5, 8.5);
  \fill[cL, fade] (62.5, 7.5) rectangle (63.5, 8.5);
  \fill[cL, fade] (63.5, 7.5) rectangle (64.5, 8.5);
  \fill[cL, fade] (64.5, 7.5) rectangle (65.5, 8.5);
  \fill[cL, fade] (65.5, 7.5) rectangle (66.5, 8.5);
  \fill[cL, fade] (66.5, 7.5) rectangle (67.5, 8.5);
  \fill[cL, fade] (67.5, 7.5) rectangle (68.5, 8.5);
  \fill[cL, fade] (68.5, 7.5) rectangle (69.5, 8.5);
  \fill[cL, fade] (69.5, 7.5) rectangle (70.5, 8.5);
  \fill[cL, fade] (70.5, 7.5) rectangle (71.5, 8.5);
  \fill[cL, fade] (71.5, 7.5) rectangle (72.5, 8.5);
  \fill[cL, fade] (72.5, 7.5) rectangle (73.5, 8.5);
  \fill[cL, fade] (73.5, 7.5) rectangle (74.5, 8.5);
  \fill[cL, fade] (74.5, 7.5) rectangle (75.5, 8.5);
  \fill[cL, fade] (75.5, 7.5) rectangle (76.5, 8.5);
  \fill[cL, fade] (0.5, 8.5) rectangle (1.5, 9.5);
  \fill[cL, fade] (1.5, 8.5) rectangle (2.5, 9.5);
  \fill[cL, fade] (2.5, 8.5) rectangle (3.5, 9.5);
  \fill[cL, fade] (3.5, 8.5) rectangle (4.5, 9.5);
  \fill[cL, fade] (4.5, 8.5) rectangle (5.5, 9.5);
  \fill[cL, fade] (5.5, 8.5) rectangle (6.5, 9.5);
  \fill[cL, fade] (6.5, 8.5) rectangle (7.5, 9.5);
  \fill[cL, fade] (7.5, 8.5) rectangle (8.5, 9.5);
  \fill[cL, fade] (8.5, 8.5) rectangle (9.5, 9.5);
  \fill[cN, fade] (9.5, 8.5) rectangle (10.5, 9.5);
  \fill[cN, fade] (10.5, 8.5) rectangle (11.5, 9.5);
  \fill[cL, fade] (11.5, 8.5) rectangle (12.5, 9.5);
  \fill[cL, fade] (12.5, 8.5) rectangle (13.5, 9.5);
  \fill[cL, fade] (13.5, 8.5) rectangle (14.5, 9.5);
  \fill[cL, fade] (14.5, 8.5) rectangle (15.5, 9.5);
  \fill[cL, fade] (15.5, 8.5) rectangle (16.5, 9.5);
  \fill[cL, fade] (16.5, 8.5) rectangle (17.5, 9.5);
  \fill[cL, fade] (17.5, 8.5) rectangle (18.5, 9.5);
  \fill[cL, fade] (18.5, 8.5) rectangle (19.5, 9.5);
  \fill[cL, fade] (19.5, 8.5) rectangle (20.5, 9.5);
  \fill[cL, fade] (20.5, 8.5) rectangle (21.5, 9.5);
  \fill[cL, fade] (21.5, 8.5) rectangle (22.5, 9.5);
  \fill[cL, fade] (22.5, 8.5) rectangle (23.5, 9.5);
  \fill[cL, fade] (23.5, 8.5) rectangle (24.5, 9.5);
  \fill[cL, fade] (24.5, 8.5) rectangle (25.5, 9.5);
  \fill[cL, fade] (25.5, 8.5) rectangle (26.5, 9.5);
  \fill[cL, fade] (26.5, 8.5) rectangle (27.5, 9.5);
  \fill[cL, fade] (27.5, 8.5) rectangle (28.5, 9.5);
  \fill[cL, fade] (28.5, 8.5) rectangle (29.5, 9.5);
  \fill[cL, fade] (29.5, 8.5) rectangle (30.5, 9.5);
  \fill[cL, fade] (30.5, 8.5) rectangle (31.5, 9.5);
  \fill[cL, fade] (31.5, 8.5) rectangle (32.5, 9.5);
  \fill[cL, fade] (32.5, 8.5) rectangle (33.5, 9.5);
  \fill[cL, fade] (33.5, 8.5) rectangle (34.5, 9.5);
  \fill[cL, fade] (34.5, 8.5) rectangle (35.5, 9.5);
  \fill[cL, fade] (35.5, 8.5) rectangle (36.5, 9.5);
  \fill[cL, fade] (36.5, 8.5) rectangle (37.5, 9.5);
  \fill[cL, fade] (37.5, 8.5) rectangle (38.5, 9.5);
  \fill[cL, fade] (38.5, 8.5) rectangle (39.5, 9.5);
  \fill[cL, fade] (39.5, 8.5) rectangle (40.5, 9.5);
  \fill[cL, fade] (40.5, 8.5) rectangle (41.5, 9.5);
  \fill[cL, fade] (41.5, 8.5) rectangle (42.5, 9.5);
  \fill[cL, fade] (42.5, 8.5) rectangle (43.5, 9.5);
  \fill[cL, fade] (43.5, 8.5) rectangle (44.5, 9.5);
  \fill[cL, fade] (44.5, 8.5) rectangle (45.5, 9.5);
  \fill[cL, fade] (45.5, 8.5) rectangle (46.5, 9.5);
  \fill[cL, fade] (46.5, 8.5) rectangle (47.5, 9.5);
  \fill[cL, fade] (47.5, 8.5) rectangle (48.5, 9.5);
  \fill[cL, fade] (48.5, 8.5) rectangle (49.5, 9.5);
  \fill[cL, fade] (49.5, 8.5) rectangle (50.5, 9.5);
  \fill[cL, fade] (50.5, 8.5) rectangle (51.5, 9.5);
  \fill[cL, fade] (51.5, 8.5) rectangle (52.5, 9.5);
  \fill[cL, fade] (52.5, 8.5) rectangle (53.5, 9.5);
  \fill[cL, fade] (53.5, 8.5) rectangle (54.5, 9.5);
  \fill[cL, fade] (54.5, 8.5) rectangle (55.5, 9.5);
  \fill[cL, fade] (55.5, 8.5) rectangle (56.5, 9.5);
  \fill[cL, fade] (56.5, 8.5) rectangle (57.5, 9.5);
  \fill[cL, fade] (57.5, 8.5) rectangle (58.5, 9.5);
  \fill[cL, fade] (58.5, 8.5) rectangle (59.5, 9.5);
  \fill[cL, fade] (59.5, 8.5) rectangle (60.5, 9.5);
  \fill[cL, fade] (60.5, 8.5) rectangle (61.5, 9.5);
  \fill[cL, fade] (61.5, 8.5) rectangle (62.5, 9.5);
  \fill[cL, fade] (62.5, 8.5) rectangle (63.5, 9.5);
  \fill[cL, fade] (63.5, 8.5) rectangle (64.5, 9.5);
  \fill[cL, fade] (64.5, 8.5) rectangle (65.5, 9.5);
  \fill[cL, fade] (65.5, 8.5) rectangle (66.5, 9.5);
  \fill[cL, fade] (66.5, 8.5) rectangle (67.5, 9.5);
  \fill[cL, fade] (67.5, 8.5) rectangle (68.5, 9.5);
  \fill[cL, fade] (68.5, 8.5) rectangle (69.5, 9.5);
  \fill[cL, fade] (69.5, 8.5) rectangle (70.5, 9.5);
  \fill[cL, fade] (70.5, 8.5) rectangle (71.5, 9.5);
  \fill[cL, fade] (71.5, 8.5) rectangle (72.5, 9.5);
  \fill[cL, fade] (72.5, 8.5) rectangle (73.5, 9.5);
  \fill[cL, fade] (73.5, 8.5) rectangle (74.5, 9.5);
  \fill[cL, fade] (74.5, 8.5) rectangle (75.5, 9.5);
  \fill[cL, fade] (75.5, 8.5) rectangle (76.5, 9.5);
  \fill[cL, fade] (0.5, 9.5) rectangle (1.5, 10.5);
  \fill[cL, fade] (1.5, 9.5) rectangle (2.5, 10.5);
  \fill[cL, fade] (2.5, 9.5) rectangle (3.5, 10.5);
  \fill[cL, fade] (3.5, 9.5) rectangle (4.5, 10.5);
  \fill[cL, fade] (4.5, 9.5) rectangle (5.5, 10.5);
  \fill[cL, fade] (5.5, 9.5) rectangle (6.5, 10.5);
  \fill[cL, fade] (6.5, 9.5) rectangle (7.5, 10.5);
  \fill[cL, fade] (7.5, 9.5) rectangle (8.5, 10.5);
  \fill[cL, fade] (8.5, 9.5) rectangle (9.5, 10.5);
  \fill[cL, fade] (9.5, 9.5) rectangle (10.5, 10.5);
  \fill[cN, fade] (10.5, 9.5) rectangle (11.5, 10.5);
  \fill[cL, fade] (11.5, 9.5) rectangle (12.5, 10.5);
  \fill[cL, fade] (12.5, 9.5) rectangle (13.5, 10.5);
  \fill[cL, fade] (13.5, 9.5) rectangle (14.5, 10.5);
  \fill[cL, fade] (14.5, 9.5) rectangle (15.5, 10.5);
  \fill[cL, fade] (15.5, 9.5) rectangle (16.5, 10.5);
  \fill[cL, fade] (16.5, 9.5) rectangle (17.5, 10.5);
  \fill[cL, fade] (17.5, 9.5) rectangle (18.5, 10.5);
  \fill[cL, fade] (18.5, 9.5) rectangle (19.5, 10.5);
  \fill[cL, fade] (19.5, 9.5) rectangle (20.5, 10.5);
  \fill[cL, fade] (20.5, 9.5) rectangle (21.5, 10.5);
  \fill[cL, fade] (21.5, 9.5) rectangle (22.5, 10.5);
  \fill[cL, fade] (22.5, 9.5) rectangle (23.5, 10.5);
  \fill[cL, fade] (23.5, 9.5) rectangle (24.5, 10.5);
  \fill[cL, fade] (24.5, 9.5) rectangle (25.5, 10.5);
  \fill[cL, fade] (25.5, 9.5) rectangle (26.5, 10.5);
  \fill[cL, fade] (26.5, 9.5) rectangle (27.5, 10.5);
  \fill[cL, fade] (27.5, 9.5) rectangle (28.5, 10.5);
  \fill[cL, fade] (28.5, 9.5) rectangle (29.5, 10.5);
  \fill[cL, fade] (29.5, 9.5) rectangle (30.5, 10.5);
  \fill[cL, fade] (30.5, 9.5) rectangle (31.5, 10.5);
  \fill[cL, fade] (31.5, 9.5) rectangle (32.5, 10.5);
  \fill[cL, fade] (32.5, 9.5) rectangle (33.5, 10.5);
  \fill[cL, fade] (33.5, 9.5) rectangle (34.5, 10.5);
  \fill[cL, fade] (34.5, 9.5) rectangle (35.5, 10.5);
  \fill[cL, fade] (35.5, 9.5) rectangle (36.5, 10.5);
  \fill[cL, fade] (36.5, 9.5) rectangle (37.5, 10.5);
  \fill[cL, fade] (37.5, 9.5) rectangle (38.5, 10.5);
  \fill[cL, fade] (38.5, 9.5) rectangle (39.5, 10.5);
  \fill[cL, fade] (39.5, 9.5) rectangle (40.5, 10.5);
  \fill[cL, fade] (40.5, 9.5) rectangle (41.5, 10.5);
  \fill[cL, fade] (41.5, 9.5) rectangle (42.5, 10.5);
  \fill[cL, fade] (42.5, 9.5) rectangle (43.5, 10.5);
  \fill[cL, fade] (43.5, 9.5) rectangle (44.5, 10.5);
  \fill[cL, fade] (44.5, 9.5) rectangle (45.5, 10.5);
  \fill[cL, fade] (45.5, 9.5) rectangle (46.5, 10.5);
  \fill[cL, fade] (46.5, 9.5) rectangle (47.5, 10.5);
  \fill[cL, fade] (47.5, 9.5) rectangle (48.5, 10.5);
  \fill[cL, fade] (48.5, 9.5) rectangle (49.5, 10.5);
  \fill[cL, fade] (49.5, 9.5) rectangle (50.5, 10.5);
  \fill[cL, fade] (50.5, 9.5) rectangle (51.5, 10.5);
  \fill[cL, fade] (51.5, 9.5) rectangle (52.5, 10.5);
  \fill[cL, fade] (52.5, 9.5) rectangle (53.5, 10.5);
  \fill[cL, fade] (53.5, 9.5) rectangle (54.5, 10.5);
  \fill[cL, fade] (54.5, 9.5) rectangle (55.5, 10.5);
  \fill[cL, fade] (55.5, 9.5) rectangle (56.5, 10.5);
  \fill[cL, fade] (56.5, 9.5) rectangle (57.5, 10.5);
  \fill[cL, fade] (57.5, 9.5) rectangle (58.5, 10.5);
  \fill[cL, fade] (58.5, 9.5) rectangle (59.5, 10.5);
  \fill[cL, fade] (59.5, 9.5) rectangle (60.5, 10.5);
  \fill[cL, fade] (60.5, 9.5) rectangle (61.5, 10.5);
  \fill[cL, fade] (61.5, 9.5) rectangle (62.5, 10.5);
  \fill[cL, fade] (62.5, 9.5) rectangle (63.5, 10.5);
  \fill[cL, fade] (63.5, 9.5) rectangle (64.5, 10.5);
  \fill[cL, fade] (64.5, 9.5) rectangle (65.5, 10.5);
  \fill[cL, fade] (65.5, 9.5) rectangle (66.5, 10.5);
  \fill[cL, fade] (66.5, 9.5) rectangle (67.5, 10.5);
  \fill[cL, fade] (67.5, 9.5) rectangle (68.5, 10.5);
  \fill[cL, fade] (68.5, 9.5) rectangle (69.5, 10.5);
  \fill[cL, fade] (69.5, 9.5) rectangle (70.5, 10.5);
  \fill[cL, fade] (70.5, 9.5) rectangle (71.5, 10.5);
  \fill[cL, fade] (71.5, 9.5) rectangle (72.5, 10.5);
  \fill[cL, fade] (72.5, 9.5) rectangle (73.5, 10.5);
  \fill[cL, fade] (73.5, 9.5) rectangle (74.5, 10.5);
  \fill[cL, fade] (74.5, 9.5) rectangle (75.5, 10.5);
  \fill[cL, fade] (75.5, 9.5) rectangle (76.5, 10.5);
  
  
  
  \foreach \n in {1, 5, 10,..., 75} \node[below, font=\tiny] at (\n, 0.5) {\n};
  \foreach \k in {2,4,...,10}{
    \ifnum\k=6
        \node[left,font=\tiny] at (0.5,\k) {\textbf{\k}};
    \else
        \node[left,font=\tiny] at (0.5,\k) {\k};
    \fi
}
  \node[below=0.4cm] at (38.5, 0.5) {\tiny Heap Size ($n$)};
  \node[rotate=90, above=1cm] at (-0.8, 0) {\tiny Truncation Level ($\tau$)};


    \draw[yellow!90,dashed, line width=0.8pt]
    (0.5,6) -- (76.5,6);
  \begin{scope}[shift={(77.5, 5.5)}]
    \draw[fill=cL, fade] (0, 0.8) rectangle (0.8, 1.6); \node[right] at (1, 1.2) {\tiny $\mathscr{L}$};
    \draw[fill=cR, fade] (0, -0.4) rectangle (0.8, 0.4); \node[right] at (1, 0) {\tiny $\mathscr{R}$};
    \draw[fill=cN, fade] (0, -1.6) rectangle (0.8, -0.8); \node[right] at (1, -1.2) {\tiny $\mathscr{N}$};
    \draw[fill=cP, fade] (0, -2.8) rectangle (0.8, -2.0); \node[right] at (1, -2.4) {\tiny$\mathscr{P}$};
  \end{scope}
\end{tikzpicture}
\caption{Outcomes of $\ts_\tau(n; 5, 11)$. Blue, red, green, and black cells represent $\mathscr{L}$, $\mathscr{R}$, $\mathscr{N}$, and $\mathscr{P}$ outcomes, respectively. The yellow dashed line corresponds to truncation level $b-a$. \iffalse The row corresponding to the truncation level $b-a$ is highlighted.\fi }
\label{fig: outcome of ts_k(n;5,11)}
\end{figure}

\begin{example}
In $H=\ts_4(n;5,11)$, the blocks are $[1,12],\;[13,24],\dots$ (refer to Figure~\ref{fig: outcome of ts_k(n;5,11)}). The first block starts with four $\lp$'s, followed by five $\np$'s, two $\rp$'s and one $\pp$.
\end{example} 

The first block contains $\tau$ consecutive $\lp$-positions. From one block to the next, the number of $\lp$-positions decreases by exactly $b-a-\tau$, while the number of $\rp$-positions increases by the same amount. 
This process continues until fewer than $b-a-\tau$  $\;\lp$-positions remain in a block. This occurs in the block with index
 \begin{align}
    \hbeta=\left \lfloor \frac{b-a}{b-a-\tau}\right\rfloor.\label{eq:betahat}
\end{align} 



Unlike the preceding blocks, the $\hbeta$-th block contains no $\pp$-position. Instead, it begins with a segment of $\lp$-positions, followed by $a$ consecutive $\np$-positions, and all remaining positions in the block are $\rp$.

When $\tau$ is close to $b-a$, the quantity $b-a-\tau$ is small, so the $\lp$-segment shrinks slowly from one block to the next and consequently, $\hbeta$ is large. Conversely, when $\tau$ is much smaller than $b-a$, the decrease is much faster and hence, $\hbeta$ is small.



For each block with index $\beta<\hbeta$, we define \begin{align}
    L_\beta&\coloneq(\beta-1)(b+1)+1,\label{eq:L-beta}\\
    N_\beta&\coloneq\beta(a+\tau+1)-a,\label{eq:N-beta}\\
    R_\beta&\coloneq\beta(a+\tau+1),\label{eq:R-beta}\\
    P_\beta&\coloneq\beta(b+1).\label{eq:P-beta}
\end{align} In the next lemma, we will see that $L_\beta$, $N_\beta$, $R_\beta$, and $P_\beta$ are the heap sizes corresponding to the first $\lp$-, $\np$-, $\rp$-, and $\pp$-position in $\beta$-th block, respectively. 

\begin{lemma}\label{lem: outcome ts below alpha block}
    Let $a,b$ and $\tau$ are positive integers such that $a<b$ and $\tau<b-a$. Then, for all $\beta<\hbeta$, 
    \begin{align*}
        o\big(\ts_\tau(n;a,b)\big)=\begin{cases}
            \lp & \text{ if } L_\beta \le  n<N_\beta ;\\
            \np & \text{ if } N_\beta\le  n< R_\beta;\\
            \rp & \text{ if } R_\beta\le  n< P_\beta;\\
            \pp & \text{ if } n=P_\beta.
        \end{cases}
    \end{align*}
\end{lemma}
\begin{proof}
    If $2\tau<b-a$, then $\hbeta=1$, and the statements holds vacuously. So, assume $2\tau\ge b-a$. We must prove that
    \begin{itemize}
        \item[(i)] Left wins $G$ if $L_\beta\le n< N_\beta$,
        \item[(ii)] First player wins $G$ if $N_\beta\le n< R_\beta$,
        \item[(iii)] Right wins $G$ if $R_\beta\le n< P_\beta$,
        \item[(iv)] Second player wins $G$ if $n=P_\beta$.
    \end{itemize}

    We first consider the case $\beta=1$. If $L_1\le n<P_1$, the statement holds using Lemma~\ref{lem: outcome TS heap<=b}. So let $n=P_1=b+1$: If Left starts, then any resulting position is either $\np$ or $\rp$ as $n-a>\tau+1=N_1$ and $n-1<P_1$. Thus, Right wins $G$. If Right starts, any resulting position is either $\lp$ or $\np$ as $n-b=1\ge L_1$ and $n-(\tau+1)\le a+\tau<R_1$. Thus, Left wins $G$. This completes the proof for $\beta=1$. 
    
    Now suppose $\beta>1$, i.e, $n\ge b+2$.

    
    Consider item~(i).  Suppose first that Right starts in $G$ and moves to $G^R$. Denote the heap size of $G^R$ by $n_r$. Since $L_\beta\le n< N_\beta$, $$L_\beta-b\le n_r<N_\beta-(\tau+1).$$ By simple calculations, we get $L_\beta-b>L_{\beta-1}$ and $N_\beta-(\tau+1)=R_{\beta-1}$. Thus, $$L_{\beta-1}<n_r< R_{\beta-1}.$$ Hence, by induction, the outcome of $G^R$ is either $\lp$ or $\np$, and thus, Left wins playing first on $G^R$. Since $G^R$ is an arbitrary Right option, Left wins $G$. 
    
    Next suppose that Left starts in $G$. We claim that Left wins by removing 1 token. Let $G^L$ denote the Left option $\ts_\tau(n-1;a,b)$. If $n=L_\beta$, i.e., $n-1=P_{\beta-1}$, then outcome of $G^L$ is $\pp$ by induction and hence, Left wins $G$. Otherwise $L_\beta\le n-1< N_\beta-1$, and in this range, we already proved that Right loses as first player. Thus, Left wins $G$ by moving to $\ts_\tau(n-1;a,b)$.
    
    Consider item~(ii). First suppose that Left starts on $G$. Since the outcome of {\sc TS} is $\lp$ below the heap size $N_\beta$ by item~(i), Left wins if she can reach the heap size $N_\beta-1$. For this purpose, she needs to remove $n-N_\beta+1$ tokens 
    and this is possible as $$1\le n-N_\beta+1< R_\beta-N_\beta+1=a+1.$$ Therefore, Left wins $G$ by doing so. 
    
    Next suppose that Right starts on $G$. 
    Right can win by reducing the heap size to $R_{\beta-1}$ by induction. For this purpose, Right needs to remove $n-R_{\beta-1}$ many tokens. Since $$\tau+1=N_\beta-R_{\beta-1}\le n-R_{\beta-1}<R_\beta-R_{\beta-1}=a+\tau+1\le b,$$ Right does so and wins $G$. 

    Consider item~(iii). Suppose Right starts on $G$. 
    By induction, Right wins if he can reach any heap size between $R_{\beta-1}$ and $P_{\beta-1}$, both included. We know \begin{align*}
        R_{\beta-1}&=R_\beta-(a+\tau+1),\\
        P_{\beta-1}&=P_\beta-(b+1).
    \end{align*}This implies $R_{\beta-1}+a+\tau+1\le n\le P_{\beta-1}+b$. Since all numbers from $a+\tau+1$ to $b$ are in the subtraction set of Right, he can always reach a heap size between $R_{\beta-1}$ and $P_{\beta-1}$, both included and thus, Right wins. 
    
    Now suppose Left starts on $G$ and her move leads to heap size $n_\ell$. Then $$N_\beta=R_\beta-a\le n_\ell< P_\beta-1.$$ We already proved that Right wins as first player on any of these heap sizes. Since, $n_\ell$ is arbitrary, Right wins $G$. 

    Consider item~(iv). Suppose Left starts on $G$ and move to heap size $n_\ell$. Then, $$N_\beta\le P_\beta-a\le n_\ell\le P_\beta-1.$$ We already proved that Right wins as first player on any of these heap sizes. Since $n_\ell$ is arbitrary, Right wins $G$. 
    
    Now suppose Right starts on $G$ and move to heap size $n_r$. Then, \begin{align}
        L_\beta\le P_\beta-b\le n_r&\le P_\beta-(\tau+1)\nonumber\\
        &=\beta(b+1)-(\tau+1)-R_\beta+R_\beta\nonumber\\
        &=\beta(b-a-\tau)-(\tau+1)+R_\beta\nonumber\\
        &<\frac{\tau+1}{b-a-\tau}(b-a-\tau)-(\tau+1)+R_\beta\label{eq:1}\\
        &=R_\beta\nonumber.
    \end{align}
    Equation~\eqref{eq:1} follows as $\beta\le \hbeta-1\le \frac{\tau}{b-a-\tau}<\frac{\tau+1}{b-a-\tau}$. We already proved that Left wins as first player on any of these heap sizes. Since, $n_r$ is arbitrary, Left wins $G$. 
\end{proof}

We now establish the outcome pattern of the $\widehat \beta$-th block. Similar to $L_\beta$, $N_\beta$, and $R_\beta$ defined in Equations~\eqref{eq:L-beta}, \eqref{eq:N-beta} and \eqref{eq:R-beta} for $\beta<\hbeta$, we define $L_{\hbeta}$, $N_{\hbeta}$, and $R_{\hbeta}$ for the $\hbeta$-th block as \begin{align}
    L_{\hbeta}&\coloneq(\hbeta-1)(b+1)+1,\\
    N_{\hbeta}&\coloneq\hbeta(a+\tau+1)-a,\\
    R_{\hbeta}&\coloneq\hbeta(a+\tau+1).
\end{align} In the next lemma, we will see that $L_{\hbeta}$, $N_{\hbeta}$, and $R_{\hbeta}$ are the heap sizes corresponding to the first $\lp$-, $\np$-, and $\rp$-position in the $\hbeta$-th block, respectively. We will also see that there is no $\pp$-position in this block.


\begin{lemma}\label{lem: outcome ts in alpha block}
    Let $a,b$ and $\tau$ are positive integers such that $a<b$ and $\tau<b-a$. 
    Then \begin{align}
        o\big(\ts_\tau(n;a,b)\big)=\begin{cases}
            \lp & \text{ if } L_{\hbeta} \le  n<N_{\hbeta} ;\\
            \np & \text{ if } N_{\hbeta}\le  n< R_{\hbeta};\\
            \rp & \text{ if } R_{\hbeta}\le  n\le {\hbeta}(b+1).
        \end{cases}
    \end{align}
\end{lemma}
\begin{proof}
    It suffices to show the following:
    \begin{enumerate}
        \item Left wins $G$ if $L_{\hbeta}\le n<N_{\hbeta}$,
        \item First players wins $G$ if $N_{\hbeta}\le n<R_{\hbeta}$,
        \item Right wins $G$ if $R_{\hbeta}\le n< {\hbeta}(b+1)$,
        \item Right wins $G$ if $n={\hbeta}(b+1)$.
    \end{enumerate}
    We skip the proof of items~(1), (2) and (3), as they are similar to the proof of items~(i), (ii) and (iii) of Lemma~\ref{lem: outcome ts below alpha block}.  

    Let $n=\hbeta(b+1)$.  Suppose Right starts on $G$, then he wins by removing $n-R_{\hbeta}$ tokens from the heap as it leads to $\ts_\tau(R_{\hbeta};a,b)$, which is an $\rp$-position by item~(3); this is a legal move as $n-R_{\hbeta}<b$ and
\begin{align}
    n-R_{\hbeta}&={\hbeta}(b+1)-{\hbeta}(a+\tau+1)\nonumber\\
    &={\hbeta}(b-a-\tau)\nonumber\\
    &=\left\lfloor\frac{b-a}{b-a-\tau}\right\rfloor \cdot (b-a-\tau)\nonumber\\
    &>\left(\frac{b-a}{b-a-\tau}-1\right) \cdot (b-a-\tau) =\tau.\label{eq:2}
\end{align}
Now suppose that Left starts on $G$ and move to heap size $n_\ell$. Then $$n-a\le n_\ell<n,$$
where \begin{align*}n-a  &=  n-(R_{\hbeta}-N_{\hbeta}) \\
&={\hbeta}(b+1)-R_{\hbeta} +N_{\hbeta} \\
&> N_{\hbeta}.\end{align*} By items~(2) and (3), the outcome of $\ts_\tau(n_\ell;a,b)$ is either $\np$ or $\rp$, and hence, Right wins $G$.  This completes the proof.
\end{proof}

We have established the pattern of the outcomes of {\sc TS} up to the heap size $\hbeta(b+1)$; we now prove that all the remaining heap sizes are $\rp$-positions.

\begin{lemma}\label{lem: right dominating}
     Consider $G=\ts_\tau(n;a,b)$ where $a<b$ and $\tau<b-a$. Then  $G\in \rp$ for all $n\ge R_{\hbeta}$.
\end{lemma}
\begin{proof}
    We prove the statement by induction. 
    The statement is true for $R_{\hbeta}\le n\le \hbeta(b+1)$ by Lemma~\ref{lem: outcome ts in alpha block}. So suppose $n>\hbeta(b+1)$. 
    
    Suppose first that Left starts on $G$ and move to heap size $n_\ell$. Then $$n_\ell>\hbeta(b+1)-a=\hbeta(b+1)-(R_{\hbeta}-N_{\hbeta})>N_{\hbeta}.$$ Therefore, by Lemma~\ref{lem: outcome ts in alpha block} and induction, the outcome of $\ts_\tau(n_\ell;a,b)$ is either $\np$ or $\rp$. Hence, Right wins this position as the next player. Since, $n_\ell$ is arbitrary, Right wins $G$. 
    
    Now suppose that Right starts on $G$ and removes $\hbeta(b+1)-R_{\hbeta}$ tokens. This is legal as $$\tau<\hbeta(b+1)-R_{\hbeta}<b,$$ by Equation~\eqref{eq:2}. 
    Since $n>\hbeta(b+1)$, the resulting heap size is greater than $R_{\hbeta}$, and by induction, this position has outcome $\rp$. Hence, Right wins $G$. 
\end{proof}

We have find the outcomes of $\ts_\tau(n;a,b)$ for all small truncation levels ($\tau<b-a$) and proved that Right wins $\ts_\tau(n;a,b)$ for all sufficiently large $n$. However, this information is not enough to play strategically in a disjunctive sum of games. Additional information like atomic weight is required. In the next section, We compute the atomic weight of the {\sc TS} ruleset. 

\section{Atomic weights of Truncated Support}\label{sec: AW of TS} 
Table~\ref{tab:aw of ts(n,3,7)} suggests a pattern in the atomic weights of $\ts_\tau(n;3,7)$. For $\tau<b-a$, the pattern begins exactly at heap size 
$R_{\hbeta}=\hbeta(a+\tau+1)$, where $\hbeta=\left\lfloor\frac{b-a}{b-a-\tau}\right\rfloor$. By Theorem~\ref{thm: outcomes of TS}, starting from $R_{\hbeta}$, all larger heap sizes have outcome $\rp$. So, Right can delay playing as long as the heap size remains larger or equal to $R_{\hbeta}$.  Left can remove at most $a$ tokens per move, and thus the minimum number of moves required for Left to reduce the heap size below $R_{\hbeta}$ is: 
\[ 
\left\lceil\frac{n-(R_{\hbeta}-1)}{a}\right\rceil.
\] 
Thus, Right can delay his move exactly $\left\lceil\frac{n-(R_{\hbeta}\,-1)}{a}\right\rceil-1=\left\lfloor\frac{n-R_{\hbeta}}{a}\right\rfloor$ turns, suggesting that the atomic weight is
\[
-\left\lfloor\frac{n-R_{\hbeta}}{a}\right\rfloor.
\]

Note the similarity of these expressions to those for {\sc Full Support}, Section~\ref{sec: AW of FS}.\footnote{In {\sc Full Support}, starting from heap size $a+1$, all larger heap sizes have outcome $\rp$. In that sense, $R_{\hbeta}=a+1$ in {\sc FS}. However, there is a difference of one in the atomic weight formula for {\sc FS} and {\sc TS}.} 
The next result will be used to prove this formula of atomic weight of {\sc TS} when $\tau<(b-a)/2$. We conjecture that the same formula holds whenever $(b-a)/2\le \tau<b-a$. Note that if $\tau<(b-a)/2$, then $$R_{\hbeta}=a+\tau+1.$$ 

Recall that far star, $\cgfarstar$, is a nimber of birthday larger than the birthday of the game added to it. 
We use basic CGT results, like the bijection of partial order with outcome classes (Theorem~\ref{thm: second fundamental theorem}), $\cgup>0$ and $\cgup\cgfarstar>0$, without further mention. Another convenient tool is Lemma~\ref{lem: n-cgup-cgfarstar}, which concerns the canonical forms  
\begin{align}\label{eq:cfnupfarstar}
n\cdot\cgup\;\cgfarstar=\{0\mid (n-1)\cdot\cgup\;\cgfarstar\} \text{ and } n\cdot\cgdown\;\cgfarstar=\{(n-1)\cdot\cgdown\;\cgfarstar\mid0\}.
\end{align}

\begin{lemma}\label{lem: help to shorten main theorem}
    Let $0<\tau<(b-a)/2$ and $\tau+1\le n\le \tau+a$. Then $\ts_\tau(n;a,b)+\cgup\cgfarstar\in \lp$.
\end{lemma}
\begin{proof}
    If Left starts, she removes $n-\tau$ tokens from the heap; this is legal as $1\le n-\tau\le a$. This results in the position  $\ts_\tau(\tau;a,b)+\cgup\cgfarstar$, which is equal to $\tau+\cgup\cgfarstar$ by Observation~\ref{obs: TS=n if n<=tau}. Since $\cgup\cgfarstar>0$ and $\tau>0$, we have $\tau+\cgup\cgfarstar>0$ and hence, Left wins. 

    If Right starts and plays on the heap, he can remove anything from $\tau+1$ to $n$ tokens as $n\le a+\tau< b$. If he removes the full heap, the game becomes $\cgup\cgfarstar$, which is positive, and hence Left wins. Otherwise, the heap size remains positive and less than $a$, because the minimum Right could remove was $\tau+1$ from $n$ and $n-\tau-1< a$. Therefore, Left can remove the full heap in the next turn and win as $\cgup\cgfarstar>0$. 

    If Right instead plays on the  $\cgup\cgfarstar$ component, then Left can remove $n-\tau$ tokens from the heap. This results in the positions $\ts_\tau(\tau;a,b)+\cgfarstar$ using Equation~\eqref{eq:cfnupfarstar}, which is strictly positive as $\cgfarstar$ is an infinitesimal and $\ts_\tau(\tau;a,b)=\tau>0$ by Observation~\ref{obs: TS=n if n<=tau}. Hence, Left wins. 
\end{proof}

We now prove the fourth main result of this paper, Theorem~\ref{thm: AW of TS}, which states that given \begin{align}
1\le \tau&<(b-a)/2, \text{ and } \label{eq:tauab2} \\
n&\ge a+\tau+1, \label{eq:n>=a+tau+1}
\end{align} the atomic weight of $G(n)=\ts_\tau(n;a,b)$ is 
\begin{align*}
    \aw(G(n)) =-\left\lfloor \frac{n-(a+\tau+1)}{a}\right\rfloor.
\end{align*}

\begin{proof}[Proof of Theorem~\ref{thm: AW of TS}]
   For all heap sizes $n$, let $$G(n)=\ts_\tau(n;a,b)\;\; \text{ and } \;\;w(n)=\left\lfloor \frac{n-(a+\tau+1)}{a}\right\rfloor.$$ 
   Observe that 
   \begin{align}
        w(n)&=w(n-a)+1,\label{eq: w(n)=w(n-a)+1}\;\;\\
        w(n)&\ge w(n-j) \text { for all } j>0\label{eq: w(n)>w(n-j)},\text{ and }\\
        w(n)&=k \text{ for all } (k+1)a+\tau+1\le n\le (k+2)a+\tau. \label{eq: compute w(n)}
    \end{align} 
    By Theorem~\ref{thm: check farstareq} and Definition~\ref{def: atomic weight}, it suffices to prove that, given the provisos~\eqref{eq:tauab2} and \eqref{eq:n>=a+tau+1},
    \begin{align}    
        \cgdown\cgfarstar<G(n)+w(n)\cdot\cgup<\cgup\cgfarstar.
    \end{align}

    That is, we prove that 
    \begin{enumerate}
        \item Left wins $H=G(n)+(w(n)+1)\cdot\cgup\;\cgfarstar$, and
        \item Right wins $J=G(n)+w(n-a)\cdot\cgup\;\cgfarstar$,
    \end{enumerate}
    where (2) also uses \eqref{eq: w(n)=w(n-a)+1}. 
    

    Intuitively, since $G(n)\in \rp$ for all $n\ge a+\tau+1$, and Right should reduce Left's advantage in $H$ and $J$ by playing on $\cgup$, whereas, Left should play in $G(n)$ to reduce the heap size by as much as possible.\\ 
    
    \noindent{\bf Case (1).} Suppose first that Left starts in $H$. She removes $a$ tokens from $G(n)$. This results in the position $$H^L=G(n-a)+(w(n)+1)\cdot\cgup\;\cgfarstar.$$ If $n-a\ge a+\tau+1$, then by induction, $\aw(G(n-a))=-w(n-a)$ and hence, Left wins by the two ahead rule (Theorem~\ref{thm: outcome using aw}) using Equation~\eqref{eq: w(n)=w(n-a)+1}. 
    Otherwise, $\tau+1\le n-a\le \tau+a$ by Equation~\eqref{eq:n>=a+tau+1}, and thus, $w(n)=0$. Hence, Left wins $H^L$ by Lemma~\ref{lem: help to shorten main theorem}.
    
    Suppose next that Right starts in $H$. If he plays in $G(n)$, by removing $\tau+1\le j\le b$ tokens, then the resulting position is 
    \begin{align}\label{eq:HR}
        H^{R_1}=G(n-j)+(w(n)+1)\cdot\cgup\;\cgfarstar.
    \end{align} For large $n$, Left is at least one move ahead in $H^{R_1}$ by induction and she is the next player, so she wins. Specifically,
    if $n-j\ge a+\tau+1$, then by induction, $\aw(G(n-j))=-w(n-j)$. As atomic weight is additive by Lemma~\ref{lem: aw of sum of games}, \begin{align*}
        \aw(H^{R_1})&=-w(n-j)+\aw\left((w(n)+1)\cdot\cgup\right)+\aw(\cgfarstar)\\
        &=-w(n-j)+w(n)+1+0\\
        &\ge 1,
    \end{align*} where the second equality holds as $\aw(k\cdot\cgup)=k$ for all integers $k$ and the inequality holds as $w(n)\ge w(n-j)$ by \eqref{eq: w(n)>w(n-j)}. As Left is the next player on $H^{R_1}$ and $\aw(H^{R_1})\ge 1$, she wins by Theorem~\ref{thm: outcome using aw}.
    For small $n$, we split the analysis into two cases:
    \begin{enumerate}
        \item If $\tau+1\le n-j\le a+\tau$, then by Lemma~\ref{lem: help to shorten main theorem}, $G(n-j)+\cgup\cgfarstar>0$. Since $w(n)\cdot \cgup\ge 0$ for all $n$, Left wins the sum \eqref{eq:HR}.
        \item If $n-j\le \tau$, then Right do not have any move on $G(n-j)$ and thus $G(n-j)=n-j\ge 0($ by Observation~\ref{obs: TS=n if n<=tau}). Since $w(n)\cdot\cgup\ge 0$ and $\cgup\;\cgfarstar>0$, Left wins the sum \eqref{eq:HR}.
    \end{enumerate}
    This completes the case when Right starts on $H$ by playing on $G(n)$. Next we consider the case when Right starts on $H$ by moving on $(w(n)+1)\cdot\cgup\;\cgfarstar$. This move results in $$H^{R_2}=G(n)+w(n)\cdot\cgup\;\cgfarstar$$ as $k\cdot\cgup\cgfarstar=\{0\mid (k-1)\cdot\cgup\;\cgfarstar\}$ by \eqref{eq:cfnupfarstar}. By Equation~\eqref{eq: w(n)=w(n-a)+1}, $w(n)=w(n-a)+1$ and consequently, $H^{R_2}=G(n)+(w(n-a)+1)\cdot\cgup\;\cgfarstar$.
    Left responds by removing $a$ tokens from $G(n)$. This results in the position $$H^{R_2L}= G(n-a)+(w(n-a)+1)\cdot\cgup\;\cgfarstar.$$ If $n-a\ge a+\tau+1$, Left wins $H^{R_2L}$ by induction. Otherwise, $\tau+1\le n-a\le a+\tau$ as $n\ge a+\tau+1$ by Equation~\eqref{eq:n>=a+tau+1}, and thus, $w(n-a)=-1$. Hence, $H^{R_2L}=G(n-a)+\cgfarstar$. In the next turn, if Right moves on $G(n-a)$ and do not empty the heap, then:
    \begin{itemize}
        \item[(a)] if the resultant heap is smaller than $\tau+1$, then Left moves from $\cgfarstar$ to 0 and wins as Right does not have any further move,
        \item[(b)] otherwise, Left wins by reducing the heap below $\tau+1$. She can do this as Right's previous move reduced the heap below $a$. 
    \end{itemize}
    If Right empties the heap in the previous move, then Left wins by moving from $\cgfarstar$ to 0. If Right moves on $\cgfarstar$ in $H^{R_2L}$, then Left reduces the heap to $\tau$ and wins as the resulting game is positive. 
    
    \noindent{\bf Case (2).} We must show that Right wins $J=G(n)+w(n-a)\cdot\cgup\;\cgfarstar$. Suppose first that Left starts and she plays on $G(n)$, by removing $1\le i\le a$ tokens, then the resulting position is 
    \begin{align}\label{eq:JL}
        J^L=G(n-i)+w(n-a)\cdot\cgup\;\cgfarstar. 
    \end{align}
    Intuitively, by this move, Left reduces Right's delay advantage, so in the next turn, Right also reduces Left delay advantage, if any. 
    We split the analysis of Right's winning strategies in the case of \eqref{eq:JL} into three cases:
    \begin{enumerate}
        \item If $n\ge 3a+\tau+1$, then $w(n-a)\ge 1$ and Right responds by moving on $w(n-a)\cdot\cgup\;\cgfarstar$ in \eqref{eq:JL}. This results in the position $$J^{LR}=G(n-i)+(w(n-a)-1)\cdot\cgup\;\cgfarstar,$$ by Equation~\eqref{eq:cfnupfarstar}. Since $w(n-a)-1\le w(n-i-a)$ by Equations~\eqref{eq: w(n)>w(n-j)} and \eqref{eq: w(n)=w(n-a)+1}, we have $$J^{LR}\le G(n-i)+w(n-i-a)\cdot\cgup\;\cgfarstar .$$ Because Right wins the lates by induction, he also wins $J^{LR}$.
        \item If $a+\tau+i<n\le 3a+\tau$, then $w(n-a)$ is either $-1$ or $0$. In both case, Right has a move from $w(n-a)\cdot\cgup\;\cgfarstar$ to $0$ in $J^L$ as $\cgdown\cgfarstar=\{\cgfarstar\mid 0\}$ by \eqref{eq:cfnupfarstar}. This results in $G(n-i)$, which is an $\rp$-position by Theorem~\ref{thm: outcomes of TS}(\ref{item:thm: outcomes of TS:3}) as $n-i\ge a+\tau+1$. Hence, Right wins.
        \item If $n\le a+\tau+i$, then $n\le 2a+\tau$,  and $w(n-a)=-1$. Right responds by removing the full heap; this is legal as $\tau+1\le a+\tau+1-i\le n-i\le a+\tau<b$, by \eqref{eq:tauab2}. This results in $0+\cgdown\cgfarstar$, an $\rp$-position. Hence, Right wins.
    \end{enumerate} 
    This completes the case when Left starts by playing on $G(n)$ in $J$. Now suppose that Left instead plays on the $w(n-a)\cdot\cgup\;\cgfarstar$ component in $J$, then again Right's response depends on the heap size $n$: 
    \begin{enumerate}[(a)]
        \item If $n\ge 3a+\tau+1$, then $w(n-a)\ge 1$ and Left's move results in $G(n)+0$, which is an $\rp$-position. Hence, Right wins. 
        \item\label{item:left play on up:2} If $2a+\tau+1\le n\le 3a+\tau$, then $w(n-a)=0$, and Left's move results in $G(n)+*m$ for some $m\ge 0$. Since $G(n)\in \rp$, if $m=0$, Right wins the current position by Theorem~\ref{thm: outcomes of TS}(\ref{item:thm: outcomes of TS:3}) and otherwise, Right wins by moving from $*m$ to $0$ using the same theorem.
        \item\label{item:left play on up:1} If $a+\tau+1\le n\le 2a+\tau$, then $w(n-a)=-1$ and therefore her move results in $G(n)+\cgfarstar$ by Equation~\eqref{eq:cfnupfarstar}. Right responds by moving from $\cgfarstar$ to $0$ and wins as $G(n)\in \rp$ by Theorem~\ref{thm: outcomes of TS}(\ref{item:thm: outcomes of TS:3}).
    \end{enumerate}
    
    We proved that Left loses as first player on $J$. Suppose next that Right starts on $J$. He reduces Left's advantage by playing on $w(n-a)\cdot\cgup\;\cgfarstar$. The analysis of the resulting position depends on $n$:
    \begin{enumerate}[(i)]
        \item If $n\ge 3a+\tau+1$, then $w(n-a)\ge 1$, and by \eqref{eq:cfnupfarstar}, his move results in $$J^R=G(n)+(w(n-a)-1)\cdot\cgup\;\cgfarstar.$$ since Left loses as first player on $J$ and $J^R<J$, Left also loses as first player on $J^R$. Hence, Right wins.
        \item If $2a+\tau+1\le n\le 3a+\tau+1$, then $w(n-a)=0$, and Right can move from $\cgfarstar$ to 0 in $J$, leading to $G(n)$, which is an $\rp$-position by Theorem~\ref{thm: outcomes of TS}(\ref{item:thm: outcomes of TS:3}) and hence, Right wins.
        \item If $a+\tau+1\le n\le 2a+\tau$, then $w(n-a)=-1$, and his move results in $G(n)$ which is an $\rp$-position by Theorem~\ref{thm: outcomes of TS}(\ref{item:thm: outcomes of TS:3}), and hence, Right wins.
    \end{enumerate}This completes Case (2).
\end{proof}

Furthermore, we conjecture that the formula remains valid for $(b-a)/2\le\tau<b-a$. Note that {\sc TS} is not an all-small ruleset and the atomic weight does not exist for all positions.

\section{Open Problems}
\begin{conjecture}[Atomic Weight of {\sc TS}]\label{con: AW of TS for all k<b-a}
    Let $G=\ts_\tau(n;a,b)$ where $a<b$. If $\tau< (b-a)$, then for all $n\ge R_{\hbeta}$
    \begin{align}
    \aw(G)=-\left\lfloor \frac{n-R_{\hbeta}}{a}\right\rfloor.
    \end{align}
\end{conjecture}


\subsection*{Domination in Truncated Support}
We find some dominated options of {\sc TS} at the knife's edge (truncation level $\kappa=b-a$). 
Given $a$ and $b$, by Theorem~\ref{thm: CF periodicity in TS at b-a}, there are at most $b+1$ distinct canonical forms of $\ts_\kappa(n;a,b)$ and those can be represented by  $$\ts_\kappa(0;a,b),\;\ts_\kappa(1;a,b),\dots,\;\ts_\kappa(b;a,b).$$ The options of these games have heap sizes ranging from $0$ to $b-1$. Thus, only the games with heap sizes $0,1,\dots,b-1$ are required to be compared. Among these, by Observation~\ref{obs: TS=n if n<=tau}, the first $\kappa+1$ heap sizes are canonical form integers, and consequently they are comparable. For the remaining part, we compare the consecutive games, $G=\ts_{\kappa}(n;a,b)$ and $H=\ts_{\kappa}(n+1;a,b)$, where $n\in \Set{\kappa+1,\dots,b-3,b-2}$, in the next proposition. We show that $G\le H$ if $n\notin  Q$, where $$Q\coloneq\Set{(m+1)(b-a)+m\SetSymbol m\in \mathbb Z_{\geq 1}}.$$

\begin{proposition}\label{prop: domination at b-a}
Consider $G=\ts_{\kappa}(n;a,b)$ and $H=\ts_{\kappa}(n+1;a,b)$ where $3\le a<b$. Then, $G\le H$ for all $n\in \Set{b-a+1,\dots,b-3,b-2}\backslash Q$.
\end{proposition}
We omit the proof. This proposition tells that if $b-a$ is large, then for most $n\in \{\kappa+1\dots,b\}$, $n\notin Q$, and therefore, most of the Left and Right options are dominated. Hence, the game values simplify. We suggest the following problem.

\begin{problem}
    Is it true that the only dominated options of $\ts_{b-a}(n;a,b)$ are given by Proposition~\ref{prop: domination at b-a}.  
\end{problem}

\begin{problem}
    Find the canonical forms of {\sc Truncated Support}.
\end{problem}

\appendix
\section{Basic CGT.}\label{sec: prelims}
    To establish the foundation for our results on canonical forms and atomic weights, we briefly revisit some fundamental concepts from Combinatorial Game Theory (CGT). 
    Familiarity with basic CGT terminology is assumed; several standard results are stated without proof, as they are well covered in the literature (see, e.g., \cites{BCG2004, S2013}).


    A game $G$ in CGT is recursively defined as $G = \{G^{\mathscr{L}} \mid G^{\mathcal{R}}\}$, where $G^{\mathcal{L}}$ and $G^{\mathcal{R}}$ denote the sets of options available to the players Left and Right, respectively.  
    The base of the recursion is the \emph{zero game}, where neither player has a move, denoted by $0 = \{\varnothing \mid \varnothing\}$. As usual, $\varnothing$ denotes an empty set (of options).
    
    Two games can be played together under the rule that on a player’s turn, they must choose one game and make a legal move in it. This idea is formalized below.

    \begin{definition}[Disjunctive Sum]\label{def: disjunctive sum}
        Let $G$ and $H$ be short games. The \emph{disjunctive sum} $G + H$ is defined recursively by
        $$G + H = \{\, G^{\mathcal{L}} + H,\, G + H^{\mathcal{L}} \mid G^{\mathcal{R}} + H,\, G + H^{\mathcal{R}} \,\}.$$
        Here, $G^{\mathcal{L}} + H = \{\, X + H : X \in G^{\mathcal{L}} \,\}$, and the other terms are defined analogously.
    \end{definition}
    We now recall the first fundamental theorem of CGT.

    \begin{theorem}[First Fundamental Theorem]
        Every short game belongs to exactly one of the four \emph{outcome classes}:
        \begin{itemize}
            \item $\mathscr{L}$: Left can force a win.
            \item $\mathscr{R}$: Right can force a win.
            \item $\mathscr{N}$: the Next (first) player can force a win.
            \item $\mathscr{P}$: the Previous (second) player can force a win.
        \end{itemize}
        The outcome class of a game $G$ is denoted by $o(G)$.
    \end{theorem}

    
    The short (finite) games are ordered with respect to their outcome classes. 
    The outcome classes are partially ordered according to their favorability to Left:
    $$\mathscr{L} \;>\; \mathscr{N}\;>\;\mathscr{R}\quad \text{and}\quad \mathscr{L} \;>\;\mathscr{P} \;>\; \mathscr{R}.$$

    The structure is visualized  in Figure~\ref{fig: partial order structure}. Whenever we use the terminology ``Right wins $G$'', then we mean $G\in \rp$, and similarly for Left.

    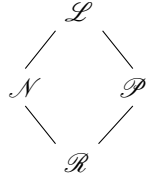
\begin{figure}[h!]
        \centering
        \begin{tikzpicture}[scale=0.5]
            \node (L) at (0,0) {$\mathscr{L}$};
            \draw (L.south west)--+(-0.8,-1) node[anchor=north] (N) {$\np$};
            \draw (L.south east)--+(0.8,-1) node[anchor=north] (P) {$\pp$};
            \draw (N.south)--+(0.8,-1) node[anchor=north west] (R) {$\rp$};
            \draw (P.south)--(R.north east);
        \end{tikzpicture}
        \caption{Partial order structure of outcome classes.}
        \label{fig: partial order structure}
    \end{figure}

    \begin{definition}
        Let $G$ and $H$ be short games. Then, $G\ge H$, if for all short games $X$, $$o(G+X)\ge o(H+X). $$
    \end{definition}
    Intuitively, $G \ge H$ means that $G$ is at least as favorable to Left as $H$.  
    While this definition captures the idea of game comparison, it is not practical for computation since the set of all games is infinite.  
    The following result, known as the \emph{Second Fundamental Theorem}, provides a more practical characterization when comparing games with the zero game.
    
   \begin{theorem}[{\cite{S2013}*{Proposition~1.18}}]\label{thm: second fundamental theorem}
        Let $G$ be a short game. Then the following holds:
        \begin{itemize}
            \item $G=0 \Leftrightarrow o(G)=\mathscr{P}$,
            \item $G\ge 0 \Leftrightarrow o(G)\ge \mathscr{P}$,
            \item $G> 0 \Leftrightarrow o(G)=\mathscr{L}$,
            \item $G\le 0 \Leftrightarrow o(G)\le \mathscr{P}$,
            \item $G< 0 \Leftrightarrow o(G)= \mathscr{R}$,
            \item $G\cglfuz 0 \Leftrightarrow o(G)\le \mathscr{N}$,
            \item $G \cggfuz 0 \Leftrightarrow o(G)\ge \mathscr{N}$,
            \item $G\cgfuzzy 0 \Leftrightarrow o(G)= \mathscr{N}$.
        \end{itemize} 
    \end{theorem}


    Using the second fundamental theorem, one can compare any game or disjunctive sum of games with 0. If inverse of a game is known with respect to the disjunctive sum, then any two games can be compared. 
    The following definition introduces the \emph{conjugate} of a game, which we shall see serves as the inverse of that game under disjunctive sum.
    \begin{definition}[Conjugate]
        Let $G$ be a short game. The conjugate of $G$, denoted by $\Bar{G}$, is defined recursively by $\Bar{G} = \{\Bar{G}^R\mid \Bar{G}^L\}.$
    \end{definition}

    Observe that the moves available to Left in a game $G$ are the same as those available to Right in its conjugate $\bar{G}$, and vice versa. Note that $\bar{G}$ is same as $-G$ defined in \cite{S2013}.

    \begin{theorem}[{\cite{S2013}*{Theorem~1.13}}]\label{thm:+ has inv}
        Let $G$ be a short game. Then $G+\bar{G}=0$.
    \end{theorem}
    Theorem~\ref{thm:+ has inv} implies that $\bar{G}$ is the inverse of $G$ under the disjunctive sum. This allows us to compare any two games.

     Using Theorem~\ref{thm: second fundamental theorem}, we obtain $\cgstar + \cgstar = 0$ as well as $\{-1 \mid 1\}=0$ as outcome of both $\cgstar+\cgstar$ and $\{-1\mid 1\}$ is $\pp$. This illustrates that combinatorial games can have multiple equivalent representations, even though their explicit forms are not identical. For instance, $\cgstar + \cgstar$ has an option $\cgstar$, whereas $0$ has no options, yet the two games are equivalent. We discuss this phenomenon in more detail in the next subsection.

    \subsection{Canonical Forms}\label{subsec: canonical form}
    As discussed above, every short combinatorial game $G$ may have several equivalent representations. This occurs because some options can be strategically inferior 
    while others may be overly advantageous. In this subsection, we show that every short game $G$ has a unique simplest form, called its \emph{canonical form}. The canonical form of a game is obtained by repeated application of two simplification rules: \emph{domination} and \emph{reversibility}.

        

    \begin{definition}[Dominated and Reversible Options]
        Let $G$ be a short game.
        \begin{itemize}
            \item A Left option $G^{L_1}$ is \emph{dominated} by another Left option $G^{L_2}$ if $G^{L_2} \ge G^{L_1}$. The definition is symmetric for Right.
            
            \item A Left option $G^{L_1}$ is \emph{reversible} through a Right option $G^{L_1R_1}$ if $G^{L_1R_1} \le G$ for some $G^{L_1R_1} \in (G^{L_1})^{\mathscr{R}}$. The definition is symmetric for Right.
        \end{itemize}
    \end{definition}
    Intuitively, if $G^{L_1}$ is dominated by $G^{L_2}$, then Left will always prefer to move to $G^{L_2}$ rather than to $G^{L_1}$, since $G^{L_2}$ is at least as favorable to her as $G^{L_1}$. If $G^{L_1}$ is reversible, then after Left moves to $G^{L_1}$, Right can respond with $G^{L_1R_1}$, a position that is no worse (and possibly better) for Right than $G$ itself. Hence, After Left's move to $G^{L_1}$, Right will always move to $G^{L_1R_1}$.

    \begin{figure}[ht]
        \centering
        \begin{subfigure}[b]{0.40\textwidth}
        \centering
            \begin{tikzpicture}[every node/.style={font=\large}]
        
                \node (G) {$G$};
                \draw[->] (G.south west)--+(-1.5,-1) node[anchor=north] (GL) {$G^{L_1}$};
                \draw[->] (G.south west)--+(-0.5,-1) node[anchor=north] (GL1) {$G^{L_2}$};
                \draw[->] (GL.south)--+(0.5,-1) node[anchor=north] (GLR1) {$G^{L_1R_1}$};
                \draw[->] (GL.south)--+(1.8,-1) node[anchor=north, draw, red] (GLR2) {$G^{L_1R_2}$};
                \draw[->] (GL.south)--+(3.1,-1) node[anchor=north] (GLR3) {$G^{L_1R_3}$};
                \draw[->] (GLR2.south)--+(-0.5,-1) node[blue,anchor=north] (H3) {$H_3$};
                \draw[->] (GLR2.south)--+(-1.5,-1) node[blue,anchor=north] (H2) {$H_2$};
                \draw[->] (GLR2.south)--+(-2.5,-1) node[blue,anchor=north] (H1) {$H_1$};        
            \end{tikzpicture}
            \caption{Game tree of literal form of $G$}
        \end{subfigure}%
        \hspace{0em}%
        \begin{subfigure}[b]{0.40\textwidth}
        \centering
            \begin{tikzpicture}
                \node (G) {$G$};
                \draw[->] (G.south west)--+(-0.5,-1) node[anchor=north] (GL1) {$G^{L_2}$};
                \draw[->] (G.south west)--+(-1.5,-1) node[anchor=north,blue] (H3) {$H_3$};
                \draw[->] (G.south west)--+(-2.5,-1) node[blue,anchor=north] (H2) {$H_2$};
                \draw[->] (G.south west)--+(-3.5,-1) node[blue,anchor=north] (H1) {$H_1$};
                \draw[->,white] (H3)--+(0,-3.03) node (b) {};
            \end{tikzpicture}
            \caption{Game tree of simplified form of $G$ when $G^{L_1R_2}\le G$}
        \end{subfigure}
        \caption{Bypassing a reversible option. If $G^{L_1R_2}\le G$, then $G^L$ can be replaced by the set of all Left options of $G^{L_1R_2}$.}
        \label{fig: bypassing a reversible option}
    \end{figure}
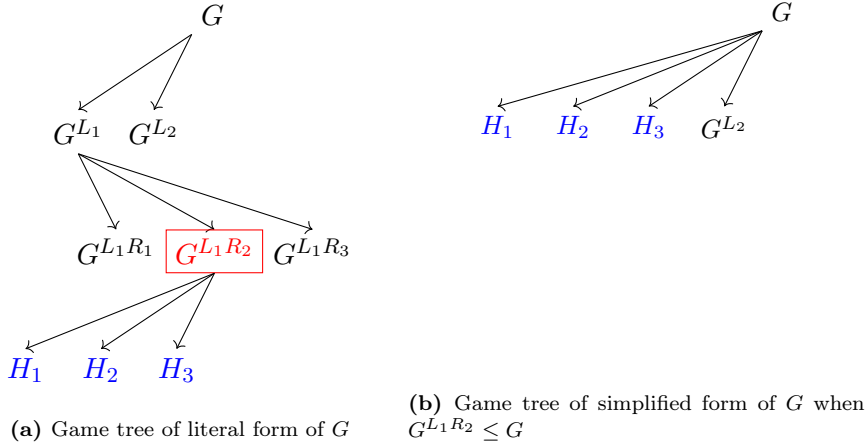
    To simplify a game $G$, one can remove all dominated options of $G$. Moreover, if a left option $G^{L_1}$ is reversible through $G^{L_1R_1}$, then $G^{L_1}$ can be replaced by the left options of $G^{L_1R_1}$, as illustrated in Figure~\ref{fig: bypassing a reversible option}; an analogous rule applies to right options. By repeating this process of removing and replacing, one eventually obtains the simplest possible form of the game, known as the \textbf{canonical form} of $G$. This result follows from \cite{S2013}*{Theorems~2.4~and~2.5}.
    

    \begin{example}
        Let $G = \uparrow + *$, where $\uparrow = \{0\mid *\}$. Then $ G=\{\uparrow+0, 0+*\mid \uparrow+0, *+*\} = \{\uparrow,*\mid \uparrow,0\}$. The canonical form of $G$ is $\{0,*\mid 0\}$.
    \end{example}

    The next subsection explores a particular type of game where neither player wants to start. 

\subsection{Numbers and stops}
    Intuitively, a game is called a `Number' if each player prefers that the other player starts. 
    Recall that a {\em sub-position} of a game can be the game itself or any option of the game or any option of options, etc.

    \begin{definition}[Numbers]\label{def: Number}
     A short game $x$ is a Number if, in the canonical form of $x$, every sub-position $y$ satisfies $y^L<y^R$ for all $y^L$ and $y^R$.\footnote{The definition of a Number in \cite{S2013} does not hold for $\{*\mid*\}$ and many other literal form games that equal some canonical form Number.}
    \end{definition}

    The most basic Number games are integer games defined as follows: For all \( k \in \mathbb{Z}_{>0} \), the \emph{integer games} \( k \) and \( -k \) recursively defined as: 
    \begin{itemize}
        \item $k = \left\{ k-1 \mid \varnothing \right\}$;
        \item $-k = \left\{ \varnothing \mid -k+1 \right\}$,
    \end{itemize}
    where $0 = \{\varnothing \mid \varnothing \}$. 

    For all odd $k\in \Z$ and $n\in \Nat$, we define the {\em dyadic rational games} recursively as: $$ \frac{k}{2^n}=\left\{ \left[\frac{k-1}{2^n}\right] \mid \left[\frac{k+1}{2^n}\right] \right\},$$ 
    where the brackets denote the reduction of the fraction such that the numerator and denominator do not have a common divisor other than 1. For example, with $k=5$ and $n=1$, the game $5/2=\{2\mid 3\}$. We will call all integers and dyadic rationals simply as {\em dyadics} and denote the set of all dyadics by $\mathbb{D}$. As shown in \cite{S2013}*{Proposition 3.5} integer and dyadic rational games follow the standard arithmetic properties. 
    For instance, the disjunctive sum of the games $1$ and $\frac{1}{2}$ equals the game $1+\frac{1}{2}=\frac{3}{2}$. By definition, all dyadics are Numbers. 
    
    Let us restate a few theorems on Numbers. 

    \begin{theorem}[{\cite{S2013}*{Theorem~3.13}}]\label{thm: number avoidance}
        Let $x$ be a Number. If $G$ is not a Number then, $G+x=\{G^\mathcal L+x\mid G^\mathcal R+x\}$.
    \end{theorem}
    
   \begin{theorem}[{\cite{S2013}*{Theorem~3.7}}]\label{thm: archimedean principle}
       For any game $G$, there exists a positive integer $n$ such that $-n<G<n$.
   \end{theorem}

    Due to Theorem~\ref{thm: archimedean principle}, any game is smaller than large enough numbers and greater than small enough numbers. 
    The supremum of all numbers smaller than a given game $G$ is called the Right Stop of $G$ and similarly, the infimum of all numbers larger than $G$ is called the Left Stop of $G$. 

    If a game is a Number, then both of its stops are equal; however, the converse does not hold. Games such as $*$ and $\cgup$ are not Numbers, yet both have Left and Right stops equal to $0$. This observation motivates the study of a broader class of games whose stops are both $0$ but which are not necessarily equal to $0$. These games are known as \emph{infinitesimals}. The next subsection explores such games and their properties in detail.
    
    \subsection{Infinitesimals}\label{subsec: infinitesimals}
    As mentioned above, infinitesimals are games whose Left and Right stops are both $0$. Before examining their structure and properties, we begin with a simpler subclass of games, namely the \emph{impartial} games, where both players have identical move options from every position.
    
    \begin{definition}[Impartial Game]
        A combinatorial game is \emph{impartial} if, from every subposition, both players have the same set of available moves.
    \end{definition}
    
    If a game $G$ is impartial, we do not distinguish between Left and Right options, since both players have identical choices from every position. 
    Hence, we represent an impartial game simply as $G = \{G_1, G_2, \dots, G_k\}.$

    A {\sc Nim} heap is an example of an impartial game and it is denoted by $*n$ where $n$ is the size of the heap and $*n=\{0,*,*2,\dots,*(n-1)\}$. These single-heap {\sc Nim} positions are called \textbf{nimbers}. As shown in \cite{S2013}*{Theorem~1.3}, every impartial game $G$ is equal to $*m$ for some non-negative integer $m$. 

    A related but more general class of games allows different sets of moves for each player, provided that at every subposition, either both players have available moves or neither does. 
    Such games are known as \emph{all-small} games. 

    \begin{definition}[All-Small Games]
     A game is called all-small if, both players have available moves from every non-terminal subposition.
    \end{definition}
    
    Every impartial game is all-small, but the converse does not hold. For example, an instance of {\sc FS} with unequal positive wealths is all-small but not impartial.
    
    If $G$ is an all-small game, then for every positive dyadic rational $x$, we have $ -x < G < x$. This inequality follows from the observation that in the game $G + x$, Left can secure a win by continually playing in $G$, while in $G - x$, Right can win by consistently playing in $G$.

    This property, however, is not limited to all-small games. Consider the game $H = \left\{ \left\{ 2 \mid 0 \right\} \mid 0 \right\}$. It satisfies  $-x < H < x$ for every positive dyadic rational $x$. But $H$ is not an all-small game as the subposition $2$ offers a move for Left but not for Right. This motivates a more general notion than all-small.
    
    \begin{definition}[Infinitesimal Game]
        A combinatorial game $G$ is called infinitesimal if, for every positive dyadic rational $x$, it satisfies $-x<G<x$.
    \end{definition}

    Returning to all-small games, the canonical forms of such games can be quite complex, as illustrated in Table~\ref{tab: cf of ts1(n;5,6)}. This complexity often makes it difficult to determine which player holds an advantage and to what extent. 
    

    The simplest non-zero all-small game is $*=\{0\mid0\}$; it offers no advantage to either player, as first player wins this game. Notably, all-small games of the form $*n$, known as \emph{nimbers}
    , behave similar to $*$ in this context. 
    
    The next simplest all-small game is Up, denoted $\cgup = \{0 \mid *\}$, which favors Left, as Left always wins this game. Its negative, called Down, denoted $\cgdown = \{* \mid 0\}$, favors Right.

    Since these are the simplest all-small games that confer an advantage to a player, it is reasonable to expect that any other all-small game which provides a player with an advantage must, in some sense, be comparable to these. Indeed, as shown in \cite{S2013}*{Theorem~4.8}, for every all-small game $G$, there is some $n$ such that $n\cdot\cgdown<G<n\cdot\cgup$, where $n\cdot\cgup$ denotes the disjunctive sum of $n$ copies of $\cgup$. This motivates the idea of evaluating the bound of advantage in an all-small game by 
    determining minimum $n$ such that $n\cdot\cgdown<G<n\cdot\cgup$. But Conway gave a stronger notion that gives the exact advantage in an infinitesimal, especially all-small games, known as \emph{atomic weight}. Broadly, the atomic weight of a game is the number of copies of $\cgup$ ($\cgdown$, if negative), the game approximates.

    To pursue this idea, he gave a more flexible notion of equivalence than that given by canonical forms. 
    Recall that {\em birthday} of a game is the height of the game tree in its literal form.

    \begin{definition}[Equivalence under $\cgfarstar$]
        Two games $G$ and $H$ are said to be equivalent under $\cgfarstar$, denoted $\cgfarstareq{G}{H}$, if for every game $X$, ~$o(G + X + \cgfarstar) = o(H + X + \cgfarstar),$ where $\cgfarstar$ denotes a nimber of birthday larger than that of the games added to it.
    \end{definition}


    The following Theorem gives a practical way for checking whether two games are equivalent under $\cgfarstar$. For any positive integer~$n$, we denote the disjunctive sum of $n$ copies of~$\cgup$ by $n\cdot\cgup$, and $n\cdot\cgup+\cgfarstar$ by $n\cdot\cgup\;\cgfarstar$. When $n=1$, we write $\cgup\;\cgfarstar$ for $1\cdot\cgup+\cgfarstar$. Similar conventions apply to~$\cgdown$.
    

    \begin{theorem}[{\cite{LiP}*{Theorem~9.30}}]\label{thm: check farstareq}
        For two games $G$ and $H$, $\cgfarstareq{G}{H}$ if and only if $~\cgdown\;\cgfarstar<G-H<\;\cgup\;\cgfarstar$.
    \end{theorem}
    
    
    Having established this relaxed equivalence relation, the next step in understanding players' advantage is to estimate how many copies of $\cgup$ or $\cgdown$ a given all-small game equals. In other words, for a game $G$, we wish to find a game $w$ such that $\cgfarstareq{G}{w \cdot \cgup}$ where $w \cdot \cgup$ denotes the `norton product' and in case $w$ is an integer, it is the disjunctive sum of $w$ copies of $\cgup$ when $w > 0$, and $|w|$ copies of $\cgdown$ when $w < 0$. It turns out that if the there exists such a  game $w$, it is unique, and it is referred to as the \emph{atomic weight} of $G$.


    \begin{definition}[Atomic Weight]\label{def: atomic weight}
        For a game $G$, if $\cgfarstareq{G}{w\cdot\cgup}$, for some game $w$, then $G$ is atomic, and the atomic weight of $G$ is  $\aw(G)=w$. \footnote{For further details, see \cites{LiP,S2013}.}
    \end{definition}
    
    If $w = \aw(G)$ is positive, then Left has an advantage in $G$; if $w$ is negative, then Right has an advantage in $G$, which is more precisely stated in Theorem~\ref{thm: outcome using aw}.

    
    Another way to interpret atomic weight is through {\em delayed} moves. We know that $\cgup$ offers Left a one move waiting  advantage, that is, Left can afford to wait through one Right move and still win, while $\cgdown$ offers the same advantage to Right.  Whenever the atomic weight of a game corresponds to the number of $\cgup$ or $\cgdown$ it is equivalent to, we can interpret the atomic weight as the number of moves a player can wait in the game before the opponent secures a win. Note that this interpretation is correct only when the atomic weight is larger than one. The atomic weight of $\cgup+\cgstar=\{0,*\mid 0\}$ is one but Left does not have a waiting  advantage, but  instead she has a parity advantage.

    The next few lemmas play an important role in computation of atomic weight of games.
    
   \begin{lemma}[{\cite{S2013}*{Page~139}}]\label{lem: n-cgup-cgfarstar}
        For any positive integer $n$, $$n\cdot\cgup\;\cgfarstar=\{0\mid (n-1)\cdot\cgup\;\cgfarstar\} \text{ and } n\cdot\cgdown\;\cgfarstar=\{(n-1)\cdot\cgdown\;\cgfarstar\mid0\}$$
    \end{lemma}

    \begin{lemma}[\cite{S2013}*{Proposition~7.12}]\label{lem: aw of sum of games}
        Let $G$ and $H$ be atomic, then \begin{enumerate}
            \item $\aw(-G)=-\aw(G)$;
            \item $\aw(G+H)=\aw(G)+\aw(H)$.
        \end{enumerate}
    \end{lemma}
    
    \begin{lemma}[Atomic Weight of Nimbers]\label{lem: aw of nimbers}
        The atomic weight of a nimber is 0.
    \end{lemma}
    \begin{proof}
          By Definition~\ref{def: atomic weight}, it is sufficient to prove that $\cgfarstareq{*n}{0\cdot \cgup}$, which is simply the same as $\cgfarstareq{*n}{0}$. According to Theorem~\ref{thm: check farstareq}, this is equivalent to verify the inequality:
        \[
        \cgdown\;\cgfarstar < *n < \cgdown\;\cgfarstar.
        \]  We will prove the left inequality, $\cgdown\;\cgfarstar<*n$; the proof of the right one is analogous. 
        
        To prove $\cgdown\;\cgfarstar<*n$, we must show $\cgdown\;\cgfarstar-*n\in \rp$. Here, $\cgfarstar+*n\;(=\cgfarstar-*n)$ is equivalent to $\cgfarstar$ as both are nimbers and $\cgfarstar$ is a large nimber. Therefore the game becomes $\cgdown + \cgfarstar$, which is equal to $\{\cgfarstar\mid 0\}$, by lemma~\ref{lem: n-cgup-cgfarstar}. If Right starts, he wins by moving to 0. If Left starts, she moves to $\cgfarstar$ from where Right can move to 0 and win.
    \end{proof}

    While Theorem~\ref{thm: check farstareq} and Definition~\ref{def: atomic weight} together provide a way to verify whether a given integer $w$ is the atomic weight of a game $G$, they do not offer a method to compute the atomic weight directly. The following theorem from \cite{LiP}, provides a way to determine the atomic weight of an all-small game. Note that it also establishes that each all-small games is atomic, i.e. that they all have well defined atomic weight. There are infinitesimal games that do not have atomic weight, for example.

    \begin{theorem}[{\cite{LiP}*{Theorem~9.39}}]\label{thm: compute aw}
        Let $G$ be all-small. If $G=\{G^L\mid G^R\}$ and $G^L$ and $G^R$ have atomic weights, then the atomic weight of $G$ is given by $\left\{ \aw(G^L)-2\mid\aw(G^R)+2\right\}$ unless this is an integer. In that case, let $x$ be the least integer such that $\aw(G^L)-2\cglfuz x$, and $y$ the greatest integer such that $y\cglfuz \aw(G^R)+2$.
        \begin{itemize}
            \item If $G\cgfuzzy \cgfarstar$, $\aw(G)=0$;
            \item If $G>\cgfarstar$, $\aw(G)=y$;
            \item If $G<\cgfarstar$, $\aw(G)=x$.
        \end{itemize}
    \end{theorem}

The next Theorem helps to determine the outcome of a game if its atomic weight in known. The second rule is commonly known as the ``two-ahead-rule''.

\begin{theorem}[\cite{S2013}*{Theorem~7.13}]\label{thm: outcome using aw}
    Let $G$ be atomic. Then
    \begin{enumerate}
        \item if $\aw(G)=1$, then $G\cggfuz 0$;
        \item if $\aw(G)\ge 2$, then $G>0$.
    \end{enumerate}
\end{theorem}
This concludes the CGT basics needed to read and understand the paper.



\bibliographystyle{plain}

\end{document}